\documentclass[10pt]{article}   
\usepackage{amsfonts,amsmath,latexsym}
\usepackage[T1]{fontenc}
\usepackage{dsfont}
\usepackage{hyperref}
\usepackage[dvips]{graphicx}
\usepackage{epsfig}
\usepackage{enumerate}
\usepackage{epsf}   
\usepackage{amsthm}
\usepackage{color}
\usepackage{amssymb}

\usepackage{caption}
\usepackage{subfigure}

\usepackage{fancyhdr} 

\usepackage{palatino}

\newtheorem{thm}{Theorem}[section]
\newtheorem{prop}[thm]{Proposition}
\newtheorem{lem}[thm]{Lemma}
\newtheorem{corro}[thm]{Corollary}
\newtheorem{defi}[thm]{Definition}
\newtheorem{rem}[thm]{Remark}

\newtheorem{ass}{Assumption}

\def\R{\mathbb R}
\def\N{\mathbb N}

\def\E{\mathbb E}
\def\P{\mathbb P}

\def\shb{{\cal B}}
\def\shc{{\cal C}}

\def\shf{{\cal F}}

\def\shm{{\cal M}}

\def\shp{{\cal P}}
\def\shs{{\cal S}}

\author{
{\sc Anthony LE CAVIL}
\thanks{ENSTA-ParisTech, Universit\'e Paris-Saclay.
Unit\'e de Math\'ematiques Appliqu\'ees (UMA).
 E-mail:{ \tt anthony.lecavil@ensta-paristech.fr}} 
{\sc,}\ {\sc Nadia OUDJANE}
\thanks{EDF R\&D,   and FiME (Laboratoire de Finance des March\'es de l'Energie
(Dauphine, CREST,  EDF R\&D) www.fime-lab.org). 
E-mail:{\tt  
nadia.oudjane@edf.fr}}
\ {\sc and}\ {\sc Francesco RUSSO} 
\thanks{ENSTA-ParisTech, Universit\'e Paris-Saclay.
Unit\'e de Math\'ematiques Appliqu\'ees (UMA). 
 E-mail:{\tt  francesco.russo@ensta-paristech.fr}.
  The financial support of this author was partially provided
  by the DFG through the CRC ''Taming uncertainty and profiting from 
 randomness and low regularity in analysis, stochastics and their application''.
 }}

\date{October 3rd 2018}

\title{Forward Feynman-Kac type representation for  semilinear
nonconservative Partial Differential Equations}

\newcommand{\MBFigure}[6]{
$\left. \right.$ \\
\refstepcounter{figure}
\addcontentsline{lof}{figure}{\numberline{\thefigure}{\ignorespaces #5}}
\begin{center}
\begin{minipage}{#1cm}
\centerline{\includegraphics[width=#2cm,angle=#3]{#4}}
\begin{center}
\upshape{F\textsc{ig} \normal
\end{center}
size{\thefigure}. $-$} #5
\end{center}
\label{#6}
\end{minipage}
\end{center}
$\left. \right.$ \\}

\begin{document}
\maketitle
 \begin{abstract} We propose a nonlinear forward  Feynman-Kac type equation,
 which represents the solution of
a nonconservative semilinear parabolic Partial Differential Equations (PDE).
 We show in particular  existence and uniqueness.
The solution  of that type of equation  can be approached 
via a weighted  particle  system.
 % For which we show the
  % convergence of the empirical joint distribution.

 \end{abstract}
\medskip\noindent {\bf Key words and phrases:}  
Semilinear Partial Differential Equations; Nonlinear Feynman-Kac type functional; Particle systems; Probabilistic representation of
 PDEs.

\medskip\noindent  {\bf 2010  AMS-classification}: 60H10; 60H30; 60J60; 65C05; 65C35;
% 68U20; 
35K58.
 
 % % % % % % % % % % % % % % % % % % % % % % % % % % % % % % %
 
\section{Introduction}

This paper situates in the framework of  forward probabilistic representations
%\eqref{eq:NLSDE2Intro} of \eqref{pdeIntro}   to the case
%when the latter is replaced by
of  nonlinear PDEs of the form
\begin{equation}
\label{epdeIntro0}
\left \{
\begin{array}{l}
\partial_t u = \frac{1}{2} \displaystyle{\sum_{i,j=1}^d} \partial_{ij}^2 \left( (\Phi \Phi^t)_{i,j}(t,x,u) u \right) - div \left( g(t,x,u) u \right) +\Lambda(t,x,u,\nabla u) u\ , \quad \textrm{for any}\  t\in [0,T]\ ,\\
u(0,dx) = {\bf u_0}(dx),
\end{array}
\right .
\end{equation}
where ${\bf u_0}$ is a Borel probability measure, 
$\Phi: [0,T] \times \R^d \times \R \rightarrow  M_{d,p}(\R^d)$,
$g: [0,T] \times \R^d \times \R \rightarrow  \R^d$,
$\Lambda: [0,T] \times \R^d \times \R \times \R^d \rightarrow \R$.
$u:]0,T] \times \R^d \rightarrow \R$ will be the unknown function.

Coming back to~\eqref{epdeIntro0}, allowing $\Lambda\neq 0$ encompasses the case of Burgers-Huxley or Burgers-Fisher equations which are of great importance to represent nonlinear phenomena in various fields such as biology~\cite{aronsonb,murray}, physiology~\cite{keener} and physics~\cite{wang}. 
These equations have the particular interest to describe the interaction between the reaction mechanisms, convection effect, and diffusion transport. 
However our aim is also to consider (via time reversal) 
PDEs coming from stochastic control as non-linear HJB equations.

For \eqref{epdeIntro0}, we propose the  forward probabilistic representation
\begin{equation}
\label{ENLSDELambdaIntro}
\left \{
\begin{array}{l}
Y_t = Y_0 + \int_0^t \Phi(s,Y_s,u(s,Y_s)) dW_s + \int_0^t g(s,Y_s,u(s,Y_s)) ds\ ,\quad\textrm{with}\quad Y_0\sim  {\bf u_0} \ , \\
u(t,\cdot):=\frac{d\nu_t}{dx} \quad
% \textrm{such that for any bounded continuous test function}\ 
% \varphi \in \mathcal{C}_b(\R^d, \R)\\ 
\nu_t(\varphi):=\E\left[ \varphi(Y_t) \, \exp \left \{\int_0^t\Lambda \big (s,Y_s,u(s,Y_s), \nabla u(s,Y_s)\big )ds\right \} \right]\ ,  \varphi \in \mathcal{C}_b(\R^d, \R),  t\in [0,T],
\end{array}
\right.
\end{equation}
%the second equation means that $u(t,x) dx$ converges weakly to $u_0(dx)$ when $t \rightarrow 0$.
where $W$ is a $p$-dimensional Brownian motion.
\eqref{ENLSDELambdaIntro} is a nonlinear stochastic differential equation
 (NLSDE) in the spirit of McKean, see e.g. \cite{Mckean}.
The justification of the proposed probabilistic representation relies on the fact whenever 
a solution $(Y,u)$ of \eqref{ENLSDELambdaIntro} exists then $u$ is a weak (in the sense of distributions)
of \eqref{epdeIntro0}; this follows in elementary way through an application of It\^o formula.

The underlying idea of our  approach consists in extending, to 
fairly general non-conservative PDEs, the probabilistic representation of 
nonlinear Fokker-Planck equations which appears when $\Lambda = 0$.
 An interesting aspect of this strategy is that it is potentially able to represent an extended class of second order nonlinear 
%(even possibly  path-dependent) 
PDEs.
%  allowing the function $\Lambda$ to non-linearly depend 
% not only on $u$ but also on its space derivatives.
% up to the second order.
 %The main contribution of this thesis is precisely the presence of the nonlinear term $\Lambda$  in \eqref{epdeIntro0}, which implies the forward representation \eqref{ENLSDELambdaIntro}.  \\

When $\Lambda = 0$, several authors have  studied NLSDEs
%in the spirit of McKean
of the form 
\eqref{ENLSDELambdaIntro}. 
Significant contributions are due to  \cite{sznitman},  \cite{MeleaRoel}, \cite{MeleardVlasov}, in the case where
the non linearity with respect to $u$ are mollified in the diffusion 
and drift coefficients.
%In the sequel, several authors have considered more singular dependence than in \eqref{eq:NLSDEIntro}; for instance 
In \cite{JourMeleard}, the authors focused on the case when 
the coefficients depend pointwisely on $u$. 
%\eqref{ENLSDELambdaIntro} with $\Lambda = 0$.
%
% \begin{equation}
% \left \{
% \begin{array}{l}
% \label{eq:NLSDE2Intro}
% Y_t = Y_0 + \int_0^t \Phi(s,Y_s,u(s,Y_s))dW_s + \int_0^t g(s,Y_s,u(s,Y_s)) ds \\
% u(t,\cdot) \textrm{ is the probability density of the law } Y_t \ ,
% \end{array}
% \right .
% \end{equation}
% where $\Phi : [0,T] \times \R^d \times \R \rightarrow \R^{d \times p}$, $g : [0,T] \times \R^d \times \R \rightarrow \R^d$ are smooth functions
% and $W$ is a $p$-dimensional Brownian motion.
The authors have proved strong existence and pathwise uniqueness of \eqref{ENLSDELambdaIntro}, when $\Phi$ and $g$ are smooth and Lipschitz
and $\Phi$ is non-degenerate.
% and they have succeeded in proving that $Z^{\varepsilon}$ converges to $Y$, solution of \eqref{eq:NLSDE2Intro}, in $L^2$-norm (see Proposition 2.5 in \cite{JourMeleard}), which makes rigorous previous formal considerations. \\
Other authors have more particularly studied an NLSDE of the  form
\begin{equation} 
\label{eq:NLSDE4}
\left \{
\begin{array}{l}
Y_t = Y_0 + \int_0^t \Phi(u(s,Y_s)) dW_s , \; Y_0 \sim {\bf u_0}, \\
u(t,x)dx \textrm{ is the law density of } Y_t, t > 0 \\
u(0,\cdot) = u_0 \ ,
\end{array}
\right .
\end{equation}
for which the particular case $d = 1$, $\Phi(u) = u^k$ for $k \geq 1$ was
 developed in \cite{Ben_Vallois}.
%
%ECREMER CE QUI SUIT EN METTANT JUSTE DES REFERENCES
%
 When $\Phi : \R \rightarrow \R$ is only assumed to be bounded, measurable and monotone, existence/uniqueness results are still available in \cite{BRR1,BRR2}. A partial extension to the multidimensional case ($d \geq 2$) is exposed in \cite{BCR3}.

 %All PDEs considered above, of the the type \eqref{pdeIntro}, are non-linear generalizations of the Fokker-Planck equation. 
The  solutions of \eqref{epdeIntro0}, when $\Lambda = 0$, 
  are probability
measures dynamics which often describe the macroscopic distribution law of a microscopic particle which behaves in a diffusive way. 
For that reason, those time evolution PDEs are conservative in the sense that
their solutions $u(t,\cdot)$ verify the property $\int_{\R^d} u(t, x)dx$ to be constant in $t$ (equal to $1$, which is the
mass of a probability measure). 
An interesting feature of this type of representation
is that the law of the solution $Y$ of the NLSDE 
can be characterized as the limiting empirical distribution of a large number 
of interacting particles.
This is a consequence of the so called {\it propagation of chaos}
phenomenon, already observed in the literature for the case of mollified 
dependence, see e.g. 
%Kac, McKean and Sznitman, M\'el\'eard and  Roelly,
\cite{kac, Mckean, sznitman, MeleaRoel}
and \cite{JourMeleard} for the case of pointwise dependence.
% The particle system constitutes a microscopic model
% related to the macroscopic dynamics provided by the PDE 
% \eqref{epdeIntro0} with $\Lambda = 0$.
\cite{bossytalay1} has  contributed to develop stochastic particle methods
 in the spirit of McKean
 to approach a PDE related to Burgers equation
 providing first the rate of convergence.
 Comparison with classical
 numerical analysis techniques was provided by \cite{piperno}.

In this paper we will concentrate on the novelty constituted 
by the introduction of $\Lambda$ depending on $u$ and $\nabla u$.
For this step
%consider PDEs of the type \eqref{epdeIntro0}  for which
 $\Phi$, $g$ will  not depend on $u$.
% On the other hand, in Chapter \ref{chap1} and Chapter \ref{chap2}, we will concentrate on the case where $\Phi$, $g$ will be nonlinear w.r.t. $u$, and $\Lambda$ will depend on $u$ as well. \\
%More precisely, in Chapter \ref{chap3},
 In this context we will focus on semilinear PDEs of the  form
\begin{equation}
\label{eq:PDE}
\left \{ 
\begin{array}{l}
\partial_t u = L_t^{\ast} u + u \Lambda(t,x,u,\nabla u) \\
u(0,\cdot) = {\bf u_0} \ ,
\end{array}
\right .
\end{equation}
with $L^{\ast}$ a partial differential operator of the type
\begin{equation} \label{eq:LstarIntro}
(L^\ast_t \varphi)(x) = \frac{1}{2} \sum_{i,j=1}^d \partial_{ij}^2  ((\Phi \Phi^t)_{i,j}(t,x)
 \varphi)(x) - \sum_{i=1}^d  \partial_{i} (g_i(t,x)  \varphi)(x), \quad \textrm{ for } \varphi \in \shc_0^{\infty}(\R^d).
\end{equation} 
For this case, \eqref{ENLSDELambdaIntro} becomes 
\begin{equation}
\label{ENLSDELambdaIntroBis}
\left \{
\begin{array}{l}
Y_t = Y_0 + \int_0^t \Phi(s,Y_s) dW_s + \int_0^t g(s,Y_s) ds\ ,\quad\textrm{with}\quad Y_0\sim  {\bf u_0} \ , \\
\int_{\R^d} u(t,x) \varphi(x) dx =\E\left[ \varphi(Y_t) \, \exp \left \{\int_0^t\Lambda \big (s,Y_s,u(s,Y_s), \nabla u(s,Y_s)\big )ds\right \} \right]\ ,  \varphi \in \mathcal{C}_b(\R^d, \R),  t\in [0,T].
\end{array}
\right.
\end{equation}

 If $\Lambda = 0$
 \eqref{eq:PDE} is the classical Fokker-Planck equation.
%In this case, the coefficients $\Phi$ and $g$ do not depend on $u$. However, the nonlinear term $\Lambda$ will concentrate all nonlinearities w.r.t. $u$ and $\nabla u$, which constitutes the main novelty in this chapter.
An alternative approach for representing this type of PDE is given by 
forward-backward stochastic differential equations.
 Those  were initially developed
  in~\cite{pardoux}, see
 also \cite{pardouxgeilo} for a survey and \cite{rascanu} for a recent
 monograph on the subject.  
However, the extension of those equations to  fully nonlinear PDEs still requires complex developments and is 
 the subject of active research, see for instance~\cite{cheridito}. 
Branching diffusion processes 
 provide another  probabilistic representation of semilinear PDEs,
see e.g. \cite{labordere,LabordereTouziTan,HOTTW}.
% involving a specific form of non-linearity on the zero order term. 
%his type of approach has been extended 
% in~\cite{labordere,LabordereTouziTan} to a
%  more general class of non-linearities on the zero order term, with the so-called \textit{marked branching process}. 
% More recently, an extension to a class of semilinear PDE has been proposed in~\cite{HOTTW}. 
Here again, extensions to second order nonlinear PDEs still constitutes a difficult issue. 

As suggested, our method potentially allows to reach a certain significant class
of PDEs with second-order non-linearity, if we allow the diffusion
coefficient to also depend on $u$. 
The general framework where $g$ and $\Phi$
also depend non linearly on $u$ while $\Lambda$ depends on $u$ and $\nabla u$ has been partially investigated in~\cite{LOR1}, where the dependence
of the coefficients with respect to $u$  is mollified and $\Lambda$ does not depend on $\nabla u$. An associated interacting particle system converging to the solution of a regularized version of the nonlinear PDE has been proposed in~\cite{LOR2}, providing encouraging numerical performances. 
The originality of the present paper is to consider a pointwise dependence of $\Lambda$ on both $u$ and $\nabla u$. The pointwise dependence on $\nabla u$ constitutes the major technical difficulty. 
For this we introduce a new approach based on the technique of 
 {\it mild solutions} making use of the semigroupe associated
with $L_t$. For this reason in this paper we concentrate on 
non-linearities only in $\Lambda$ leaving extensions 
in the forthcoming paper \cite{LieberOR}
where we authorize
the coefficient $b$ to depend on $u$
 
%This extension is part of a current research project.

%One of the main advantage of this approach compared to FBSDEs is that it does not involve any regression computation
%to calculate conditional expectations.
 
More specifically, we propose to associate~\eqref{eq:PDE} with a forward probabilistic representation given by a couple $(Y,u)$ solution of \eqref{ENLSDELambdaIntro} 
where $\Phi$ and $g$ are the functions  intervening in \eqref{eq:LstarIntro}.
% \begin{equation}
% \label{eq:NLSDEChap3Intro}
% \left \{ 
% \begin{array}{l}
% Y_t = Y_0 + \int_0^t \Phi(s,Y_s) dW_s + \int_0^t g(s,Y_s) ds, \quad Y_0 \sim u_0 \\
% \int_{\R^d} \varphi(x) u(t,x)dx = \E\Big[ \varphi(Y_t) \exp\Big(\int_0^t \Lambda(s,Y_s,u(s,Y_s), \nabla u(s,Y_s)) \Big) \Big], \quad \textrm{for } t\in (0,T]\,,\  \varphi \in \shc_b(\R^d) \ .
% \end{array}
% \right .
% \end{equation}
In this case, the second  line equation of \eqref{ENLSDELambdaIntro}
 will be called {\bfseries Feynman-Kac equation} and a solution
 $u : [0,T] \times \R^d \rightarrow \R$ will be called {\bfseries{Feynman-Kac type representation}} of \eqref{eq:PDE}. When $\Lambda$ vanishes, the functions $(u(t, \cdot), t > 0)$ are indeed the marginal law densities of the process $(Y_t, t > 0)$ and \eqref{eq:PDE} coincides with the classical Fokker-Planck PDE.
% associated to 
%\eqref{eq:NLSDEChap3Intro}.
%\eqref{ENLSDE}.
 When $\Lambda \neq 0$, the proof of well-posedness of the Feynman-Kac equation
is not obvious and
 it is one of the contributions of 
% Actually, it is the first main challenging contribution developped in the 
the paper. The strategy used relies on two steps. Under a Lipschitz 
condition   on    $\Lambda$ 
%w.r.t.   the space variables,
 in Theorem \ref{thm:FKForm}, we first prove
 that a function $u : [0,T] \times \R^d \rightarrow \R$ 
 (belonging to $L^1([0,T],W^{1,1}(\R^d))$)
is a solution of the Feynman-Kac 
equation \eqref{ENLSDELambdaIntro}
 %type representation of \eqref{eq:PDEChap3}
 if and only if it is a mild solution of \eqref{eq:PDE}.
The latter concept is introduced in  item 2. of Definition \ref{def:SolPDE}.
Then, under Lipschitz
 type conditions on $\Phi$ and $g$, Theorem \ref{prop:UniTrFun} establishes 
 the existence and uniqueness of a mild solution of \eqref{eq:PDE}.
%in the mild sense (see item 2. of Definition \ref{def:SolPDE}) in $L^1([0,T],W^{1,1}(\R^d) \cap L^{\infty}([0,T],\R^d)$. 
As a second contribution, we propose and analyze a corresponding
 particle system.
 %based  numerical scheme.
 This  relies on two approximation steps: 
a regularization procedure based on a kernel convolution and 
the law of large numbers.
% a space
%  discretization based on Monte Carlo simulations of the diffusion
% $Y$ and a time discretization. 
The convergence of the particle system is stated in
 Theorem \ref{thm:ThmCvg}.

The theoretical analysis of the performance of the time-discretized algorithm 
related to the present paper has been performed in 
Theorem 3.4 in
\cite{LOR4}.
In that  paper we test the  algorithm
 with respect to the Burgers and KPZ equations
for which there are explicit solutions.

 \section{Preliminaries} %Notations and assumptions
 \label{S2P3}

\subsection{Notations}
\label{S21}

\setcounter{equation}{0}

%%%%%%%%%%%%%%%%%%%%%%%%%%%%%%%%%%%
Let $d \in \N^{\star}$. Let us consider $\shc^d:=\mathcal{C}([0,T],\R^d)$ metricized by the supremum norm $\Vert \cdot \Vert_{\infty}$, equipped with its Borel $\sigma-$ field $\mathcal{B}(\shc^d)$
% = \sigma(X_t,t \geq 0)$ (and $\mathcal{B}_t(\shc^d) := \sigma(X_u,0 \leq u \leq t)$ the canonical filtration)
and endowed with the topology of uniform convergence.  
%$X$ will be the canonical process on $\shc^d$.\\
If $(E,d_E)$ is a Polish space,
%Given $r \ge 0$, $\mathcal{P}_r(E)$ is the set of Borel probability measures on $E$ admitting a moment of order $r$. For $r=0$,
 $\mathcal{P}(E)$
% := \mathcal{P}_0(E)$
 denotes the Polish space (with respect to the weak convergence topology) of Borel probability measures on $E$ naturally equipped with its Borel $\sigma$-field $\mathcal{B}(\shp(E))$. 
 The reader can consult
 Proposition 7.20 and Proposition 7.23, 
Section 7.4 Chapter 7 in \cite{BertShre} for more exhaustive information.
When $d=1$, we often omit it and we simply note 
 $\mathcal{C} :=  \mathcal{C}^1$. $\shc_b(E)$ denotes the space of bounded, continuous real-valued functions on $E$.
%still equipped with the  norm $\Vert \cdot \Vert_{\infty}$. \\
%Given $N \in \N^{\star}$, $l \in \shc^d$, $l^1, \cdots, l^N \in \shc^d$, a significant role in this paper will be played by the Borel measures on $\shc ^d$ given by $\delta_l$ and $\displaystyle{ \frac{1}{N} \sum_{j=1}^N \delta_{l^j}}$.\\
 
In this paper, $\R^d$ is equipped with the Euclidean scalar product $\cdot $ and $\vert x\vert $  stands for the induced norm for $x \in \R^d$. 
The gradient operator for
%(w.r.t. $x \in \R^d$) for differentiable 
functions defined on $\R^d$ is 
denoted by  $\nabla$.
If a function $u$ depends on a variable $x \in \R^d$ and other variables,
we still denote by $\nabla u$ the gradient of $u$ with respect
to $x$, if there is no ambiguity.
% If there is no confusion, $\nabla$ will simply denote the gradient on $\R^d$.
%Given two reals $a$ and $b$, we set $a \wedge b := \min(a,b)$ and $a \vee b := \max(a,b)$.
$M_{d,p}(\R)$ denotes the space of $\R^{d \times p}$ real matrices equipped with the Frobenius norm (also denoted $\vert \cdot \vert$), i.e. the one induced by the scalar product $(A,B) \in M_{d,p}(\R^d) \times M_{d,p}(\R)  \mapsto Tr(A^tB)$ where $A^t$ stands for the transpose matrix of $A$ and $Tr$ is the trace operator. $\shs_d$ is the set of symmetric, non-negative definite $d \times d$ real matrices and $\shs_d^+$ the set of strictly positive definite matrices of $\shs_d$. \\
$\mathcal{M}_f(\R^d)$
%(resp. $\shm_f^+(\R^d)$)
is the space of finite
%(resp. finite, positive)
Borel measures on $\R^d$. When it is endowed with the weak convergence topology, $\shb(\shm_f(\R^d))$
%(resp. $\shb(\shm_{f}^+(\R^d))$)
stands for its Borel $\sigma$-field. It is well-known that $(\shm_f(\R^d),\Vert \cdot \Vert_{TV})$ is a Banach space, where $\Vert \cdot \Vert_{TV}$ denotes the total variation norm. 
$\mathcal{S}(\R^d)$ is the space of Schwartz fast decreasing test functions and $\mathcal{S}'(\R^d)$ is its dual. $\mathcal{C}_b(\R^d)$ is the space of bounded, continuous functions on 
$\R^d$ and $\mathcal{C}^{\infty}_0(\R^d)$ the space of smooth functions with compact support. For any positive integers $p,k \in \N$, $C^{k,p}_b := C^{k,p}_b([0,T] \times \R^d, \R)$ denotes the set of continuously differentiable bounded functions $[0,T] \times \R^d \rightarrow \R$ with uniformly bounded derivatives with respect to the time variable $t$ (resp. with respect to space variable $x$) up to order $k$ (resp. up to order $p$).  In particular, for $k = p = 0$, $C_b^{0,0}$ coincides with the space of bounded, continuous functions also denoted by $\shc_b$.
$\mathcal{C}^{\infty}_b(\R^d)$ is the space of bounded and smooth functions. $\mathcal{C}_0(\R^d)$ denotes the space of continuous functions with compact support in $\R^d$. For $r \in \N$, $W^{r,p}(\R^d)$ is the Sobolev space of order $r$ in $(L^p(\R^d),||\cdot||_{p})$, with $1 \leq p \leq \infty$. \\
%PEUT-ETRE PAS NECESSAIRE D'INTRODUIRE CETTE NOTATION $W_{loc}^{1,1}(\R^d)$ 
%denotes the space of functions $f : \R^d \rightarrow \R$ such that $f$ and $\nabla f$ (existing in the weak sense) belong to $L^1_{loc}(\R^d)$. \\
For convenience  we introduce the following notation. 
 \begin{itemize}   
 	\item $V\,:[0,T] \times {\mathcal C}^d \times \shc \times \shc^d $ is defined for any functions $x \in \shc^d$, $y \in \shc$ and $ z \in \shc^d$, by 
\begin{equation}
\label{eq:VP3}
V_t(x,y,z):=\exp \left ( \int_0^t \Lambda(s,x_s,y_s,z_s) ds\right )\quad \textrm{for any} \ t\in [0,T]\ .
\end{equation} 
The finite increments theorem gives, for all $(a,b) \in \R^2$, we have
\begin{eqnarray}
\label{eq:Vmajor}
\exp(a) - \exp(b) = (b-a) \int_0^1 \exp(\alpha a + (1-\alpha)b) d \alpha.
\end{eqnarray}
Therefore, if $\Lambda$ is supposed to be bounded and Lipschitz w.r.t.
 to its space variables $(x,y,z)$, uniformly w.r.t. $t$,
 we observe that \eqref{eq:Vmajor} implies 
 that, for all $t \in [0,T]$, $x,x' \in \shc^d$, $y,y' \in \shc$, $z,z' \in \shc^d$,
\begin{eqnarray}
\label{eq:LipV}
\vert V_t(x,y,z) - V_t(x',y',z') \vert \leq L_{\Lambda} e^{tM_{\Lambda}} \int_0^t \big( \vert x_s - x'_s \vert + \vert y_s - y'_s \vert + \vert z_s - z'_s \vert \big) ds \ ,
\end{eqnarray}
$M_{\Lambda}$ (resp. $L_{\Lambda}$) denoting an upper bound of $\vert \Lambda \vert$ (resp. the Lipschitz constant of $\Lambda$), see also Assumption \ref{ass:mainP3}. \\
\item  For every $\varepsilon$, $K_{\varepsilon}: \R^d  \rightarrow \R$ denotes a mollifier such that
\begin{equation}
\label{eq:Keps}
K_\varepsilon(x):= \frac{1}{\varepsilon^d}K\left(\frac{x}{\varepsilon}\right), \forall x \in \R^d\ ,
%\quad \textrm{and}\quad K_{\varepsilon} \xrightarrow[\varepsilon \rightarrow 0]{} \delta_0\ ,
\end{equation}
where $K$ is a probability density on $\R^d$ such that 
%\begin{enumerate}
%\item \begin{equation} \label{EKappa1}
  %K \ge 0 \ \int_{\R^d} K(x)\, dx = 1\ ,
%\end{equation}
%\item
\begin{equation}
\label{eq:HypK}
 K \in W^{1,1}(\R^d) \cap W^{1,\infty}(\R^d) \ . 
\end{equation}
In the sequel, $K$ may be asked to additionally verify the following conditions.
\begin{equation} \label{EKappa}
%   \int_{\R^d} x\, K(x)\, dx = 0 \quad \textrm{and}\quad 
 \kappa := \frac{1}{2}\int_{\R^d} \vert x \vert\, K(x)\, dx < \infty\ .
\end{equation}
%\item 
%\item 
\begin{equation}
\label{eq:Kcond}
 \int_{\R^d} \vert x \vert^{d+1} \; K(x) dx<\infty\ ,\quad \textrm{and}\quad \int_{\R^d}\vert x\vert^{d+1} \; \vert \nabla K (x)\vert dx<\infty\ .
\end{equation}
%\end{enumerate}  
\end{itemize}

%%% AJUSTER DANS LES DIFFERENTES PROPOSITIONS

In the whole paper, $(\Omega,\shf,(\shf_t)_{t \geq 0}, \P)$ will denote 
a filtered probability space and $W$ an $\R^p$-valued $(\shf_t)$-Brownian motion.

\subsection{Mild and Weak solutions}
\label{S22}
We first introduce the following assumption.
\begin{ass}
\label{ass:main0}
\begin{enumerate}
\item $\Phi$ and $g$ are functions defined on $[0,T] \times \R^d$ taking values in $M_{d,p}(\R^d)$ and $\R^d$. \\
There exist $L_{\Phi}, L_g > 0$ such that for any $t \in [0,T]$, $(x,x') \in \R^d \times \R^d$,
\begin{eqnarray*}
\vert \Phi(t,x) - \Phi(t,x') \vert &\leq& L_{\Phi} \vert x-x' \vert \ , \\
%\end{equation}
%and
%\begin{equation}
\vert g(t,x) - g(t,x') \vert &\leq& L_{g} \vert x-x' \vert \ .
\end{eqnarray*}
\item The functions $s \in [0,T] \mapsto \vert \Phi(s,0) \vert$ and $s \in [0,T] \mapsto \vert g(s,0) \vert$ are bounded.
\end{enumerate}
\end{ass} 
\noindent Given any    $\sigma(W_r, r \le s)$-measurable r.v. $Y_s$, classical theorems for SDE with Lipschitz coefficients imply strong existence and pathwise uniqueness for the  SDE 
\begin{equation}
\label{eq:SDEbis}
%\left\{ 
%\begin{array}{l}
dY_t = \Phi(t,Y_t) dW_t + g(t,Y_t)dt, t \in [s,T].
%\\ Y_0 \sim u_0 \; ,
%\end{array}
%\right .
\end{equation}
%We denote by $\psi_s(t)$ previous solution
%and $Y_t = \psi_0(t), t\ge 0$.
%\medskip
%% J'AI CHANGE CE QU'IL Y AVAIT APRES
% and the initial condition ${\bf u_0}$ of \eqref{eq:PDE} has to be understood in the sense that 
% $$
% \lim_{t \rightarrow 0} \int_{\R^d} \varphi(x)u_t(dx) = \int_{\R^d} \varphi(x){\bf u_0}(dx) \; , \textrm{ for all } \varphi \in \shc_0^{\infty}(\R^d) \ ,
% $$
% since, a priori, it can be irregular and not necessarily a function.
%%%  ALTERNATIVE. NE PAS PARLER DE SOLUTIONS FONDAMENTALES
%%MAIS SEULEMENT DE FOKKER-PLANCK \cite{eq:FokketPlanck}
%We observe that
Section 2.2, Chapter 2 in \cite{stroock} and Section 2.1, Chapter 1 of 
\cite{friedmanEDS1} one introduces the notion of 
{\bf Markov transition function}.
Under Assumption \ref{ass:main0}, by Theorem 3.1 chap. 5 of \cite{friedmanEDS1},
 it is well-known, that there exists a \textit{good} family of Markov 
transition functions $P(s,x_0,t, \cdot)$
such that, for every $0 \le s \le t \le T$ and  Borel subset $A$ of
$\R^d$ we have
\begin{equation}  \label{eq:lawY2}
 P\{Y_t \in A \vert \sigma(W_r, r \le s) \} =
 P\{Y_t \in A \vert Y_s \} =  P(s,Y_s,t,A).   
%P\{\psi_s(t) \in A \},  \nonumber
 \end{equation}
%using the notations of \eqref{eq:SDE}.
Let $s= 0, Y_0 \sim {\bf u_0}$.
From now on $Y$ will be the unique strong solution of
the SDE 
\begin{equation}
\label{eq:SDE}
%\left\{ 
%\begin{array}{l}
dY_t = \Phi(t,Y_t) dW_t + g(t,Y_t)dt, t \in [0,T].
%\\ Y_0 \sim u_0 \; ,
%\end{array}
%\right .
\end{equation}
For $t \in [0,T]$, the marginal law of $Y_t$ is given for all $\varphi \in \shc_b(\R^d)$ by
%\begin{enumerate}
%\item
\begin{equation}
\label{eq:lawY1}
\E[\varphi(Y_t)]  =  \int_{\R^d} {\bf u_0}(dx_0) \int_{\R^d} \varphi(x) P(0,x_0,t,dx) \ .
\end{equation} 
% \item For all $\varphi \in \shc_b(\R^d)$ and $0 \leq s < t \leq T$,
% \begin{eqnarray}
% \label{eq:lawY2}
% \E[\varphi(Y_t) \vert Y_s] = \int_{\R^d} \varphi(x) P(s,Y_s,t,dx).
% \end{eqnarray}
% \end{enumerate}
 \\
In the whole paper we will write $a = \Phi \Phi^t$; in particular $a : [0,T] \times \R^d \longrightarrow \shs_d$.
Through some definitions, we make here precise in which sense we will consider solutions of the PDE \eqref{eq:PDE}. We are interested in different concepts of solutions $u : [0,T] \times \R^d \longrightarrow \R$ of that semilinear 
PDE where, for $t \in [0,T]$, $L_t$ is given by
% \begin{equation}
% \label{eq:PDE}
% \left \{
% \begin{array}{l}
% \partial_t u = L^{\ast}_t u + u\Lambda(t,x,u,\nabla_x u) \\
% u(0,\cdot) = u_0, \; \ ,
% \end{array}
% \right .
% \end{equation}
\begin{eqnarray}
\label{eq:generat}
(L_t \varphi)(x) = \frac{1}{2} \sum_{i,j=1}^d a_{i,j}(t,x) \partial_{ij}^2 \varphi(x) + \sum_{i=1}^d g_i(t,x) \partial_{i} \varphi(x), \; \varphi \in \shc_0^{\infty}(\R^d).
\end{eqnarray}
Its ''adjoint''  $L^{\ast}_t$ defined in  
%\eqref{eq:AdjGen2},
\eqref{eq:LstarIntro},
 verifies
\begin{eqnarray}
\label{eq:AdjGen}
\int_{\R^d} L_t \varphi(x) \psi(x) dx = \int_{\R^d} \varphi(x) L_t^{\ast}
 \psi(x) dx \; , \; (\varphi,\psi) \in \shc_0^{\infty}(\R^d), t \in [0,T].
\end{eqnarray}
Let   $\nu_0$ be a Borel probability measure on $\R^d$.
 By an easy application of It\^o formula to $Y_t$ when
$Y_s \sim \nu_0$ 
with smooth $\varphi$ with compact support, 
 one can show that
%(see Introduction and Section 2.2, Chapter 2 in \cite{stroock}),
% the Fokker-Planck equation (understood in the sense of distributions) is verified, i.e.
%for which it is well known it satisfies the Fokker-Planck equation (in the sense of distribution),
%for every finite signed measure
%   $q_0$ on $\R^d$, 
 the 
measure-valued function
\begin{equation}\label{fundamentalsolution}
\nu_s(t,dx):=\int\limits_{\mathbb{R}^d} P(s,x_0,t,dx)\nu_0(d x_0)
\end{equation}
is a  solution in the sense of distributions to the Fokker-Planck equation
\begin{align}\label{eq:FokkerPlanck}
\begin{cases} \partial_t \nu_s(t,dx) = &L_t^* \nu_s(t,dx)   \quad \forall (t,x) \in ]s,T]\times \R^d \\
\nu_s(s,\cdot)= & \nu_0,
\end{cases}
\end{align}
i.e., for all $\varphi \in C_0^{\infty},$
\begin{align}\label{eq:rather_implicit}
\int\limits_{\mathbb{R}^d} \varphi (x)\nu_s(t,dx) -\int\limits_{\mathbb{R}^d} \varphi (x)\nu_0(d x)=\int\limits_s^t \int\limits_{\mathbb{R}^d} L_r\varphi (x)
 \nu_s(r,dx) d r.  
\end{align}
In particular \eqref{eq:FokkerPlanck} 
with $\nu_0 =  \delta_{x_0}$
says that
\begin{align}\label{eq:rather_implicitP}
\int\limits_{\mathbb{R}^d} \varphi (x) P(s,x_0,t,dx) - 
 \varphi (x_0) = \int\limits_s^t \int\limits_{\mathbb{R}^d} 
L_r\varphi (x)  P(s,x_0,r,dx)   d r,  
\end{align}
which means
\begin{equation}
\label{eq:FKLin} 
\left \{
\begin{array}{l}
\partial_t P(s,x_0,t, \cdot) = L_t^{\ast} P(s,x_0,t,\cdot) \\
%\lim_{t \downarrow s} P(s,x_0,t, \cdot)
P(s,x_0,s, \cdot)
 = \delta_{x_0}, \quad 0 \leq s  \leq T, x_0 \in \R^d \ .
\end{array}
\right .
\end{equation}

% This implies that 

%This is the case for instance when $\Phi$ and $g$ are Lipschitz or bounded continuous
%and $\Phi$ is non-degenerate, i.e.
%there exists $c > 0$ such that for all $y \in \R^d$ 
%\begin{equation}  \label{def:NonDeg}
%\inf_{s\in[0,T]} \inf_{v \in \R^d \setminus \{0\}} \; \frac{\langle v,\Phi(s,y)\Phi^{t}(s,y)v 
%\rangle}{\vert v \vert^2} \geq c > 0.
%\end{equation}
%For a given
%On our given filtered probability space $(\Omega,\shf,(\shf_t)_{t \geq 0}, \P)$,
% By the classical theory of Markov processes (see e.g. Chapter 2 in \cite{stroock}), we know that the transition probability function $P$, satisfying \eqref{eq:FKLin}, defines and characterizes uniquely the law of the process $Y$, provided the law ${\bf u_0}$ of $Y_0$ is specified. 
% In particular,
% we have the following.
% \begin{enumerate}
% \item For $t \in [0,T]$, the marginal law of $Y_t$ is given for all $\varphi \in \shc_b(\R^d)$ by
% \begin{eqnarray}
% \label{eq:lawY1}
% \E[\varphi(Y_t)] & = & \int_{\R^d} {\bf u_0}(dx_0) \int_{\R^d} \varphi(x) P(0,x_0,t,dx) \ .
% \end{eqnarray} 
% \item For all $\varphi \in \shc_b(\R^d)$ and $0 \leq s < t \leq T$,
% \begin{eqnarray}
% \label{eq:lawY2}
% \E[\varphi(Y_t) \vert Y_s] = \int_{\R^d} \varphi(x) P(s,Y_s,t,dx).
% \end{eqnarray}
% \end{enumerate} 
Let $\Lambda : [0,T] \times \R^d \times \R \times \R^d \longrightarrow \R$ be bounded, Borel measurable, we recall the notions of {\bf{weak solution}} and {\bf{mild solution}} associated to \eqref{eq:PDE}.
%  If $\Lambda = 0$, \eqref{eq:PDE} is the classical Fokker-Planck equation which can be understood in the sense of distributions in the following sense: for all $\varphi \in \shc_0^{\infty}(\R^d)$, $t \in [0,T]$,
% \begin{eqnarray}
% \label{eq:SolFP}
% \int_{\R^d} u(t,x) \varphi(x) dx = \int_{\R^d} {\bf u_0}(dx) \varphi(x) + \int_0^t \int_{\R^d} u(s,x) (L_s \varphi)(x) dx ds  \ .
% \end{eqnarray}
%% FIN DEPLACER
\begin{defi}
\label{def:SolPDE}
Let $u : [0,T] \times \R^d \longrightarrow \R$ be a Borel function such that for every $t \in ]0,T]$, $u(t,\cdot)\in W^{1,1}(\R^d)$. 
\begin{enumerate}
\item %By {\bf{weak solution}} of \eqref{eq:PDE} we intend a Borel function 
$u$ will be called  {\bf{weak solution}} of \eqref{eq:PDE}
if for all $\varphi \in \shc_0^{\infty}(\R^d)$, $t \in [0,T]$,
\begin{eqnarray}
\label{eq:DefSolPDE}
\int_{\R^d} \varphi(x) u(t,x)dx - \int_{\R^d} \varphi(x) {\bf u_0}(dx) & = & \int_0^t \int_{\R^d} u(s,x) L_s\varphi(x)dx ds \nonumber \\
&& + \; \int_0^t \int_{\R^d} \varphi(x) \Lambda(s,x,u(s,x),\nabla u(s,x)) u(s,x) dx ds \ . \nonumber \\
\end{eqnarray}
\item %Let $u : [0,T] \times \R^d \longrightarrow \R$, be a Borel function such that for every $t \in [0,T]$, $u(t,\cdot)$ belongs to $W^{1,1}_{\rm loc}(\R^d)$. 
$u$ will be called  {\bf{mild solution}} of \eqref{eq:PDE} if for all $\varphi \in \shc_0^{\infty}(\R^d)$, $t \in [0,T]$,
\begin{eqnarray}
\label{eq:DefMildSol}
\int_{\R^d} \varphi(x) u(t,x)dx & = & \int_{\R^d} \varphi(x) \int_{\R^d} 
{\bf u_0}(dx_0) P(0,x_0,t,dx)  \nonumber \\
&& + \; \int_{[0,t] \times \R^d} \Big( \int_{\R^d} \varphi(x) P(s,x_0,t,dx) \Big) \Lambda(s,x_0,u(s,x_0), \nabla u(s,x_0)) u(s,x_0)dx_0 ds  \ . \nonumber \\
\end{eqnarray}
\end{enumerate}
\end{defi}

%  which can be understood in the sense of distributions in the following sense: for all $\varphi \in \shc_0^{\infty}(\R^d)$, $t \in [0,T]$,
% \begin{eqnarray}
% \label{eq:SolFP}
% \int_{\R^d} u(t,x) \varphi(x) dx = \int_{\R^d} {\bf u_0}(dx) \varphi(x) + \int_0^t \int_{\R^d} u(s,x) (L_s \varphi)(x) dx ds  \ .
% \end{eqnarray}
% %% J'AI ENLEVE

%The object of the first lemma below is to show to what extent the concept of mild solution is equivalent to the weak one. OLD
As mentioned in the introduction, a natural approach to show the link between \eqref{eq:PDE} and
\eqref{ENLSDELambdaIntroBis} consists in applying It\^o's formula to the solution $Y$ of
\eqref{eq:SDE}: if $(Y,u)$ is a solution of \eqref{ENLSDELambdaIntroBis}, then $u$ is
a weak solution of \eqref{eq:PDE}.
However, in this paper, instead of the notion of weak solution, we will make use of
 the notion of mild solution.
The link between those two notions is discussed
in the proposition below.
\begin{prop}
\label{lem:MildWeak}
We assume that $\nu = 0$ is the unique  solution in the sense of distributions
of  \eqref{eq:FokkerPlanck} with $\nu_0 = 0$.
% \begin{equation}
% \label{eq:FK}
% \left \{ 
% \begin{array}{l}
% \partial_t v = L^{\ast}_t v \\
% v(0,\cdot) = 0 \ ,
% \end{array}
% \right .
% \end{equation}
%in the sense of distributions,
% where $L_t^{\ast}$ is given by 
%\eqref{eq:AdjGen2}. \\
%\eqref{eq:LstarIntro}.
Then, $u$ is a mild solution of \eqref{eq:PDE} if and only if $u$ is a weak solution of \eqref{eq:PDE}
\end{prop}

%%% OLD
% \begin{lem}
% \label{lem:MildWeak}
% We assume that $\nu = 0$ is the unique  solution in the sense of distributions
% of  \eqref{eq:FokkerPlanck} with $\nu_0 = 0$.
% \begin{equation}
% \label{eq:FK}
% \left \{ 
% \begin{array}{l}
% \partial_t v = L^{\ast}_t v \\
% v(0,\cdot) = 0 \ ,
% \end{array}
% \right .
% \end{equation}
%in the sense of distributions,
% where $L_t^{\ast}$ is given by 
%\eqref{eq:AdjGen2}. \\
%\eqref{eq:LstarIntro}.
%Then, $u$ is a mild solution of \eqref{eq:PDE} if and only if $u$ is a weak solution of \eqref{eq:PDE}
%\end{prop}
\begin{proof}
Postponed to the Appendix, see Section \ref{PL22}.

\end{proof}

\begin{rem} \label{R23}
There exist several sets of technical assumptions (see e.g. \cite{bogachevkrylovBook,friedmanEDS1}) leading to the uniqueness assumed in Proposition \ref{lem:MildWeak} above. In particular, under items 1., 2. and 3. of Assumption \ref{ass:mainP3} stated in Section \ref{SFK} (which will constitutes our framework in the sequel), Theorem 4.7 in Chapter 4 of \cite{friedmanEDS1} ensures (classical) existence and uniqueness of the solution of \eqref{eq:FokkerPlanck}, see also Lemma \ref{lem:transfun} in the Appendix.
\end{rem}

\section{Feynman-Kac type representation}
\setcounter{equation}{0}
\label{SFK}
We suppose here the validity of Assumption \ref{ass:main0}.
Let ${\bf u_0} \in \shp(\R^d)$ and
% i.e.  a Borel probability measure.
 %We fix a filtered probability space $(\Omega,\shf,(\shf_t)_{t \geq 0},\P)$, 
 fix  a random variable $Y_0$ distributed according to ${\bf u_0}$ and consider the strong solution $Y$ of \eqref{eq:SDE}.
From now on $Y$ will be fixed. 

The aim of this section is to show how a mild solution of \eqref{eq:PDE} can be linked with a Feynman-Kac type equation, where we recall that a solution is given by a function $u : [0,T] \times \R^d \rightarrow \R$ satisfying the second line equation of
% \eqref{eq:NLSDE}. \\
\eqref{ENLSDELambdaIntro}. 

Given $\tilde{\Lambda}:
 [0,T] \times \R^d \longrightarrow \R$ a bounded, Borel measurable function, let us consider 
the measure-valued map $ \mu : [0,T] \longrightarrow \shm_f(\R^d)$ defined by
\begin{eqnarray}
\label{eq:defMu}
\int_{\R^d} \varphi(x) \mu(t,dx) = \E \Big[ \varphi(Y_t) \exp \Big( \int_0^t \tilde{\Lambda}(s,Y_s)ds \Big) \Big], \textrm{ for all } \varphi \in \shc_b(\R^d),  t \in [0,T] \ .
\end{eqnarray}
The first proposition below shows how the map $t \mapsto \mu(t,\cdot)$ can be characterized as a solution of the linear parabolic PDE
\begin{equation}
\label{eq:PDEMu0}
\left \{
\begin{array}{l}
\partial_t v = L_t^{\ast} v + \tilde{\Lambda}(t,x) v \\
v(0,\cdot) = {\bf u_0} \ .
\end{array}
\right .
\end{equation}
Before stating the corresponding proposition, we introduce the notion of {\itshape measure-mild solution}.
\begin{defi}
\label{def:MeasureMild}
Let $\mu : [0,T] \rightarrow \shm_f(\R^d)$ be 
a  measure-valued map such that
$$ \int_0^T \Vert \mu(t, \cdot)\Vert dt < \infty.$$
$\mu$  will be called {\bf{measure-mild solution}} of \eqref{eq:PDEMu0} if for all $\varphi \in \shc_0^{\infty}(\R^d)$, $t \in [0,T]$,
\begin{eqnarray}
\label{eq:MildSolLinPde}
\int_{\R^d} \varphi(x) \mu(t,dx) & = & \int_{\R^d} \varphi(x) \int_{\R^d} {\bf u_0}(dx_0) P(0,x_0,t,dx) \nonumber \\
&& + \; \int_{[0,t] \times \R^d} \Big( \int_{\R^d} P(r,x_0,t,dx) \varphi(x)  \Big) \tilde{\Lambda}(r,x_0) \mu(r,dx_0) dr .
\end{eqnarray}
\end{defi}
\begin{rem} \label{R32}
\begin{enumerate}
\item 
%Since $\mu$ is a (finite) measure valued function, 
By usual approximation arguments, it is not difficult to show that an equivalent formulation for  Definition  \ref{def:SolPDE}
 can be expressed taking $\varphi$ in $\shc_b(\R^d)$ instead of $\varphi \in \shc_0^{\infty}(\R^d)$.
\item
Although the definition of {\bf{mild solution}} (see item 2. of Definition of \ref{def:SolPDE}) and the one of {\bf{measure-mild solution}} seem to be formally close, the two concepts do not make sense in the same situations. Indeed, the notion of mild-solution makes sense for PDEs with nonlinear terms of the general form $\Lambda(t,x,u,\nabla u)$, whereas a measure-mild solution can exist only for linear PDEs.
%we have introduced the notion of mild-solution, valid for semi-linear PDE of the form \eqref{eq:PDE} and allowing nonlinear terms of the form $f(t,x,u,\nabla_x u)$. The measure-mild solution concept (as defined by \eqref{eq:PDEMu}) only makes sense for linear PDEs since it allows as Schwartz distributions as solution and not only functions. In this case, nonlinearities w.r.t. $u$, $\nabla u$ do not make sense. \\
However, in the case where a measure $\mu$ on $\R^d$, absolutely continuous w.r.t. the Lebesgue measure $dx$, is a measure-mild solution of the linear PDE \eqref{eq:PDEMu0}, its density indeed coincides with the mild solution (in the sense of item 2. of Definition \ref{def:SolPDE}) of \eqref{eq:PDEMu0}.
\end{enumerate}
\end{rem}
\begin{prop}
\label{prop:MuPDE} 
%We assume that $\Phi$ and $g$ are bounded, continuous with $\Phi$ non-degenerate in the sense of \eqref{def:NonDeg}. \\
Under Assumption \ref{ass:main0} the measure-valued map $\mu$ defined by \eqref{eq:defMu} is the unique measure-mild solution of
\begin{equation}
\label{eq:PDEMu}
\left \{
\begin{array}{l}
\partial_t v = L_t^{\ast} v + \tilde{\Lambda}(t,x) v \\
v(0,\cdot) = {\bf u_0} \ ,
\end{array}
\right .
\end{equation}
where the operator $L_t^{\ast}$ is defined by 
\eqref{eq:LstarIntro}.
%\eqref{eq:AdjGen2}. 
%with $P$ satisfying $\eqref{eq:FKLin}$.
\end{prop}
\begin{proof}
We first prove that a function $\mu$ defined by \eqref{eq:defMu} is a measure-mild solution of \eqref{eq:PDEMu}. \\
Observe that for all $t \in [0,T]$, 
\begin{eqnarray}
\exp \Big( \int_0^t \tilde{\Lambda}(r,Y_r) dr \Big) & = & 1 + \int_0^t \tilde{\Lambda}(r,Y_r)e^{\int_0^r \tilde{\Lambda}(s,Y_s) ds} dr\ .
% \; d\P \textrm{-a.s.}
\end{eqnarray}
From \eqref{eq:defMu}, it follows that for all test function $\varphi \in \shc^{\infty}_0(\R^d)$ and $t \in [0,T]$,
\begin{eqnarray}
\label{eq:DevelopMu}
\int_{\R^d} \varphi(x) \mu(t,dx) & = & \E \Big[ \varphi(Y_t) \exp \Big( \int_0^t \tilde{\Lambda}(r,Y_r) dr \Big) \Big] \nonumber \\
& = & \E[\varphi(Y_t)] + \int_0^t \E \Big[ \varphi(Y_t) \tilde{\Lambda}(r,Y_r) e^{\int_0^r \tilde{\Lambda}(s,Y_s) ds} \Big] dr \ .
\end{eqnarray}
On the one hand, by \eqref{eq:lawY1}, we have 
\begin{eqnarray}
\label{eq:512Bis}
\E[\varphi(Y_t)] = \int_{\R^d} {\bf u_0}(dx_0) \int_{\R^d} \varphi(x)  P(0,x_0,t,dx)  \ , \quad \varphi \in \shc^{\infty}_0(\R^d) \textrm{ and } \; t \in [0,T] \ .
\end{eqnarray}
On the other hand, using \eqref{eq:lawY2} yields, for $\varphi \in \shc^{\infty}_0(\R^d)$, $0 \leq r \leq t$,
\begin{eqnarray}
\E \Big[ \varphi(Y_t) \tilde{\Lambda}(r,Y_r) e^{\int_0^r \tilde \Lambda(s,Y_s) ds}  \Big] %& = & \E \Big[ \E \Big[ \varphi(Y_t) \tilde{\Lambda}(s,Y_s) e^{\int_0^r \Lambda(s,Y_s) ds} \Big \vert \sigma(Y_{\theta}, 0 \leq \theta \leq r) \Big] \Big] \nonumber \\
& = & \E \Big[ \tilde{\Lambda}(r,Y_r) e^{\int_0^r \tilde{\Lambda}(s,Y_s) ds} \E \Big[ \varphi(Y_t) \Big \vert Y_r \Big]   \Big] \nonumber \\
%\label{eq:418}
\nonumber
& = & \E \Big[ \Big( \tilde{\Lambda}(r,Y_r) \int_{\R^d} \varphi(x) P(r,Y_r,t,dx) \Big) e^{\int_0^r \tilde{\Lambda}(s,Y_s) ds} \Big]  \\
\label{eq:2ndDevMu}
& = & \int_{\R^d} \Big( \tilde{\Lambda}(r,x_0) \int_{\R^d} \varphi(x) P(r,x_0,t,dx) \Big) \mu(r,dx_0) \ ,
\end{eqnarray}
where the third equality above comes from \eqref{eq:defMu} applied to
 the bounded, measurable test function $z \mapsto \tilde{\Lambda}(r,z) \int_{\R^d} \varphi(x) P(r,z,t,dx)$. Injecting \eqref{eq:2ndDevMu} and \eqref{eq:512Bis} in the right-hand side (r.h.s.) of \eqref{eq:DevelopMu} gives for all $\varphi \in \shc^{\infty}_0(\R^d)$, $t \in [0,T]$,
\begin{eqnarray}
\label{eq:MildSolMu}
\int_{\R^d} \varphi(x) \mu(t,dx) & = & \int_{\R^d} {\bf u_0}(dx_0) \int_{\R^d} P(0,x_0,t,dx) \varphi(x)dx \nonumber \\
&& + \; \int_0^t \int_{\R^d} \mu(r,dx_0) \tilde{\Lambda}(r,x_0) \int_{\R^d} P(r,x_0,t,dx) \varphi(x) \; dr \ .
\end{eqnarray}
It remains now to prove uniqueness of the measure-mild solution of \eqref{eq:PDEMu}. 
We recall that $\shm_f(\R^d)$ denotes the vector space of finite Borel measures on $\R^d$, that is here equipped with the total variation norm $\Vert \cdot \Vert_{TV}$. We also recall that an equivalent definition of the total variation norm is given by 
\begin{eqnarray}
\label{eq:TotVar}
\Vert \mu \Vert_{TV} = \underset{\Vert \psi \Vert_{\infty} \leq 1}{\sup_{\psi \in \shc_b(\R^d)}} \Big \vert \int_{\R^d} \psi(x)\mu(dx) \Big \vert \ .
\end{eqnarray}
%We also consider the linear space $\shc([0,T],\shm(\R^d))$ equipped with the norm $\Vert v \Vert := \sup_{t \in [0,T]} \Vert v(t,\cdot) \Vert_{var}$. \\
Consider $t \in [0,T]$ and let $\mu_1$, $\mu_2$ be two measure-mild solutions of PDE \eqref{eq:PDEMu}. We set $\nu := \mu_1 - \mu_2$. Since $\tilde{\Lambda}$ is bounded, we observe that \eqref{eq:defMu} implies $\Vert \nu(t,\cdot) \Vert_{TV} < + \infty$. Moreover, taking into account item 1. of Remark \ref{R32}, we have that $\nu$ satisfies,
\begin{eqnarray}
\label{eq:EqV}
\forall \; \varphi \in \shc_b(\R^d), \; \int_{\R^d} \varphi(x)\nu(t,dx) = \int_0^t \int_{\R^d} \tilde{\Lambda}(r,x_0) \nu(r,dx_0) \int_{\R^d} \varphi(x) P(r,x_0,t,dx) \; dr \ .
\end{eqnarray}
Taking the supremum over $\varphi$ such that $\Vert \varphi \Vert_{\infty} \leq 1$ in each side of \eqref{eq:EqV}, we get
\begin{eqnarray}
\label{eq:EqVNorm}
\Vert \nu(t,\cdot) \Vert_{TV} & \leq & \sup_{(s,x) \in [0,T] \times \R^d} \vert \tilde{\Lambda}(s,x) \vert \int_0^t \Vert \nu(r,\cdot) \Vert_{TV} \; dr \ .
\end{eqnarray}
Gronwall's lemma implies that $\nu(t,\cdot) = 0$.
Uniqueness of measure-mild solution for \eqref{eq:PDEMu} follows. This ends the proof.
\end{proof}
%The next lemma shows that existence of a measure-mild solution on each sub intervals $[r,r+\tau]$ induces a measure-mild solution on whole the interval $[0,T]$.
The next lemma shows how a measure-mild solution of \eqref{eq:PDEMu}, which is a function defined  on $[0,T]$ can be built by defining it recursively on each sub-interval of the form $[r, r + \tau]$. In particular, it will be used in Theorem \ref{prop:UniTrFun} and Proposition \ref{C54}. Its proof is postponed in Appendix (see Section \ref{SProofLem:RecolSol}).
\begin{lem}
\label{lem:RecolSol}
Let $N$ be a strictly positive integer.
Let us fix $\tau > 0$ be a real constant and $\delta := (\alpha_0 := 0  < \cdots < \alpha_k := k \tau < \cdots < \alpha_N := T)$ be a finite 
%regular 
partition of $[0,T]$. \\
A measure-valued map $\mu : [0,T] \rightarrow \shm_f(\R^d)$ satisfies
\begin{equation}
\left\{
\begin{array}{l}
\label{eq:HypInduc}
\mu(0,\cdot) = {\bf u_0} \\
\mu(t,dx) = \int_{\R^d} P(k \tau,x_0,t,dx) \mu(k \tau,dx_0) + \int_{k \tau}^t ds \int_{\R^d} P(s,x_0,t,dx) \tilde{\Lambda}(s,x_0) \mu(s,dx_0)  \ ,
\end{array}
\right .
\end{equation}
for all $t \in [k \tau,(k+1) \tau]$ and $k \in \{0,\cdots,N-1\}$, if and only if $\mu$ is a measure-mild solution (in the sense of Definition \ref{def:MeasureMild}) of \eqref{eq:PDEMu}.
\end{lem}
We now come back to the case where the bounded, Borel measurable real-valued function $\Lambda$ is defined on $[0,T] \times \R^d \times \R \times \R^d$. Let $u : [0,T] \times \R^d \rightarrow \R$ belonging to $L^1([0,T], W^{1,1}(\R^d))$. In the sequel, we set $\tilde{\Lambda}^{u}(t,x) := \Lambda(t,x,u(t,x), \nabla u(t,x))$. $\mu^u$ will denote the measure-valued map $\mu$ defined by \eqref{eq:defMu} with $\tilde{\Lambda} = \tilde{\Lambda}^u$, i.e.,
\begin{eqnarray}
\label{eq:524}
\int_{\R^d} \varphi(x) \mu^u(t,dx) = \E \Big[ \varphi(Y_t) \exp \Big( \int_0^t \tilde{\Lambda}^u(s,Y_s)ds \Big) \Big], \textrm{ for all } \varphi \in \shc_b(\R^d), t \in [0,T] \ .
\end{eqnarray}
By Proposition \ref{prop:MuPDE}, it follows that $\mu^u$ is the unique measure-mild solution of the linear PDE \eqref{eq:PDEMu} with $\tilde{\Lambda} = \tilde{\Lambda}^u$. \eqref{eq:524} can be interpreted as a Feynman-Kac type representation for the measure-mild solution $\mu^u$ of the linear PDE \eqref{eq:PDEMu}, for the corresponding $\tilde{\Lambda}^u$. More generally, Theorem \ref{thm:FKForm} below establishes such representation formula for a mild solution of the semilinear PDE \eqref{eq:PDE}.
% which constitutes the main result of this part.
\begin{thm}
\label{thm:FKForm}
Assume that Assumption \ref{ass:main0} is fulfilled. We indicate by $Y$ the unique strong solution of \eqref{eq:SDE}. 
Suppose that $\Lambda : [0,T] \times \R^d \times \R \times \R^d \rightarrow \R$ is bounded and Borel measurable.
A function $u : [0,T] \times \R^d \longrightarrow \R$ in $L^1([0,T],W^{1,1}(\R^d))$ is a mild solution of \eqref{eq:PDE} if and only if, for all 
%$\varphi \in \shc_0^{\infty}(\R^d)$, 
$\varphi \in \shc_b(\R^d)$,
$t \in [0,T]$,
\begin{eqnarray}
\label{eq:FKForm}
\int_{\R^d} \varphi(x)u(t,x) dx = \E\Big[\varphi(Y_t) \exp \Big(\int_0^t \Lambda(s,Y_s,u(s,Y_s), \nabla u(s,Y_s)) \Big)   \Big] \ .
\end{eqnarray}
\end{thm}
A function $u$ verifying \eqref{eq:FKForm} will be called a {\bf{Feynman-Kac type representation}}
 of \eqref{eq:PDE}.
\begin{proof}
% We first suppose that $u$ is a mild solution of \eqref{eq:PDE}. The aim is then to show that $u$ satisfies the Feynman-Kac equation \eqref{eq:FKForm}. \\
% Since $\Lambda$ is supposed to be bounded, Borel and $u$ is Borel, 
%  it is clear that 
We set
\begin{eqnarray}
\label{eq:tildeLambda}
\tilde{\Lambda}^{u}(t,x) := \Lambda(t,x,u(t,x), \nabla u(t,x)) \ ,
\end{eqnarray}
which is is bounded and Borel measurable. 
The result follows by applying Proposition \ref{prop:MuPDE}
with $\tilde \Lambda = \tilde{\Lambda}^{u}$.

\end{proof}
We now precise more restrictive assumptions to ensure regularity properties of the transition probability function $P(s,x_0,t,dx)$ used in the sequel.
\begin{ass}
\label{ass:mainP3}
\begin{enumerate}
	\item   $\Phi$ and $g$ are functions defined on $[0,T] \times \R^d$ taking values respectively in $M_{d,p}(\R)$ and $\R^d$.
% and are uniformly Holder continuous (exponent $0 < \alpha \leq 1$) w.r.t. $t$ and uniformly Lipschitz w.r.t. $x$ : 
There exist $\alpha \in ]0,1]$, $C_{\alpha}, L_{\Phi}, L_g > 0$, 
such that for any $(t,t',x,x') \in [0,T] \times [0,T] \times \R^d \times \R^d$,
\begin{eqnarray*}
	\vert \Phi(t,x) - \Phi(t,x') \vert \leq C_{\alpha} \vert t-t' \vert^{\alpha} + L_{\Phi} \vert x-x' \vert \ , \\
	\vert g(t,x) - g(t',x') \vert \leq C_{\alpha} \vert t-t' \vert^{\alpha} + L_{g} \vert x-x' \vert \ .	
\end{eqnarray*}
	\item $\Phi$ and $g$ belong to $C^{0,3}_b$. In particular, $\Phi$, $g$ are uniformly bounded and $M_\Phi$ (resp. $M_g$) denote the upper bound of $\vert \Phi \vert$ (resp. $\vert g \vert$).
    \item $\Phi$ is non-degenerate, i.e. there exists $c > 0$ such that for all $x \in \R^d$ 
    \begin{equation}  \label{def:NonDeg}
    \inf_{s\in[0,T]} \inf_{v \in \R^d \setminus \{0\}} \; \frac{\langle v,\Phi(s,x)\Phi^{t}(s,x)v 
    \rangle}{\vert v \vert^2} \geq c > 0.
    \end{equation}

 % following sense: there exists $c > 0$ s.th. for all $(y,z) \in \R^d \times \R$ 
 % \begin{equation}  \label{ENonDeg2}
 % \inf_{s\in[0,T]} \inf_{v \in \R^d \setminus \{0\}} \; \frac{\langle v,\Phi(s,y,z)\Phi^{t}(s,y,z)v 
 % \rangle}{\vert v \vert^2} \geq c > 0.
 % \end{equation}
	\item $\Lambda$ is a Borel real-valued function defined on $[0,T]\times \R^d\times \R \times \R^d$ and Lipschitz uniformly w.r.t. $(t,x)$ i.e. 
	there exists a finite positive real, $L_{\Lambda}$, such that  for any $(t,x,z_1,z_1',z_2,z_2')\in [0,T] \times \R^d \times \R^2 \times (\R^d)^2$,
we have
\begin{eqnarray}
\label{eq:LipAss}
\vert \Lambda(t,x,z_1,z_2)-\Lambda(t,x,z_1',z_2')\vert \leq L_{\Lambda} ( \vert z_1-z_1' \vert + \vert z_2-z_2' \vert )\ .
\end{eqnarray} 

	\item $\Lambda$ is supposed to be uniformly bounded:  let $M_\Lambda$ be an upper bound for $\vert \Lambda \vert$.

\item ${\bf u_0}$ is a Borel probability measure on $\R^d$ admitting a bounded density (still denoted by the 
same letter)  belonging to $W^{1,1}(\R^d)$.
\end{enumerate}
\end{ass}
\begin{thm}
\label{prop:UniTrFun}
Under Assumption \ref{ass:mainP3}, there exists a unique mild solution
 $u$ of \eqref{eq:PDE} in $L^{1}([0,T],W^{1,1}(\R^d)) \cap L^{\infty}([0,T] \times \R^d,\R)$.
\end{thm}
%!!!VOIR SI GARDER MEME NOTATION $\Vert \cdot \Vert_{\infty}$ POUR $L^{\infty}$ !!! \\
The idea of the proof will be first to construct a unique "mild solution" $u ^k$ of \eqref{eq:PDE} on each subintervals of the form $[k \tau,(k+1) \tau]$ with $k \in \{0,\cdots,N-1\}$ and $\tau > 0$ a constant supposed to be fixed for the moment. This will be the object of Lemma \ref{lem:LocalSol}. Secondly we will show that the function $u : [0,T] \times \R^d \rightarrow \R$, defined by being equal to $u^k$ on each $[k \tau,(k+1) \tau]$, is indeed a mild solution of \eqref{eq:PDE} on $[0,T] \times \R^d$. This will be a consequence of Lemma \ref{lem:RecolSol}. Finally, uniqueness will follow classically from the 
Lipschitz property of $\Lambda$. 

 By Assumption  \ref{ass:mainP3} 
 and item 1. of Lemma \ref{lem:transfun}, the transition kernels are absolutely continuous and 
 $  P(s,x_0,t,dx) = p(s,x_0,t,x)dx$ for some Borel function $p$.
Let us fix  $\phi \in L^1(\R^d) \cap L^{\infty}(\R^d)$.
For $r \in [0,T-\tau]$, we first define a function $\widehat{u_0}$ on $[r,r+\tau] \times \R^d $ by setting
\begin{equation}
\label{eq:DEfu0}
\widehat{u_0}(r,\phi)(t,x) := \int_{\R^d} p(r,x_0,t,x) \phi(x_0)dx_0, \; (t,x) \in [r,r+\tau] \times \R^d. 
\end{equation}
Consider now the map $\Pi : L^1([r,r+\tau],W^{1,1}(\R^d)) \rightarrow L^1([r,r+\tau],W^{1,1}(\R^d)) $ given by
\begin{equation}
\label{eq:DefPi}
\begin{array}{l}
\Pi(v)(t,x) := \int_{r}^t ds \int_{\R^d} p(s,x_0,t,x) \Lambda(s,x_0,v + \widehat{u_0}(r,\phi),\nabla (v + \widehat{u_0}(r,\phi))) 
\big(v + \widehat{u_0})(r,\phi)(s,x_0) dx_0 \ , \\
\end{array}
\end{equation}
\begin{eqnarray}
\label{eq:notation}
\Lambda(t,z,v, \nabla v) := \Lambda(t,z,v(t,z), \nabla v(t,z)) \quad \textrm{with} \quad (t,z) \in [0,T] \times \R^d \ ,
\end{eqnarray}
that will also be used in the sequel.

Later, the dependence on $r, \phi$ will be omitted when it is self-explanatory.
Since $\phi$ 
%and $u_ 0$
 belongs to $L^1(\R^d) \cap L^{\infty}(\R^d)$, also taking into 
account  \eqref{eq:majordens},
we have
\begin{equation}
\begin{array}{l}
\Vert \widehat{u_0}(t,\cdot) \Vert_1 \leq \Vert \phi \Vert_1 \quad \textrm{and} \quad \Vert \widehat{u_0}(t,\cdot) \Vert_{\infty} \leq \Vert \phi \Vert_{\infty}, \textrm{ if } t \in [r,r+\tau] \ .
\end{array}
\end{equation}
The lemma below establishes, under a suitable choice of $\tau > 0$, existence and uniqueness of 
the mild solution on $[r,r+\tau]$, with initial condition $\phi$ at time $r$,
i.e. existence and uniqueness of the fixed-point for the application $\Pi$.
\begin{lem}
\label{lem:LocalSol}
Assume the validity of Assumption \ref{ass:mainP3}. Let $\phi \in L^1(\R^d) \cap L^{\infty}(\R^d)$. \\
Let $M > 0$ such that $M \geq \max(\Vert \phi \Vert_{\infty};\Vert \phi \Vert_1)$.
Then, there is $\tau > 0$ only depending on $M_\Lambda$ and on $C_u,c_u$ (the constants coming from inequalities 
 \eqref{eq:1011a} and
 \eqref{eq:1011}, only depending on $\Phi$, $g$)
such that for any  $r \in [0,T-\tau]$,
 $\Pi$ admits a unique fixed-point in $L^1([r,r + \tau],B(0,M)) \cap B_{\infty}(0,M)$, where $B(0,M)$ (resp. $B_{\infty}(0,M)$) denotes the centered ball in $W^{1,1}(\R^d)$ (resp. $L^{\infty}([r,r+\tau] \times \R^d, \R)$) of radius $M$.
\end{lem}
\begin{proof}

We first insist on the fact that all along the proof, the dependence of $\widehat{u}_0$ w.r.t. $r,\phi$ in \eqref{eq:DefPi} will be omitted to simplify notations. Let us fix $r \in [0,T - \tau]$. 
% By item 1. of Lemma \ref{lem:transfun}, the transition probabilities are absolutely continuous and 
% $  P(s,x_0,t,dx) = p(s,x_0,t,x)dx$ for some Borel function $p$. 

The rest of the proof relies on a fixed-point argument in the Banach space $L^{1}([r,r+\tau],W^{1,1}(\R^d))$ equipped with the norm $\Vert f \Vert_{1,1} := \int_r^{r+\tau} \Vert f(s,\cdot) \Vert_{W^{1,1}(\R^d)} ds$ and for the map $\Pi$ \eqref{eq:DefPi}. Moreover, we emphasize that $L^1([r,r + \tau],B(0,M)) \cap B_{\infty}(0,M)$ is complete as a closed subset of $L^1([r,r + \tau],B(0,M))$. 

We first check that $\Pi \Big(L^1([r,r+\tau],B(0,M)) \cap B_{\infty}(0,M) \Big) \subset L^1([r,r+\tau],B(0,M)) \cap B_{\infty}(0,M)$. Let us fix $v \in L^1([r,r+\tau],B(0,M)) \cap B_{\infty}(0,M)$. 
 For $t \in [r,r+\tau]$,
\begin{eqnarray}
\label{eq:L1Pi}
\Vert \Pi(v)(t,\cdot) \Vert_1 & = & \int_{\R^d} \vert \Pi(v)(t,x) \vert dx \nonumber \\
& \leq & M_{\Lambda} \int_r^t \big( \Vert v(s,\cdot) \Vert_1 + \Vert \widehat{u_0}(s,\cdot) \Vert_1 \big) ds \nonumber \\
& \leq & 2M_{\Lambda}M \tau \ ,
\end{eqnarray}
where we have used the fact that $x \mapsto p(s,x_0,t,x)$ is a probability density, the boundedness of $\Lambda$ and the bounds $\Vert v(s,\cdot) \Vert_1 \leq M$ and $\Vert \widehat{u_0}(s,\cdot) \Vert_1 \leq M$ for $s \in [r,r+\tau]$. 
 %All interversion of integrals are justified by Fubini theorem. \\

Let us fix $t \in[r,r+\tau]$.
%%OLD.  J'ESPERE QUE L'ON NA PAS BESOIN DE DIFFERENTIER DEUX FOIS
%%Since the transition probability function $x \mapsto p(s,x_0,t,x)$ is 
%twice continuously differentiable for $0 \leq s < t \leq T$ (see item 2. of Lemma \ref{lem:transfun}) and taking into account inequality \eqref{eq:1011},
By item 2. of Lemma \ref{lem:transfun}, taking into account
inequality \eqref{eq:1011},
  we differentiate under the integral sign
with respect to $x$,  to 
obtain
that $\nabla \Pi(v)(t,\cdot)$ exists (in the sense of distributions) and is a real-valued function such that for almost all $x \in \R^d$,
\begin{eqnarray}
\label{eq:TfGrad}
\nabla \Pi(v)(t,x) & = & \int_r^t ds \int_{\R^d} \nabla_x p(s,x_0,t,x) \big(v + \widehat{u_0} \big)(s,x_0) \Lambda(s,x_0,v+\widehat{u_0}, \nabla (v+\widehat{u_0})) dx_0.
\end{eqnarray}
Integrating each side of \eqref{eq:TfGrad} on $\R^d$ w.r.t. $dx$ and using inequality \eqref{eq:1011} 
%(with $(m_1,m_2) = (0,1)$) 
yield
\begin{eqnarray}
\label{eq:L1GradPi}
\Vert \nabla \Pi(v)(t,\cdot) \Vert_1 & = & \int_{\R^d} \vert \nabla \Pi(v)(t,x) \vert dx \nonumber \\
& \leq & M_{\Lambda} \int_r^t \frac{ds}{\sqrt{t-s}} \int_{\R^d} dx \int_{\R^d} C_u
%\frac{e^{\frac{-c_u\vert x-x_0 \vert}{t-s}}}{\sqrt{(t-s)^d}}
 q(s,x_0,t,x)
 \big( \vert v(s,x_0) \vert + \vert \widehat{u_0}(s,x_0) \vert \big) dx_0 \nonumber \\
& = & C_u M_{\Lambda} \int_r^t \frac{ds}{\sqrt{t-s}} \int_{\R^d} \big( \vert v(s,x_0) \vert + \vert \widehat{u_0}(s,x_0) \vert \big) dx_0 \nonumber \\
& \leq &   C_u M_{\Lambda} \int_r^t \big(\Vert v(s,\cdot) \Vert_1 + \Vert \widehat{u_0}(s,\cdot) \Vert_1 \big) \frac{ds}{\sqrt{t-s}} \nonumber \\
& \leq & 4 C_u M_{\Lambda} M \sqrt{\tau}  \ ,
\end{eqnarray}
where the constant $C_u$ and
% $\hat C := \hat C(C_u,c_u) > 0$ and $C_u,c_u$  and 
the Gaussian kernel
$q$ 
come from inequality \eqref{eq:1011} and only depending on $\Phi$ and $g$. Consequently, taking into account \eqref{eq:L1Pi} and \eqref{eq:L1GradPi}, we obtain,
\begin{eqnarray}
 \Vert \Pi(v) \Vert_{1,1} = \int_{r}^{r + \tau} \Vert \Pi(v)(t,\cdot) \Vert_{W^{1,1}(\R^d)} dt & \leq & 2MM_{\Lambda} (\tau^2 + 2 C_u \tau \sqrt{\tau}) \ .
\end{eqnarray}
Moreover using inequality \eqref{eq:1011a}, 
% with $(m_1,m_2) = (0,0)$,
 %gives existence of a constant $\bar C := \bar C(C_u,c_u)$ such that
gives
\begin{eqnarray}
\Vert \Pi(v) \Vert_{\infty} & \leq & 2 C_u MM_{\Lambda} \tau \ .
\end{eqnarray}
Now, setting
%% OLD
% \begin{equation}
% \label{eq:DefTau}
% \tau := \min \Big(\sqrt{\frac{1}{6M_{\Lambda}}}; \left(\frac{1}{12\hat C M_{\Lambda}}\right)^{\frac{2}{3}}; \frac{1}{6 \bar C M_{\Lambda}} \Big),
% \end{equation}
%%% FIN OLD
\begin{equation}
\label{eq:DefTau}
\tau := \min \Big(\sqrt{\frac{1}{6M_{\Lambda}}}; \left(\frac{1}{6 C_u M_{\Lambda}}\right)^{\frac{2}{3}}; \frac{1}{2 C_u M_{\Lambda}} \Big),
\end{equation}
we have 
$$
2MM_{\Lambda} (\tau^2 + 2 C_u \tau \sqrt{\tau}) \leq M \quad \textrm{and} \quad 2 C_u MM_{\Lambda} \tau \leq M \ ,
$$
%%% OLD
% $$
% 2MM_{\Lambda} (\tau^2 + 2\hat C\tau \sqrt{\tau}) \leq \frac{2M}{3} \quad \textrm{and} \quad 2 \bar C MM_{\Lambda} \tau \leq \frac{M}{3} \ ,
% $$
%% FIN OLD
which implies
$$
\Vert \Pi(v) \Vert_{1,1} \leq M \quad \textrm{ and } \quad \Vert \Pi(v) \Vert_{\infty} \leq M \ .
$$
We deduce that $\Pi(v) \in L^1([r,r+\tau],B(0,M)) \cap B_{\infty}(0,M)$. \\
Let us fix $t \in [r,r+\tau]$, $v_1,v_2 \in L^1([r,r+\tau],B(0,M)) \cap B_{\infty}(0,M)$. $\Lambda$ being bounded and Lipschitz, the notation introduced in \eqref{eq:notation} and inequality \eqref{eq:LipV} imply
\begin{eqnarray}
\label{eq:contractT}
\Vert \Pi(v_1)(t,\cdot) - \Pi(v_2)(t,\cdot) \Vert_{1} & \leq & \int_r^t ds \int_{\R^d} \Big \vert v_1(s,x_0) \Lambda(s,x_0,v_1+\widehat{u_0}, \nabla (v_1+\widehat{u_0})) - v_2(s,x_0) \Lambda(s,x_0,v_2+\widehat{u_0}, \nabla (v_2+\widehat{u_0})) \Big \vert dx_0 \nonumber \\
&& \; + \int_r^t ds \int_{\R^d} \vert \widehat{u_0}(s,x_0) \vert \; \Big \vert \Lambda(s,x_0,v_1+\widehat{u_0}, \nabla (v_1+\widehat{u_0})) - \Lambda(s,x_0,v_2+\widehat{u_0}, \nabla (v_2+\widehat{u_0})) \Big \vert dx_0 \nonumber \\
& \leq & \int_r^t ds \Big( \int_{\R^d} \vert v_1(s,x_0) - v_2(s,x_0) \vert\;  \vert \Lambda(s,x_0,v_1+\widehat{u_0}, \nabla (v_1+\widehat{u_0})) \vert dx_0 \nonumber \\
%&& \; + \int_r^t ds \int_{\R^d} \vert \nabla v_1(s,x_0) - \nabla v_2(s,x_0) \vert\;  \vert \Lambda(s,x_0,v_1+\widehat{u_0}, \nabla (v_1+\widehat{u_0})) \vert dx_0 \nonumber \\
&& \; + L_{\Lambda} \int_r^t ds \int_{\R^d} \big( \vert \widehat{u_0}(s,x_0) \vert  + \vert v_2(s,x_0) \vert \big) \; \vert v_1(s,x_0) - v_2(s,x_0) \vert dx_0 \nonumber \\
&& \; + L_{\Lambda} \int_r^t ds \int_{\R^d} \big( \vert \widehat{u_0}(s,x_0) \vert  + \vert v_2(s,x_0) \vert \big) \; \vert \nabla v_1(s,x_0) - \nabla v_2(s,x_0) \vert dx_0  \nonumber \\
%&& \; + \int_r^t ds \int_{\R^d} \big( \vert \widehat{u_0}(s,x_0) \vert + \vert v_2(s,x_0) \vert \big) \; \vert \Lambda(s,x_0,v_1+\widehat{u_0}, \nabla (v_1+\widehat{u_0})) - \Lambda(s,x_0,v_2+\widehat{u_0}, \nabla (v_2+\widehat{u_0})) \vert dx_0 \Big) \nonumber \\
%& \leq & (M_{\Lambda} + 2ML_{\Lambda}) \int_0^t ds \int_{\R^d} \vert v_1(s,x_0) - v_2(s,x_0) \vert + \vert \nabla_{x_0} v_1(s,x_0) - \nabla_{x_0} v_2(s,x_0) \vert dx_0 \nonumber \\
& \leq & (M_{\Lambda} + 2ML_{\Lambda}) \int_r^t \Vert v_1(s,\cdot) - v_2(s,\cdot) \Vert_{W^{1,1}(\R^d)} ds \ ,
\end{eqnarray}
where we have used the fact that $\int_{\R^d} p(s,x_0,t,x) dx = 1, \; 0 \leq s < t \leq T $. \\ 
In the same way 
%and by 
using inequality \eqref{eq:1011} 
% with $(m_1,m_2) = (0,1)$, 
\begin{eqnarray}
\label{eq:contractT2}
\Big \Vert \nabla \Big(\Pi(v_1) - \Pi(v_2) \Big)(t,\cdot) \Big \Vert_{1} & \leq & %\int_{\R^d} \int_0^t \int_{\R^d} \nabla p(s,x_0,t,x) \Big \vert v_1(s,x_0) \Lambda(s,x_0,v_1, \nabla v_1)) - v_2(s,x_0) \Lambda(s,x_0,v_2, \nabla v_2) \Big \vert dx_0 ds \; dx \nonumber \\
%&& \; + \int_{\R^d} \int_0^t \int_{\R^d} \nabla p(s,x_0,t,x) \; \vert \widehat{u_0}(s,x_0) \vert \; \Big \vert \Lambda(s,x_0,v_1, \nabla v_1)) - \Lambda(s,x_0,v_2, \nabla v_2) \Big \vert dx_0 ds \; dx \nonumber \\
%& \leq & CL_{\Lambda} \int_{\R^d} \int_0^t \int_{\R^d} \nabla p(s,x_0,t,x) \Big ( \vert v_1(s,x_0) - v_2(s,x_0) \vert + \vert \nabla_{x_0} v_1(s,x_0) - \nabla_{x_0} v_2(s,x_0) \vert \Big) dx_0 ds \nonumber \\
C_u (M_{\Lambda} + 2ML_{\Lambda}) \int_{\R^d} \int_r^t \int_{\R^d} 
\frac{1}{\sqrt{t-s}} 
%\frac{1}{\sqrt{(t-s)^d}}e^{-c_u \frac{\vert x-x_0 \vert^2}{t-s}}
q(s,x_0,t,x)
 \Big ( \vert v_1(s,x_0) - v_2(s,x_0) \vert \nonumber \\ 
&& + \; \vert \nabla v_1(s,x_0) - \nabla v_2(s,x_0) \vert \Big) dx_0 ds dx\ .
\end{eqnarray}
By Fubini's theorem we have
\begin{eqnarray}
\label{eq:contractT22}
\Big \Vert \nabla \Big(\Pi(v_1) - \Pi(v_2) \Big)(t,\cdot) \Big \Vert_{1} & \leq & \tilde{C} \int_r^t \frac{1}{\sqrt{t-s}} \Vert v_1(s,\cdot) - v_2(s,\cdot) \Vert_{W^{1,1}(\R^d)} ds \ ,
\end{eqnarray}
with $\tilde{C} := \tilde{C}(C_u,c_u,M_{\Lambda},L_{\Lambda},M)$ some positive constant. From \eqref{eq:contractT} and \eqref{eq:contractT22}, we deduce there exists a strictly positive constant $C = C(C_u,c_u,\Phi, g, \Lambda, M)$ (which may change from line to line) such that for all $t \in [r,r+\tau]$,
\begin{eqnarray}
\label{eq:contractT3}
\Vert \Pi(v_1)(t,\cdot) - \Pi(v_2)(t,\cdot) \Vert_{W^{1,1}(\R^d)} & \leq & C \Big \{ \int_r^t \Vert v_1(s,\cdot) - v_2(s,\cdot) \Vert_{W^{1,1}(\R^d)} ds \nonumber \\
&& + \; \int_r^t \frac{1}{\sqrt{t-s}} \Vert v_1(s,\cdot) - v_2(s,\cdot) \Vert_{W^{1,1}(\R^d)} ds \Big \} \ .
\end{eqnarray}
Iterating the procedure once again yields
\begin{eqnarray}
\label{eq:contractT4}
\Vert \Pi^2(v_1)(t,\cdot) - \Pi^2(v_2)(t,\cdot) \Vert_{W^{1,1}(\R^d)} & \leq & C \Big \{ \int_r^t \int_r^s \Vert v_1(\theta,\cdot) - v_2(\theta,\cdot) \Vert_{W^{1,1}(\R^d)} d\theta ds \nonumber \\
&& + \; \int_r^t \int_r^{s} \frac{1}{\sqrt{t-s}} \frac{1}{\sqrt{s-\theta}} \Vert v_1(\theta,\cdot) - v_2(\theta,\cdot) \Vert_{W^{1,1}(\R^d)} d\theta ds \Big \}\; ,  \nonumber \\
\end{eqnarray}
for all $t \in [r,r+\tau]$. Interchanging the order in the second integral in the r.h.s. of \eqref{eq:contractT4}, we obtain 
\begin{eqnarray}
\label{eq:848}
\int_r^t \int_r^{s} \frac{1}{\sqrt{t-s}} \frac{1}{\sqrt{s-\theta}} \Vert v_1(\theta,\cdot) - v_2(\theta,\cdot) \Vert_{W^{1,1}(\R^d)} d\theta ds & = & \int_r^t  d\theta \; \Vert v_1(\theta,\cdot) - v_2(\theta,\cdot) \Vert_{W^{1,1}(\R^d)} \int_{\theta}^{t} \frac{1}{\sqrt{t-s}} \frac{1}{\sqrt{s-\theta}} ds \nonumber \\
& = & \int_r^t  d\theta \; \Vert v_1(\theta,\cdot) - v_2(\theta,\cdot) \Vert_{W^{1,1}(\R^d)} \int_0^{\alpha} \frac{1}{\sqrt{\alpha-\omega}} \frac{1}{\sqrt{\omega}} d \omega, \; 
%\alpha = t-\theta 
\nonumber \\
& \leq & 4 \int_r^t \Vert v_1(\theta,\cdot) - v_2(\theta,\cdot) \Vert_{W^{1,1}(\R^d)} d\theta \ ,
\end{eqnarray}
where the latter line above comes from the fact for all $\theta > 0$, 
$$\int_0^{\theta} \frac{1}{\sqrt{\theta-\omega}} \frac{1}{\sqrt{\omega}} d \omega = \int_0^{1} \frac{1}{\sqrt{1-\omega}} \frac{1}{\sqrt{\omega}} d \omega = 
B\left(\frac{1}{2},\frac{1}{2}\right) =  \Gamma^2(\frac{1}{2}) = \pi,$$
 $\Gamma, B$ denoting respectively  the Euler gamma and Beta functions. \\
Injecting inequality \eqref{eq:848} in \eqref{eq:contractT4}, we obtain for all $t \in [r,r+\tau]$
\begin{eqnarray}
\label{eq:849}
\Vert \Pi^2(v_1)(t,\cdot) - \Pi^2(v_2)(t,\cdot) \Vert_{W^{1,1}(\R^d)} & \leq & 
C (4+\tau) \int_r^t \Vert v_1(s,\cdot) - v_2(s,\cdot) \Vert_{W^{1,1}(\R^d)} ds \ .
\end{eqnarray}
Iterating previous inequality, one obtains the following.
%By induction, we have 
For all $k \geq 1$, $t \in [r,r+\tau]$,
\begin{eqnarray}
\label{eq:850}
\Vert \Pi^{2k}(v_1)(t,\cdot) - \Pi^{2k}(v_2)(t,\cdot) \Vert_{W^{1,1}(\R^d)} & \leq & C^k (4+\tau)^{k} \int_r^t \frac{(t-s)^{k-1}}{(k-1)!}\Vert v_1(s,\cdot) - v_2(s,\cdot) \Vert_{W^{1,1}(\R^d)} ds \ .
\end{eqnarray}
By induction on $k \ge 1$ \eqref{eq:850} can indeed be established.
Finally, by integrating each sides of \eqref{eq:850} w.r.t. $dt$ and using Fubini's theorem,  for $k \geq 1$, we obtain
\begin{eqnarray}
\label{eq:851}
\int_r^{r+\tau} \Vert \Pi^{2k}(v_1)(t,\cdot) - \Pi^{2k}(v_2)(t,\cdot) \Vert_{W^{1,1}(\R^d)} dt & \leq & C^k (4+\tau)^{k}  \frac{T^{k}}{k!} \int_r^{r+\tau} \Vert v_1(s,\cdot) - v_2(s,\cdot) \Vert_{W^{1,1}(\R^d)} ds \ . \nonumber \\
\end{eqnarray}
For $k_0 \in \N$ large enough, $\frac{(C(4+\tau)T)^{k_0}}{T(k_0)!}$ will be strictly smaller than $1$ and $\Pi^{2 k_0}$ will admit a unique fixed-point by Banach fixed-point theorem. In consequence, it implies easily that $\Pi$ will also admit a unique fixed-point and this concludes the proof of Lemma   \ref{lem:LocalSol}.
\end{proof}

\begin{proof}[Proof of Theorem \ref{prop:UniTrFun}.]
Without restriction of generality, we can suppose there exists $N \in \N^*$ such that $T = N \tau$,
where we recall that $\tau$ is given by \eqref{eq:DefTau}. Similarly to the notations used in the preceding proof, in all the sequel, we agree that for $M > 0$, $B(0,M)$ (resp. $B_{\infty}(0,M)$) denotes the centered ball of radius $M$ in $W^{1,1}(\R^d)$ (resp. in $L^{\infty}([0,T] \times \R^d,\R)$ or in $L^{\infty}([r,r + \tau] \times \R^d,\R)$ for $r \in [0,T-\tau]$ according to the context). The notations introduced in \eqref{eq:notation} will also be used in the present proof. 

%For $\tau$ given by \eqref{eq:DefTau}, the previous lemma gives existence and uniqueness of a fixed-point $v^{k}$ of the map $\Pi$, in $L^1([k\tau,(k+1)\tau],B(0,M)) \cap B_{\infty}(0,M)$, for $k \in \{0,\cdots,N-1\}$ and $M$ well-chosen. \\
%such that $M \geq \max(\Vert u_0 \Vert_{\infty};\Vert u_0 \Vert_1)$. 
Indeed, for $r=0$, $\phi = u_0$ and $M \geq \max(\Vert u_0 \Vert_{\infty};\Vert u_0 \Vert_1)$, Lemma \ref{lem:LocalSol} implies there exists a unique function $v^0 : [0,\tau] \times \R^d \rightarrow \R$ (belonging to $L^1([0,\tau],B(0,M)) \cap B_{\infty}(0,M)$) such that for $(t,x) \in [0,\tau] \times \R^d$,
%In particular, it follows that for all $k \in \{1,\cdots,N\}$, $v^k$ is given by
\begin{equation}
\label{eq:vk}
v^0(t,x) = \int_{0}^t ds \int_{\R^d} p(s,x_0,t,x) (v^0(s,x_0)+\widehat{u_0}^0(s,x_0))\Lambda \Big(s,x_0,v^0+\widehat{u_0}^0,\nabla ( v^0 + \widehat{u_0}^0 ) \Big) dx_0 \ ,
\end{equation}
where $\widehat{u_0}^0(t,x)$ is given by \eqref{eq:DEfu0} with $\phi = u_0$, i.e.
\begin{equation}
\widehat{u_0}^0(t,x) = \int_{\R^d} p(0,x_0,t,x) u_0(x_0)dx_0 \ , \quad (t,x) \in [0,\tau] \times \R^d \ .
\end{equation}
Setting $u^0 := \widehat{u_0}^0 + v^0 $, i.e. 
\begin{equation}
u^0(t,\cdot) = \int_{\R^d} p(0,x_0,t,\cdot) u_0(x_0)dx_0 + v^0(t,\cdot), \quad t \in [0,\tau] \ ,
\end{equation}
it appears that $u^0$ satisfies for all $(t,x) \in [0,\tau] \times \R^d$
\begin{equation}
\label{eq:uPtFixe}
u^0(t,x) = \int_{\R^d} p(0,x_0,t,x) u_0(x_0) dx_0 + \int_{0}^t ds \int_{\R^d} p(s,x_0,t,x) u^0(s,x_0) \Lambda(s,x_0,u^0, \nabla u^0) dx_0 \ .
\end{equation}
Let us fix $k \in \{1,\cdots,N-1\}$. Suppose now given a family of functions $u^1, u^2, \cdots, u^{k-1}$, where for all $j \in \{1,\cdots, k-1\}$, $u^j \in L^1([j\tau,(j+1)\tau],W^{1,1}(\R^d)) \cap L^{\infty}([j\tau,(j+1)\tau] \times \R^d, \R)$ and satisfies for all $(t,x) \in [j\tau,(j+1)\tau] \times \R^d$,
\begin{equation}
\label{eq:uPtFixe2}
u^j(t,x) = \int_{\R^d} p(j\tau,x_0,t,x) u^{j-1}(j\tau,x_0) dx_0 + \int_{j \tau}^t ds \int_{\R^d} p(s,x_0,t,x) u^j(s,x_0) \Lambda(s,x_0,u^j, \nabla u^j) dx_0 \ .
\end{equation}
%We want to show that there exists a function $u^{k+1} : [(k+1) \tau, (k+2)\tau] \times \R^d \rightarrow \R$ such that for all $ (t,x) \in [(k+1) \tau, (k+2)\tau] \times \R^d$,
%\begin{equation}
%\label{eq:uPtFixe3}
%u^{k+1}(t,x) = \int_{\R^d} p((k+1)\tau,x_0,t,x) u^{k}((k+1)\tau,x_0) dx_0 + \int_{(k+1) \tau}^t ds \int_{\R^d} p(s,x_0,t,x) u^{k+1}(s,x_0) \Lambda(s,x_0,u^{k+1}, \nabla u^{k+1}) dx_0 \ .
%\end{equation}
Let us introduce
\begin{eqnarray}
\widehat{u_0}^k(t,x) & := & \widehat{u_0}(u^{k-1})(t,x) \nonumber \\
& = & \int_{\R^d} p(k\tau,x_0,t,x) u^{k-1}(k\tau,x_0)dx_0 \ , \quad (t,x) \in [k\tau,(k+1)\tau] \times \R^d \ ,
\end{eqnarray}
where the second inequality comes from \eqref{eq:DEfu0} with $r = k \tau$ and $\phi = u^{k-1}(k \tau,\cdot)$. 

By choosing the real $M$ large enough (i.e. $M \geq \max(\Vert u^{k-1}(k \tau,\cdot) \Vert_{\infty};\Vert u^{k-1}(k \tau,\cdot) \Vert_{1})$), Lemma \ref{lem:LocalSol} applied with $r = k \tau$, $\phi = u^{k-1}(k \tau,\cdot)$ implies existence and uniqueness of a function $v^{k} : [k \tau, (k+1)\tau] \times \R^d \rightarrow \R$ that belongs to $L^1([k\tau,(k+1)\tau],B(0,M)) \cap B_{\infty}(0,M)$ and satisfying
\begin{equation}
\label{eq:vk2}
v^k(t,x) = \int_{k \tau}^t ds \int_{\R^d} p(s,x_0,t,x) ( v^k(s,x_0) + \widehat{u_0}^k(s,x_0) ) \Lambda \big(s,x_0,v^k+\widehat{u_0}^k,\nabla (v^k+\widehat{u_0}^k) \big)dx_0 \ ,
\end{equation}
for all $(t,x) \in [k\tau,(k+1)\tau] \times \R^d$. Setting $u^k := \widehat{u_0}^k + v^k$, we have for all $(t,x) \in [k\tau,(k+1)\tau] \times \R^d$
\begin{equation}
\label{eq:uPtFixe3}
u^{k}(t,x) = \int_{\R^d} p(k\tau,x_0,t,x) u^{k-1}(k\tau,x_0) dx_0 + \int_{k \tau}^t ds \int_{\R^d} p(s,x_0,t,x) u^{k}(s,x_0) \Lambda(s,x_0,u^{k}, \nabla u^{k}) dx_0 \ .
\end{equation}
Consequently, by induction we can construct a family of functions $(u^k : [k \tau,(k+1) \tau] \times \R^d \rightarrow \R)_{k=0,\cdots,N-1}$ such that for all $k \in \{0,\cdots,N-1\}$, $u^k \in L^1([k\tau,(k+1)\tau],W^{1,1}(\R^d)) \cap L^{\infty}([k\tau,(k+1)\tau] \times \R^d,\R)$ and verifies \eqref{eq:uPtFixe3}.

We now consider the real-valued function $u : [0,T] \times \R^d \rightarrow \R$ defined as being equal to $u^k$ (defined by \eqref{eq:uPtFixe3}) on each interval $[k\tau,(k+1)\tau]$.
Then, Lemma \ref{lem:RecolSol} applied with $\tau$ given by \eqref{eq:DefTau} and
\begin{equation}
\delta = (\alpha_0 := 0 < \cdots  < \alpha_k := k \tau < \cdots < \alpha_N := T = N \tau) \quad \textrm{,} \quad \mu(t,dx) = u(t,x)dx ,
\end{equation}
shows that $u$ is a mild solution of \eqref{eq:PDE} on $[0,T] \times \R^d$, in the sense of Definition \ref{def:SolPDE}, item 2. It now remains to ensure that $u$ is indeed the unique mild solution of \eqref{eq:PDE} on $[0,T] \times \R^d$ belonging to $L^1([0,T],W^{1,1}(\R^d)) \cap L^{\infty}([0,T] \times \R^d, \R)$. This follows, in a classical way, by boundedness and Lipschitz property of $\Lambda$. \\
Indeed, if $U$, $V$ are two mild solutions of \eqref{eq:PDE}, then very similar computations as the ones done in \eqref{eq:contractT}, \eqref{eq:contractT22},
% \eqref{eq:Reccur}
 and \eqref{eq:850} to obtain \eqref{eq:851} give the following.
%For $j \in \N^{\star}$ large enough, 
There exists 
%$C := C(\Phi,g, \Lambda,U,V,j) > 0$ 
$C := C(\Phi,g, \Lambda,U,V) > 0$
such that
\begin{equation}
\label{eq:457}
\int_0^T \Vert U(t,\cdot) - V(t,\cdot) \Vert_{W^{1,1}(\R^d)} dt
 \leq  (5C)^{j} \frac{T^{j-1}}{(j-1)!} \int_0^T \Vert U(s,\cdot) - V(s,\cdot) \Vert_{W^{1,1}(\R^d)} ds.
%\nonumber \\
%& < & \int_0^T \Vert U(s,\cdot) - V(s,\cdot) \Vert_{W^{1,1}(\R^d)} ds \ ,
\end{equation}
If we choose  $j \in \N^{\star}$ large enough so that $(5C)^{j} \frac{T^{j-1}}{(j-1)!} < 1$, 
 we obtain  $U(t,x)=V(t,x)$ for almost all $(t,x) \in [0,T] \times \R^d$. This concludes the proof of Theorem \ref{prop:UniTrFun}.
\end{proof}
\begin{corro}  \label{CFK}
Under Assumption \ref{ass:mainP3}, there exists a unique function $u : [0,T] \times \R^d \longrightarrow \R$ satisfying the Feynman-Kac equation \eqref{eq:FKForm}. In particular, such $u$ coincides with the mild solution of \eqref{eq:PDE}.
\end{corro}

In the case where the function $\Lambda$ does not depend on $\nabla u$, existence and uniqueness of a solution of \eqref{eq:PDE} in the mild sense can be proved under weaker assumptions. This is the object of the following result.
\begin{thm} \label{CasLambda}
Assume that Assumption \ref{ass:main0} is satisfied. Let ${\bf u_0} \in \shp(\R^d)$ admitting a bounded density (still denoted by the same letter).
% belonging to $L^1(\R^d)$ and
Let $Y$ the the strong solution of \eqref{eq:SDE} with prescribed $Y_0$. 

We suppose that the transition probability function $P$ (see \eqref{eq:FKLin})
% associated to $Y$
 admits a density $p$ such that $P(s,x_0,t,dx) = p(s,x_0,t,x)dx$, for all $s,t \in [0,T]$, $x_0 \in \R^d$. $\Lambda$ is supposed to satisfy items 4. and 5. of Assumption \ref{ass:mainP3}. Then, there exists a unique mild solution $u$ of \eqref{eq:PDE} in $L^1([0,T],L^1(\R^d))$, i.e. $u$ satisfies
\begin{eqnarray}
u(t,x) = \int_{\R^d} p(0,x_0,t,x)u_0(x_0)dx_0 + \int_0^t  \int_{\R^d} p(s,x_0,t,x) u(s,x_0) \Lambda(s,x_0,u(s,x_0))dx_0 \; ds, \quad (t,x) \in [0,T] \times \R^d \ . \nonumber \\
\end{eqnarray}
\end{thm}
\begin{proof}
Since this theorem can be proved in a very similar way as Theorem \ref{prop:UniTrFun} but with simpler computations, we omit the details.
\end{proof}

\section{Existence/uniqueness of the Regularized Feynman-Kac equation}
\setcounter{equation}{0}
\label{SConverg}
In this section, we introduce a regularized version of PDE~\eqref{eq:PDE} to which we associate a regularized Feynman-Kac equation corresponding to a regularized version of~\eqref{eq:FKForm}. This regularization procedure constitutes a preliminary step for the construction of a particle scheme approximating~\eqref{eq:FKForm}. Indeed, as detailed in the next section devoted to the particle approximation, 
the point dependence of $\Lambda$ on $u$ and $\nabla u$ will require to derive from a discrete measure (based on the particle system) estimates of densities $u$ and their derivatives $\nabla u$, which can classically be achieved by kernel convolution. \\
Assumption \ref{ass:main0} is in force. 
%Let $(\Omega,\shf,(\shf_t)_{t \geq 0},\P)$ be a fixed filtered probability space,
Let ${\bf u_0}$ be a Borel probability measure on $\R^d$ and $Y_0$ a random variable distributed according to ${\bf u_0}$. We consider $Y$ the strong solution of the SDE \eqref{eq:SDE}.

%-------------------------
Let us consider  $(K_{\varepsilon})_{\varepsilon > 0}$, a sequence of mollifiers
 verifying~\eqref{eq:Keps} such that $K$ verifies~\eqref{eq:HypK}. 
% \begin{equation}
% \label{eq:HypK}
% K_{\varepsilon} \xrightarrow[\varepsilon \rightarrow 0]{} \delta_0, \textrm{ (weakly)} \quad \textrm{and} \quad \forall \; \varepsilon > 0, K_{\varepsilon} \in W^{1,1}(\R^d) \cap W^{1,\infty}(\R^d) \ . 
% \end{equation}
Let $\Lambda : [0,T] \times \R^d \times \R \times \R^d \rightarrow \R$ be bounded, Borel measurable. As announced, we introduce the following integro-PDE corresponding to a regularized version of~\eqref{eq:PDE}
\begin{equation} 
\label{eq:PDEReg}
\left \{
\begin{array}{l}
\partial_t \gamma_t = L^{\ast}_t \gamma_t + \gamma_t \Lambda(t,x,K_{\varepsilon} \ast \gamma_t,\nabla K_{\varepsilon} \ast \gamma_t) \\
\gamma_0 = {\bf u_0} \ .
\end{array}
\right .
\end{equation}
The concept of mild solution associated to this type of equation is clarified by the following definition. 
\begin{defi}
\label{def:MildSolRegPDE}
A Borel measure-valued function $ \gamma  : [0,T] \longrightarrow \shm_f(\R^d)$ will be called a {\bf{mild solution}} of \eqref{eq:PDEReg} if it satisfies, for all $\varphi \in \shc_0^{\infty}(\R^d)$, $t \in [0,T]$,
\begin{eqnarray}
\label{eq:MildSolReg}
\int_{\R^d} \varphi(x) \gamma(t,dx) & = & \int_{\R^d} \varphi(x) \int_{\R^d} {\bf u_0}(dx_0) P(0,x_0,t,dx)  \nonumber \\
&& + \; \int_{[0,t] \times \R^d} \Big( \int_{\R^d} \varphi(x) P(s,x_0,t,dx) \Big) \Lambda \big(s,x_0,(K_{\varepsilon} \ast \gamma(s,\cdot))(x_0), (\nabla K_{\varepsilon} \ast \gamma(s,\cdot))(x_0) \big) \gamma(s,dx_0) ds  \ . \nonumber \\
\end{eqnarray}
% IL FAUDRAIT DIRE UN MOT SUR LE FAIT QUE LES INTEGRALES SONT BIEBN DEFINIES. CECI EST VRAI CAR LES FONCIONS SONT BORNEES
\end{defi}
 Similarly  as Theorem \ref{thm:FKForm}, we straightforwardly derive the following equivalence result. 
\begin{prop}
\label{prop:FKRegEquiv}
Suppose that Assumption \ref{ass:main0}
% and \eqref{eq:HypK}
 is fulfilled. We indicate by $Y$ the unique strong solution of \eqref{eq:SDE}
with prescribed $Y_0 \sim {\bf u_0}$. 
A Borel measure-valued function $ \gamma^{\varepsilon}  : [0,T] \longrightarrow \shm_f(\R^d)$  is a mild solution of \eqref{eq:PDEReg} if and only if, for all $\varphi \in \shc_b(\R^d)$, $t \in [0,T]$,
\begin{equation}
\label{eq:FKRegul}
\int_{\R^d} \varphi(x) \gamma^{\varepsilon}_t(dx) = \E \Big[ \varphi(Y_t) \exp \Big (\int_0^t \Lambda \big(s,Y_s,(K_{\varepsilon} \ast \gamma_s^{\varepsilon})(Y_s), (\nabla K_{\varepsilon} \ast \gamma_s^{\varepsilon})(Y_s) \big) \; ds\Big ) \Big] \ .
\end{equation}
%\item Setting $u^{\varepsilon} := K_{\varepsilon} \ast \gamma^{\varepsilon}$, item 1. above implies that $u^{\varepsilon}$ satisfies \eqref{eq:RFK}.
%\end{enumerate}
\end{prop}
\begin{proof}
The proof follows the same lines as the proof of Theorem \ref{thm:FKForm}. First assume that $\gamma^{\varepsilon}$ satisfies \eqref{eq:FKRegul}, we can show that $\gamma^{\varepsilon}$ is a mild solution \eqref{eq:PDEReg} by imitating the first step of the proof of Proposition \ref{prop:MuPDE}. 
Secondly, the converse is proved by applying Proposition \ref{prop:MuPDE} with $\tilde{\Lambda}(t,x) := \Lambda(t,x,(K_{\varepsilon} \ast \gamma^{\varepsilon}_t)(x), (\nabla K_{\varepsilon} \ast \gamma^{\varepsilon}_t)(x))$ and $\mu(t,dx) := \gamma^{\varepsilon}_t(dx)$.
\end{proof}

Let us now prove existence and uniqueness of a mild solution for
 the integro-PDE \eqref{eq:PDEReg}. To this end, we proceed 
similarly as for the proof of Theorem \ref{prop:UniTrFun} using
%in a similar way 
%as for the proof of
 Lemma \ref{lem:LocalSol}. Let  $\tau > 0$ be a constant supposed to be fixed for the moment and let us fix $\varepsilon > 0$, $r \in [0,T-\tau]$. $\shb([r,r+\tau],\shm_f(\R^d))$ denotes the space of
 bounded,  measure-valued maps, where  $\shm_f(\R^d)$ is equipped with the total variation norm $\Vert \cdot \Vert_{TV}$.
Given $M > 0$, we denote by  $B(0,M)$ denotes the centered ball in $(\shm_f(\R^d),\Vert \cdot \Vert_{TV} )$ with radius $M$ and by
$\shb([r,r+\tau],B(0,M))$, the closed subset of $\shb([r,r+\tau],\shm_f(\R^d))$
of $B(0,M)$-valued maps defined on $[r,r+\tau]$.
We introduce the measure-valued application
$
\Pi_{\varepsilon} : \beta 
\in \shb([r,r+\tau],\shm_f(\R^d))
 \longrightarrow \Pi_{\varepsilon}(\beta),
$ 
defined by
\begin{eqnarray}
\label{eq:DefPiReg}
\Pi_{\varepsilon}(\beta)(t,dx) & = & \int_r^t \int_{\R^d} P(s,x_0,t,dx) \Lambda \big(s,x_0,(K_{\varepsilon} \ast \widehat{\beta}(s,\cdot))(x_0), (\nabla K_{\varepsilon} \ast \widehat{\beta}(s,\cdot))(x_0) \big) \; \widehat{\beta}(s,dx_0) ds \ . \nonumber \\
\widehat{\beta}(s,\cdot) & = & \beta(s,\cdot) + \widehat{u_0}(s,\cdot) \ ,
\end{eqnarray}
where the function $\widehat{u_0}$, defined on $[r,r+\tau] \times \shm_f(\R^d)$, is given by
\begin{equation}
\label{eq:DEfu0Reg}
\widehat{u_0}(r,\pi)(t,dx) := \int_{\R^d} p(r,x_0,t,dx) \pi(dx_0), \; t \in [r,r+\tau], \; \pi \in \shm_f(\R^d) \ ,
\end{equation}
similarly to \eqref{eq:DEfu0}. In the sequel, the dependence of $\widehat{u}_0$ w.r.t. $r,\pi$ will be omitted when it is self-explanatory.
%where, similarly to \eqref{eq:DEfu0} (and \eqref{eq:Defu00}), $\widehat{u_0}$ will depend of some $\eta \in \shm_f(\R^d)$ to be specified from time to time.
%
\begin{lem}
\label{lem:ExistRFK}
Assume the validity of items 4. and 5. of Assumption \ref{ass:mainP3}.
%and of \eqref{eq:HypK}. 
Let $\pi \in B(0,M)$. \\
Let us fix $\varepsilon > 0$ and $M > 0$ such that $M \geq  \Vert \pi \Vert_{TV}$. Then, there is $\tau > 0$ only depending on $M_{\Lambda}$ such that for any $r \in [0,T-\tau]$, $\Pi_{\varepsilon}$ admits a unique fixed-point in $\shb([r,r+\tau],B(0,M))$. 
%where $B(0,M)$ denotes here the centered ball in $(\shm_f(\R^d),\Vert \cdot \Vert_{TV} )$ with radius $M$.
\end{lem}
\begin{proof}
Let us define $\tau := \frac{1}{2M_{\Lambda}}$.
For every $\lambda \ge 0$, 
$\shb([r,r+\tau],\shm_f(\R^d))$ will be equipped with one of the equivalent
norms 
%we define an equivalent norm on $\shb([r,r+\tau],\shm_f(\R^d))$ by
\begin{eqnarray} 
\label{eq:equivnorm}
\Vert \beta \Vert_{TV,\lambda} := \sup_{t \in [r,r+\tau]} e^{-\lambda t} \Vert \beta(t,\cdot) \Vert_{TV} \ .
\end{eqnarray}
Recalling \eqref{eq:DefPiReg}, where $\widehat{u_0}$ is defined by \eqref{eq:DEfu0Reg}, it follows that for all $\beta \in \shb([r,r+\tau],B(0,M))$, $t \in [r,r+\tau]$,
\begin{eqnarray}
\label{eq:Stab}
\Vert \Pi_{\varepsilon}(\beta)(t,\cdot) \Vert_{TV} \leq M_{\Lambda} \int_r^t \Vert \beta(s,\cdot) \Vert_{TV} ds + M_{\Lambda} M \tau \leq 2MM_{\Lambda} \tau \leq M\ ,
\end{eqnarray}
where for the latter inequality of \eqref{eq:Stab} we have used the definition of $\tau := \frac{1}{2M_{\Lambda}}$. We deduce that $\Pi(\shb([r,r+\tau],B(0,M))) \subset \shb([r,r+\tau],B(0,M))$. \\
Consider now $\beta^1, \beta^2 \in \shb([r,r+\tau],B(0,M))$.
For all $\lambda > 0$ we have
\begin{eqnarray}
\label{eq:Contract}
\Vert \Pi_{\varepsilon}(\beta^1(t,\cdot)) - \Pi_{\varepsilon}(\beta^2(t,\cdot)) \Vert_{TV} & \leq & 
%%OLD
% \int_r^t \Vert \beta^1(s,\cdot) - \beta^2(s,\cdot) \Vert_{TV} (L_{\Lambda} \Vert K_{\varepsilon}\Vert_{\infty} \Vert \beta^1(s,\cdot) \Vert_{TV} + M_{\Lambda}) ds \nonumber \\
M_\Lambda \int_r^t \Vert \beta^1(s,\cdot) - \beta^2(s,\cdot) \Vert_{TV} ds 
\nonumber \\
&& + \; L_{\Lambda} (\Vert K_{\varepsilon} + \Vert \nabla K_{\varepsilon} \Vert_{\infty})
\int_r^t \Vert \beta^1(s,\cdot) \Vert_{TV} \Vert \beta^1(s,\cdot) - \beta^2(s,\cdot) \Vert_{TV} ds \nonumber \\
% && + \; L_{\Lambda} \Vert \nabla K_{\varepsilon} \Vert_{\infty} \int_r^t \Vert \beta^1(s,\cdot) \Vert_{TV} \Vert \beta^1(s,\cdot) - \beta^2(s,\cdot) \Vert_{TV} ds \nonumber \\
&& + \; L_{\Lambda} (\Vert K_{\varepsilon} \Vert_{\infty} + \Vert \nabla K_{\varepsilon} \Vert_{\infty} )  \int_r^t \Vert \widehat{u_0}(s,\cdot) \Vert_{TV} \Vert \beta^1(s,\cdot) - \beta^2(s,\cdot) \Vert_{TV} ds \nonumber \\
& \leq & C_{\varepsilon,T} \int_0^t \Vert \beta^1(s,\cdot) - \beta^2(s,\cdot) \Vert_{TV} ds \nonumber \\
& \leq & C_{\varepsilon,T} \int_0^t e^{s \lambda} \; \Vert \beta^1 - \beta^2 \Vert_{TV,\lambda} \; ds \nonumber \\ 
& = & C_{\varepsilon,T} \Vert \beta^1 - \beta^2 \Vert_{TV,\lambda} \; \frac{e^{\lambda t} - 1}{\lambda} \ , 
\end{eqnarray}
with $C_{\varepsilon,T} := 2L_{\Lambda}M ( \Vert K_{\varepsilon} \Vert_{\infty}+ \Vert \nabla K_{\varepsilon} \Vert_{\infty}) + M_{\Lambda}$.
%and  $M_{K_{\varepsilon}}$ (resp. $M_{\nabla K_{\varepsilon}}$) is the essential supremum of $ \vert K_{\varepsilon} \vert$ (resp. $\vert \nabla K_{\varepsilon} \vert$).
 It follows
\begin{eqnarray}
\label{eq:Contract2}
\Vert \Pi_{\varepsilon}(\beta^1) - \Pi_{\varepsilon}(\beta^2) \Vert_{TV,\lambda} & = & \sup_{t \in [r,r+\tau]} e^{-\lambda t} \Vert \Pi(\beta^1)(t,\cdot) - \Pi(\beta^2)(t,\cdot) \Vert_{TV} \nonumber \\
& \leq &  C_{\varepsilon,T} \Vert \beta^1 - \beta^2 \Vert_{TV,\lambda} \; \sup_{t \geq 0} \Big(\frac{1 - e^{-\lambda t}}{\lambda} \Big) \nonumber \\
& \leq & \frac{C_{\varepsilon,T}}{\lambda} \Vert \beta^1 - \beta^2 \Vert_{TV,\lambda} \ . 
\end{eqnarray}
Hence, taking $ \lambda > C_{\varepsilon,T}$, $\Pi_{\varepsilon}$ is a contraction on $\shb([r,r+\tau],B(0,M))$. 

Since $\shb \big([r,r+\tau],(\shm_f(\R^d),\Vert \cdot \Vert_{TV,\lambda}) \big)$ is a Banach space whose $\shb([r,r+\tau],B(0,M))$ is a closed subset, the proof ends by a simple application of Banach fixed-point theorem.
\end{proof}
The next step is to show how the  proposition above, with the help of Lemma \ref{lem:RecolSol}, permits us to construct a mild solution of \eqref{eq:PDEReg}. The reasoning is similar to the one explained in the proof of Theorem \ref{prop:UniTrFun}. Indeed, without restriction of generality, we can suppose there exists $N \in \N^{\star}$ such that $T = N \tau$. Then, for all $k = 0,\cdots,N-1$, Lemma \ref{lem:ExistRFK} applied on each interval $[k \tau, (k+1)\tau]$ (with $r = k\tau$, $ \pi = \beta_{\varepsilon}^{k-1}(k \tau,\cdot)$ for $k \geq 1$ and $\pi = u_0$ for $k = 0$) gives existence of a family of measure-valued maps $(\beta_{\varepsilon}^k : [k\tau,(k+1)\tau] \rightarrow \shm_f(\R^d), k=0,\cdots,N-1)$
%and  $(\widehat{u_0}^k : [k\tau,k\tau +1] \rightarrow \shm_f(\R^d), k = 1,\cdots,N)$
defined by
\begin{eqnarray}
\label{eq:PiRegRecol}
\beta_{\varepsilon}^k(t,dx) & = & \int_{k \tau}^t \int_{\R^d} P(s,x_0,t,dx) \Lambda \big(s,x_0,(K_{\varepsilon} \ast \widehat{\beta}_{\varepsilon}^k(s,\cdot))(x_0), (\nabla K_{\varepsilon} \ast \widehat{\beta}_{\varepsilon}^k(s,\cdot))(x_0) \big) \; \widehat{\beta}_{\varepsilon}^k(s,dx_0) ds \ . \nonumber \\
\widehat{\beta_{\varepsilon}^k}(s,\cdot) & = & \beta_{\varepsilon}^k(s,\cdot) + \widehat{u_0}^k(s,\cdot) \ ,
\end{eqnarray}
where for $k = 0$, $t \in [0,\tau]$,
\begin{eqnarray}
\widehat{u}_0^0(t,dx) = \int_{\R^d} P(0,x_0,t,dx) {\bf u_0}(dx_0) \ , \quad \textrm{by } \eqref{eq:DEfu0Reg} \textrm{ with } \pi = u_0 \ ,
\end{eqnarray}
and for all $k \in \{1,\cdots,N \}$, $t \in [k\tau,(k+1)\tau]$,
\begin{eqnarray}
\widehat{u_0}^k(t,dx) & := & \widehat{u_0}(\beta_{\varepsilon}^{k-1})(t,dx) \nonumber \\
& = & \int_{\R^d} P(k\tau,x_0,t,dx) \beta_{\varepsilon}^{k-1}(k\tau,dx_0) \ , \quad \textrm{by } \eqref{eq:DEfu0Reg} \textrm{ with } \pi = \beta_{\varepsilon}^{k-1}(k \tau,\cdot) \ .
\end{eqnarray}
We now consider the following measure-valued maps $\widehat{U_0} : [0,T] \rightarrow \shm_f(\R^d)$ and $\beta_{\varepsilon} : [0,T] \rightarrow \shm_f(\R^d)$ defined by their restrictions on each interval $[k\tau,(k+1)\tau], k=0,\cdots,N-1$ such that
\begin{eqnarray}
\label{eq:DefsSolRecol}
\widehat{U_0}(t,x) := \widehat{u_0}^k(t,x) \quad \textrm{and} \quad \beta_{\varepsilon}(t,x) := \beta_{\varepsilon}^k(t,x) \quad \textrm{ for } (t,x) \in [k\tau,(k+1)\tau] \times \R^d \ ,
\end{eqnarray} 
and we finally define $\gamma^{\varepsilon} : [0,T] \rightarrow \shm_f(\R^d)$ by
\begin{equation}
\label{eq:GammaEps}
\gamma^{\varepsilon} := \widehat{U_0} + \beta_{\varepsilon} \textrm{ on } [0,T] \times \R^d .
\end{equation}
To ensure that $\gamma^{\varepsilon}$ is indeed a mild solution on $[0,T] \times \R^d$ (in the sense of Definition \ref{def:MildSolRegPDE}) of the 
integro-PDE \eqref{eq:PDEReg}, it is enough to apply Lemma \ref{lem:RecolSol} with $\tau := \frac{1}{2M_{\Lambda}}$, $\mu(t,dx) := \gamma^{\varepsilon}(t,dx)$ and $(\alpha_k := k \tau)_{k = 0,\cdots,N}$. 

Previous   discussion leads us to the following proposition.
\begin{prop} \label{C54}
Suppose the validity of Assumption \ref{ass:main0} and items 4. and 5. of Assumption \ref{ass:mainP3}. %Suppose also that \eqref{eq:HypK} is fulfilled.
Let us fix $\varepsilon > 0$ and let $\gamma^{\varepsilon}$ denote the map defined by \eqref{eq:GammaEps}. 
%For a given sequence of smooth, bounded and integrable mollifiers $(K_{\varepsilon})_{\varepsilon > 0}$, 
%For any $\varepsilon > 0$, 
The following statements hold.
\begin{enumerate}
\item 
$\gamma^{\varepsilon}$ 
%(given by \eqref{eq:GammaEps})
 is the unique mild solution of the integro-PDE \eqref{eq:PDEReg}, see Definition \ref{def:MildSolRegPDE}. 
%see \eqref{eq:GammaEps}.
\item $\gamma^{\varepsilon}$ is the unique solution to the regularized Feynman-Kac equation~\eqref{eq:FKRegul}.
%The function $u^{\varepsilon}(t,x) := (K_{\varepsilon} \ast \gamma^{\varepsilon}_t)(x)$ satisfies the regularized Feynman-Kac type equation \eqref{eq:RFK}.
\end{enumerate}
\end{prop}
\begin{proof}
%Let us fix $\varepsilon > 0$. Deja fixe.
 The existence of a mild  solution $\gamma^{\varepsilon}$ of~\eqref{eq:PDEReg}  has already been proved through the discussion just above.
 It remains to justify uniqueness. Consider $\gamma^{\varepsilon,1}$, $\gamma^{\varepsilon,2}$ be two mild solutions of \eqref{eq:MildSolReg}. Then, with similar computations as the ones leading to inequality \eqref{eq:Contract2}, there exists a constant $ \mathfrak{C} := \mathfrak{C}(M_{\Lambda}, L_{\Lambda}, \Vert K_{\varepsilon} \Vert_{\infty}, \Vert \nabla K_{\varepsilon} \Vert_{\infty}) > 0$ %($M_{K_{\varepsilon}}$ and $M_{\nabla K_{\varepsilon}}$ denoting respectively the essential supremum of $\vert K_{\varepsilon} \vert$ and $\vert \nabla K_{\varepsilon} \vert$)
such that
\begin{eqnarray}
\Vert \gamma^{\varepsilon,1} - \gamma^{\varepsilon,2} \Vert_{TV,\lambda} \leq \frac{\mathfrak{C}}{\lambda} \Vert \gamma^{\varepsilon,1} - \gamma^{\varepsilon,2} \Vert_{TV,\lambda} \ ,
\end{eqnarray}
for all $\lambda > 0$ and where we recall that $\Vert \cdot \Vert_{TV,\lambda}$ has been defined by \eqref{eq:equivnorm}. Taking $\lambda > \mathfrak{C}$, uniqueness follows. This shows item 1.
Item 2. follows then by Proposition \ref{prop:FKRegEquiv}.
%The second item follows from item 2. of Remark \ref{prop:FKRegEquiv}.
\end{proof}

%The Theorem below, establishes the convergence of the function  $u^{\varepsilon}$ defined by $u^\varepsilon(t,\cdot):=K_{\varepsilon}\ast \gamma^\varepsilon_t$. 
%defined in Proposition~\ref{C54} 2. 
%  of \eqref{eq:RFK} 
%to the solution   $u$  of \eqref{eq:PDE}, see Corollary \ref{CFK}, 
%when $\varepsilon$ goes to $0$.
% constitutes the main result of this section.
The  theorem below states the convergence of the solution of the 
regularized Feynman-Kac 
equation~\eqref{eq:FKRegul} to the solution to the
 Feynman-Kac equation~\eqref{eq:FKForm}. This is equivalent to 
 the convergence of the solution of the regularized PDE~\eqref{eq:PDEReg} to 
solution of the target PDE~\eqref{eq:PDE}, when the regularization parameter $\varepsilon$ goes to zero. 
\begin{thm}
\label{thm:cvguu'} %Let us fix a sequence $(K_{\varepsilon})_{\varepsilon > 0}$ of smooth, bounded and integrable mollifiers, weakly convergent to the Dirac measure at $0$.
Suppose the validity of Assumption \ref{ass:mainP3}. 
%Suppose also that \eqref{eq:HypK} is fulfilled. 
For any $\varepsilon>0$, consider the real valued function $u^{\varepsilon}$ such that for any $t \in [0,T]$, 
\begin{equation}
\label{eq:ue}
u^\varepsilon(t,\cdot) := K_{\varepsilon}\ast \gamma^\varepsilon_t\ ,
\end{equation}
 where $\gamma^\varepsilon$ is the unique  solution 
of \eqref{eq:FKRegul} 
(or equivalently the unique mild solution of
\eqref{eq:PDEReg}). Then $u^\varepsilon$ converges to 
 $u$,  the unique solution of \eqref{eq:FKForm}
(or equivalently the unique mild solution of
\eqref{eq:PDE}). More precisely we have
\begin{eqnarray}
\label{eq:Cvguu'}
\Vert u^{\varepsilon}(t,\cdot) - u(t,\cdot) \Vert_1 + \Vert \nabla u^{\varepsilon}(t,\cdot) - \nabla u(t,\cdot) \Vert_1\xrightarrow[\varepsilon \rightarrow 0]{} 0 \ ,\quad\textrm{for any}\ t\in [0,T]\ .
\end{eqnarray}
\end{thm}
Before proving  Theorem  \ref{thm:cvguu'}, we state and prove a preliminary lemma.
\begin{lem} \label{L56}
Suppose the validity of Assumption~\ref{ass:mainP3}.
%and of \eqref{eq:HypK}. 
Consider $u$ the unique solution of~\eqref{eq:FKForm}, then for all  $t \in [0,T]$  
\begin{equation}
\label{eq:873}
u(t,x)  =  F_0(t,x) + \int_0^t \E \Big[p(s,Y_s,t,x) \Lambda(s,Y_s,u,\nabla u) e^{\int_0^s \Lambda(r,Y_r,u,\nabla u) dr} \Big] ds, \ dx \ {\rm a.e.}
\end{equation}
For a given $\varepsilon > 0$, consider  $u^\varepsilon$ defined by~\eqref{eq:ue}. Then for almost all $x \in \R^d$ and all $t \in [0,T]$,
\begin{equation}
\label{eq:874}
u^{\varepsilon}(t,x)  =  (K_{\varepsilon} \ast F_0(t,\cdot))(x) + \int_0^t \E \Big[(K_{\varepsilon} \ast p(s,Y_s,t,\cdot))(x) \Lambda(s,Y_s,u^{\varepsilon},\nabla u^{\varepsilon}) e^{\int_0^s \Lambda(r,Y_r,u^{\varepsilon},\nabla u^{\varepsilon}) dr} \Big] ds  \ , 
\end{equation}
where $F_0(t,x) := \int_{\R^d} p(0,x_0,t,x) u_0(x_0)dx_0$ for $t > 0, x \in \R^d$ and $F_0(0,\cdot) := u_0$.
We remark that  we have used again the  notation
\begin{eqnarray}
\label{eq:NotLamb}
\Lambda(s,\cdot,v,\nabla v) := \Lambda(s,\cdot,v(s,\cdot),\nabla v(s,\cdot)), \quad t \in [0,T] \ ,
\end{eqnarray}
 for $v \in L^1([0,T],W^{1,1}(\R^d))$.
\end{lem}
\begin{proof}
Equalities \eqref{eq:873} and \eqref{eq:874} are proved in a very similar way,
so we only  provide the proof of equation \eqref{eq:874}. \\
We observe that for all $t \in [0,T]$, $v \in L^1([0,T],W^{1,1}(\R^d))$, 
\begin{eqnarray}
\label{eq:421}
e^{\int_0^t \Lambda(s,Y_s,v(s,Y_s),\nabla v(s,Y_s)) ds} & = & 1 + \int_0^t \Lambda(r,Y_r,v(r,Y_r), \nabla v(r,Y_r))e^{\int_0^r \Lambda(s,Y_s,v(s,Y_s),\nabla v(s,Y_s)) ds} dr \ .
% \; d\P \textrm{-a.s.} 
\end{eqnarray}
Taking into account the notation introduced  in \eqref{eq:FKRegul}, 
\eqref{eq:ue} and \eqref{eq:NotLamb}, the identity \eqref{eq:421} above implies for almost all $x \in \R^d$,
\begin{eqnarray}
u^{\varepsilon}(t,x) & = & \E \Big[K_{\varepsilon}(x-Y_t) \, \exp \left \{\int_0^t\Lambda \Big (s,Y_s,u^{\varepsilon},\nabla u^{\varepsilon}\Big )ds\right \} \Big] \nonumber \\
& = & \E \Big[K_{\varepsilon}(x-Y_t) \Big] + \int_0^t \E \Big[K_{\varepsilon}(x-Y_t) \Lambda(r,Y_r,u^{\varepsilon}, \nabla u^{\varepsilon})e^{\int_0^r \Lambda(s,Y_s,u^{\varepsilon},\nabla u^{\varepsilon}) ds} \Big] dr \nonumber \\
& = & \int_{\R^d} K_{\varepsilon}(x-y) \int_{\R^d} p(0,x_0,t,y)u_0(x_0)dx_0 \; dy + \nonumber \\
&& \int_0^t \E \Big[ \E\Big[ K_{\varepsilon}(x-Y_t) \Big \vert Y_r \Big] \Lambda(r,Y_r,u^{\varepsilon}, \nabla u^{\varepsilon})e^{\int_0^r \Lambda(s,Y_s,u^{\varepsilon},\nabla u^{\varepsilon}) ds} \Big] dr \nonumber \\
& = & (K_{\varepsilon} \ast F_0)(t,\cdot)(x) + \nonumber \\
&& \int_0^t \E \Big[ \Big( \int_{\R^d} K_{\varepsilon}(x-y) p(r,Y_r,t,y) dy \Big)\; \Lambda(r,Y_r,u^{\varepsilon}, \nabla u^{\varepsilon})e^{\int_0^r \Lambda(s,Y_s,u^{\varepsilon},\nabla u^{\varepsilon}) ds} \Big] dr \nonumber \\
& = & (K_{\varepsilon} \ast F_0)(t,\cdot)(x)  + \nonumber \\
&& \int_0^t \E \Big[ (K_{\varepsilon} \ast p(r,Y_r,t,\cdot))(x) \; \Lambda(r,Y_r,u^{\varepsilon}, \nabla u^{\varepsilon})e^{\int_0^r \Lambda(s,Y_s,u^{\varepsilon},\nabla u^{\varepsilon}) ds} \Big] dr \ . 
\end{eqnarray}
This ends the proof.
\end{proof}

\begin{proof}[Proof of Theorem \ref{thm:cvguu'}]
In this proof, $C$ denotes a real constant that may change from line to line, only depending on $M_{\Lambda}$, $L_{\Lambda}$, $C_u$ and $\Vert u_0 \Vert_{\infty}$, where we recall that the constant $C_u$ only depends on $\Phi,g$ and 
come from inequality \eqref{eq:1011}. 

 We first observe that for $t = 0$, the convergence of $u^{\varepsilon}(0,\cdot)$ (resp. $\nabla u^{\varepsilon}(0,\cdot)$) to $u(0,\cdot)$ (resp. $\nabla u(0,\cdot)$) in $L^1(\R^d)$-norm when $\varepsilon$ goes to $0$ is clear. Let us fix $t \in (0,T]$. \\
By Lemma \ref{L56}, for almost all $x \in \R^d$, we have the  decomposition
\begin{eqnarray}
\label{eq:877}
u^{\varepsilon}(t,x) - u(t,x) & = & (K_{\varepsilon} \ast F_0(t,\cdot))(x) - F_0(t,x) + \nonumber \\
&& \int_0^t \E \Big[\Big\{ (K_{\varepsilon} \ast p(s,Y_s,t,\cdot))(x) - p(s,Y_s,t,x) \Big \}  \Lambda(s,Y_s,u^{\varepsilon},\nabla u^{\varepsilon}) e^{\int_0^s \Lambda(r,Y_r,u^{\varepsilon},\nabla u^{\varepsilon}) dr} \Big] ds + \nonumber \\
&& \int_0^t \E \Big[ p(s,Y_s,t,x)  \Big \{ \Lambda(s,Y_s,u^{\varepsilon},\nabla u^{\varepsilon}) e^{\int_0^s \Lambda(r,Y_r,u^{\varepsilon},\nabla u^{\varepsilon}) dr} - \Lambda(s,Y_s,u,\nabla u) e^{\int_0^s \Lambda(r,Y_r,u,\nabla u) dr}\Big \}\Big] ds \ . \nonumber \\
\end{eqnarray}
By integrating the absolute value of both sides of \eqref{eq:877} w.r.t. $dx$, it follows there exists a constant $C > 0$ such that
\begin{eqnarray}
\label{eq:878}
\Vert u^{\varepsilon}(t,\cdot) - u(t,\cdot) \Vert_1 & \leq & C \Big\{ \Vert K_{\varepsilon} \ast F_0 - F_0 \Vert_1 + \int_0^t \E \Big[\Vert K_{\varepsilon} \ast p(s,Y_s,t,\cdot) - p(s,Y_s,t,\cdot) \Vert_1 \Big] ds \nonumber \\
&& + \; \int_0^t \E \Big[ \vert u^{\varepsilon}(s,Y_s) - u(s,Y_s) \vert + \vert \nabla u^{\varepsilon}(s,Y_s) - \nabla u(s,Y_s) \vert   \Big] ds \nonumber \\
&& + \; \int_0^t \int_0^s \E \Big[ \vert u^{\varepsilon}(r,Y_r) - u(r,Y_r) \vert + \vert \nabla u^{\varepsilon}(r,Y_r) - \nabla u(r,Y_r) \vert   \Big] dr \; ds \Big \} \ . 
\end{eqnarray}
Moreover, by \eqref{eq:lawY1},
\begin{equation} \label{UDensity} 
p_s(x) = \int_{\R^d} p(0,x_0,s,x) u_0(x_0)  dx_0,
\end{equation}
 is  the law density of $Y_s$, 
by inequality \eqref{eq:majordens} of Lemma \ref{lem:transfun} we get
%to bound the marginal density $p_s$ of $Y_s$.
\begin{eqnarray}
\label{eq:BoundU}
\E \Big[ \vert u^{\varepsilon}(s,Y_s) - u(s,Y_s) \vert \Big] & = & \int_{\R^d} \vert u^{\varepsilon}(s,x) - u(s,x) \vert p_s(x)dx \nonumber \\
& \leq & C_u \Vert u_0 \Vert_{\infty} \int_{\R^d} \vert u^{\varepsilon}(s,x) - u(s,x) \vert dx \nonumber \\
& = & C_u \Vert u^{\varepsilon}(s,\cdot) - u(s,\cdot) \Vert_1 \ , \quad s \in [0,T] \ ,
\end{eqnarray}
and 
\begin{eqnarray}
\label{eq:BoundGrad}
\E \Big[ \vert \nabla u^{\varepsilon}(s,Y_s) - \nabla u(s,Y_s) \vert \Big] & = & \int_{\R^d} \vert \nabla u^{\varepsilon}(s,x) - \nabla u(s,x) \vert p_s(x)dx \nonumber \\
& \leq & C_u  \Vert u_0 \Vert_{\infty} \int_{\R^d} \vert \nabla u^{\varepsilon}(s,x) - \nabla u(s,x) \vert dx \nonumber \\
& = & C_u \Vert \nabla u^{\varepsilon}(s,\cdot) - \nabla u(s,\cdot) \Vert_1 \ , \quad s \in [0,T] \ .
\end{eqnarray}
 Injecting \eqref{eq:BoundU} and \eqref{eq:BoundGrad} into the r.h.s. of \eqref{eq:878}, it comes
\begin{eqnarray}
\label{eq:880}
\Vert u^{\varepsilon}(t,\cdot) - u(t,\cdot) \Vert_1 & \leq & C \Big\{ \Vert K_{\varepsilon} \ast F_0 - F_0 \Vert_1 + \int_0^t \E \Big[\Vert K_{\varepsilon} \ast p(s,Y_s,t,\cdot) - p(s,Y_s,t,\cdot) \Vert_1 \Big] ds \nonumber \\
&& + \; \int_0^t \Vert u^{\varepsilon}(s,\cdot) - u(s,\cdot) \Vert_1 + \Vert \nabla u^{\varepsilon}(s,\cdot) - \nabla u(s,\cdot) \Vert_1 ds \Big \}\ .
\end{eqnarray}
Now, we need to bound $\Vert \nabla u^{\varepsilon}(t,\cdot) - \nabla u(t,\cdot) \Vert_1$. To this end, we can remark that for almost all $x \in \R^d$,
\begin{eqnarray}
\label{eq:881}
\nabla u(t,x) = \nabla F_0(t,x) + \int_0^t \E\Big[ \nabla_x p(s,Y_s,t,x) \Lambda(s,Y_s,u,\nabla u) e^{\int_0^s \Lambda(r,Y_r,u,\nabla u) dr} \Big] ds \ ,
\end{eqnarray}
and
\begin{eqnarray}
\label{eq:882}
\nabla u^{\varepsilon}(t,x) & = & (K_{\varepsilon} \ast \nabla F_0(t,\cdot))(x) \nonumber \\
&& \; + \int_0^t \E \Big[(K_{\varepsilon} \ast \nabla_x p(s,Y_s,t,\cdot))(x) \Lambda(s,Y_s,u^{\varepsilon},\nabla u^{\varepsilon}) e^{\int_0^s \Lambda(r,Y_r,u^{\varepsilon},\nabla u^{\varepsilon}) dr} \Big] ds \ . 
\end{eqnarray}
These equalities follow by computing the derivative of $u(t,\cdot)$ and $u^{\varepsilon}(t,\cdot)$ in the sense of distributions. \\
Taking into account \eqref{eq:881} and \eqref{eq:882}, it is easy to see that very similar arguments as those invoked above to prove \eqref{eq:880}, lead to
\begin{eqnarray}
\label{eq:nablauu'}
\Vert \nabla u^{\varepsilon}(t,\cdot) - \nabla u(t,\cdot) \Vert_1 & \leq & C \Big\{ \Vert K_{\varepsilon} \ast \nabla F_0(t,\cdot) - \nabla F_0(t,\cdot) \Vert_1 + \int_0^t \E \Big[\Vert K_{\varepsilon} \ast \nabla_x p(s,Y_s,t,\cdot) - \nabla_x p(s,Y_s,t,\cdot) \Vert_1 \Big] ds \nonumber \\
&& + \; \int_0^t \Vert u^{\varepsilon}(s,\cdot) - u(s,\cdot) \Vert_1 + \Vert \nabla u^{\varepsilon}(s,\cdot) - \nabla u(s,\cdot) \Vert_1 ds \Big \}\ .
\end{eqnarray}
Gathering \eqref{eq:880} together with \eqref{eq:nablauu'} yields
\begin{eqnarray}
\label{eq:884}
\Vert u^{\varepsilon}(t,\cdot) - u(t,\cdot) \Vert_1 + \Vert \nabla u^{\varepsilon}(t,\cdot) - \nabla u(t,\cdot) \Vert_1 & \leq & C \;\Big\{ \Vert K_{\varepsilon} \ast F_0(t,\cdot) - F_0(t,\cdot) \Vert_1 + \Vert K_{\varepsilon} \ast \nabla F_0(t,\cdot) - \nabla F_0(t,\cdot) \Vert_1 +\nonumber \\
&& \int_0^t \E \Big[\Vert K_{\varepsilon} \ast p(s,Y_s,t,\cdot) - p(s,Y_s,t,\cdot) \Vert_1 \Big] ds + \nonumber \\
&& \int_0^t \E \Big[\Vert K_{\varepsilon} \ast \nabla_x p(s,Y_s,t,\cdot) - \nabla_x p(s,Y_s,t,\cdot) \Vert_1 \Big] ds \Big \} \nonumber \\
&& + \; \int_0^t \Vert u^{\varepsilon}(s,\cdot) - u(s,\cdot) \Vert_1 + \Vert \nabla u^{\varepsilon}(s,\cdot) - \nabla u(s,\cdot) \Vert_1 ds \ . \nonumber \\ 
\end{eqnarray}
Applying Gronwall's lemma to the real-valued function
$$
t \mapsto \Vert u^{\varepsilon}(t,\cdot) - u(t,\cdot) \Vert_1 + \Vert \nabla u^{\varepsilon}(t,\cdot) - \nabla u(t,\cdot) \Vert_1 \ ,
$$
we obtain
\begin{eqnarray}
\label{eq:531}
\Vert u^{\varepsilon}(t,\cdot) - u(t,\cdot) \Vert_1 + \Vert \nabla u^{\varepsilon}(t,\cdot) - \nabla u(t,\cdot) \Vert_1 & \leq & C e^{CT} \; \Big\{ \Vert K_{\varepsilon} \ast F_0(t,\cdot) - F_0(t,\cdot) \Vert_1 + \Vert K_{\varepsilon} \ast \nabla F_0(t,\cdot) - \nabla F_0(t,\cdot) \Vert_1 +\nonumber \\
&& \int_0^t \E \Big[\Vert K_{\varepsilon} \ast p(s,Y_s,t,\cdot) - p(s,Y_s,t,\cdot) \Vert_1 \Big] ds + \nonumber \\
&& \int_0^t \E \Big[\Vert K_{\varepsilon} \ast \nabla_x p(s,Y_s,t,\cdot) - \nabla_x p(s,Y_s,t,\cdot) \Vert_1 \Big] ds \Big \}.
\end{eqnarray}
Since $F_0(t,\cdot)$, $\nabla F_0(t,\cdot)$, $x \mapsto p(s,x_0,t,x)$ and $x \mapsto \nabla_x p(s,x_0,t,x)$ belong to $L^1(\R^d)$, classical properties of  convergence of the mollifiers give
\begin{eqnarray}
\Vert K_{\varepsilon} \ast F_0(t,\cdot) - F_0(t,\cdot) \Vert_1 \rightarrow 0, \quad \Vert K_{\varepsilon} \ast \nabla F_0(t,\cdot) - \nabla F_0(t,\cdot) \Vert_1 \rightarrow 0,
\end{eqnarray}
and
\begin{eqnarray}
\Vert K_{\varepsilon} \ast p(s,Y_s,t,\cdot) - p(s,Y_s,t,\cdot) \Vert_1 \rightarrow 0, \quad \Vert K_{\varepsilon} \ast \nabla_x p(s,Y_s,t,\cdot) - \nabla_x p(s,Y_s,t,\cdot) \Vert_1 \rightarrow 0,\; a.s. 
\end{eqnarray}
Moreover, by inequalities \eqref{eq:1011a} and \eqref{eq:1011} of Lemma \ref{lem:transfun},
% there exists a constant $C := C(C_u) > 0$ such that 
for $0 \leq s < t \leq T$,
%\begin{eqnarray}
%\Vert K_{\varepsilon} \ast F_0(t,\cdot) - F_0(t,\cdot) \Vert_1 + \Vert K_{\varepsilon} \ast \nabla F_0(t,\cdot) - \nabla F_0(t,\cdot) \Vert_1 \leq 2C\Vert u_0 \Vert_1 (1 + \frac{1}{\sqrt{t}}), \; d\P-a.s \ ,
%\end{eqnarray}
%and for $0 < s < t \leq T$, 
\begin{eqnarray}
\Vert K_{\varepsilon} \ast p(s,Y_s,t,\cdot) - p(s,Y_s,t,\cdot) \Vert_1 + \Vert K_{\varepsilon} \ast \nabla_x p(s,Y_s,t,\cdot) - \nabla_x p(s,Y_s,t,\cdot) \Vert_1 \leq 2C_u\left(1 + \frac{1}{\sqrt{t-s}}\right)\; a.s.  \nonumber \\
\end{eqnarray}
Lebesgue dominated convergence theorem then implies that the third and fourth terms in the r.h.s. of \eqref{eq:531} converge to $0$ when $\varepsilon$ goes to $0$. This ends the proof.
\end{proof}
\begin{prop} \label{C47}
We assume here that $K$ verifies~\eqref{EKappa}. 
%We assume here that $(K_{\varepsilon})_{\varepsilon > 0}$ is explicitly given by 
%\begin{equation}
%\label{eq:Keps}
%K_{\varepsilon}(x) := \frac{1}{\varepsilon^d} K(\frac{x}{\varepsilon}) \ ,
%\end{equation} 
%with $K \ge 0$ satisfying
%\begin{equation} \label{EKappa}
   %\int_{\R^d} K(x)\, dx = 1\ ,
%\quad 
   %\int_{\R^d} x\, K(x)\, dx = 0 \quad \textrm{and}\quad 
 %\kappa := \frac{1}{2}\int_{\R^d} \vert x \vert\, K(x)\, dx < \infty.
%\end{equation}
Let $u^{\varepsilon}$ be the real-valued function defined by \eqref{eq:ue}.
% (such that for any $t \in [0,T]$, $u^\varepsilon (t,\cdot):=K_{\varepsilon}\ast \gamma_t^\varepsilon$). 
Under Assumption \ref{ass:mainP3} and in the particular case where the function $(t,x,y,z)\mapsto \Lambda (t,x,y,z)$ does not depend on the $z$ variable corresponding to the gradient $\nabla u$, there exists a constant 
\begin{equation} \label{Cstar}
C := C(M_\Lambda, L_\Lambda, C_u, \Vert u_0 \Vert_\infty, \kappa) > 0,
\end{equation}
with $C_u$ denoting the constant  given by \eqref{eq:1011a} (only depending on $\Phi$, $g$) 
 such that the following holds.
 For all $t \in ]0,T]$,
\begin{equation}
\label{eq:Cvgu}
\Vert u^{\varepsilon}(t,\cdot) - u(t,\cdot) \Vert_1 \leq %\left \{
\begin{array}{ll}
\varepsilon C  \big( \frac{1}{\sqrt{t}} + 2 \sqrt{t} \big), 
%\varepsilon \kappa \; \displaystyle{ \max_{i=1,\cdots,d}\Vert \partial_i u_0 \Vert_1 }, & \textrm{ if } t = 0 \ ,
\end{array}
%\right .
\end{equation}

\end{prop}

\begin{proof} 
In the proof $C$ is a constant fulfilling \eqref{Cstar}.
The  arguments are the same as the ones used in the proof of Theorem \ref{thm:cvguu'} since in the present case, $\Lambda$ only depends on $(t,x,u)$ and not on $\nabla u$. In particular, we obtain for $t \in ]0,T]$,
\begin{eqnarray}
\Vert u^{\varepsilon}(t,\cdot) - u(t,\cdot) \Vert_1 \leq C e^{CT} \Vert K_{\varepsilon} \ast F_0(t,\cdot) - F_0(t,\cdot) \Vert_1 + \int_0^t \E \Big[\Vert K_{\varepsilon} \ast p(s,Y_s,t,\cdot) - p(s,Y_s,t,\cdot) \Vert_1 \Big] ds,
\end{eqnarray}
that corresponds to inequality \eqref{eq:531} in the proof above, without the term containing the gradient $\nabla u$. 
Invoking inequality \eqref{eq:1011a} of Lemma \ref{lem:transfun} and inequality \eqref{eq:lem:biasW1} of Lemma \ref{lem:bias} with $H = K$, and successively with $f = F_0(t, \cdot)$ and
$f = p(s,y,t,\cdot)$ for fixed $y \in \R^d$, imply that
%% OLD
% \begin{equation}
% \Vert K_{\varepsilon} \ast F_0(t,\cdot) - F_0(t,\cdot) \Vert_1 \leq \frac{\varepsilon  \kappa \mathfrak{C}}{\sqrt{t}} \ , \quad 0 < t \leq T,
% \end{equation}
% and
% \begin{equation}
% \Vert K_{\varepsilon} \ast p(s,Y_s,t,\cdot) - p(s,Y_s,t,\cdot) \Vert_1 \leq \frac{\varepsilon \kappa \mathfrak{C}}{\sqrt{t-s}} \ , \textrm{a.e}, \quad 0 \leq s < t < T \ ,
% \end{equation}
% with $\mathfrak{C} = \mathfrak{C}(C_u,c_u)$. 
% When $t = 0$, the result just follows from the use again of Lemma \ref{lem:bias} with $f= u_0$.
 \begin{equation}
 \Vert K_{\varepsilon} \ast F_0(t,\cdot) - F_0(t,\cdot) \Vert_1 \leq \frac{C \varepsilon}{\sqrt{t}} \ , \quad 0 < t \leq T,
 \end{equation}
 and
 \begin{equation}
 \Vert K_{\varepsilon} \ast p(s,Y_s,t,\cdot) - p(s,Y_s,t,\cdot) \Vert_1 \leq 
\frac{C \varepsilon}{\sqrt{t-s}} \ , \textrm{a.e.} \quad 0 \leq s < t \le T.
 \end{equation}
This concludes the proof of \eqref{eq:Cvgu}.
\end{proof}

\section{Particles system algorithm}
\label{SPartAlgo}

\setcounter{equation}{0}
To simplify notations in the rest of the paper, $f_t$ will denote $f(t)$ 
where $f : [0,T] \rightarrow E$ is an $E$-valued Borel function and 
$(E,d_E)$ a metric space. 

In previous sections, we have studied existence, uniqueness 
for a semilinear PDE of the form~\eqref{eq:PDE} and 
we have established a Feynman-Kac type representation
 for the corresponding solution $u$, see Theorem \ref{thm:FKForm}.
%, of the
% semilinear PDEs of the form~\eqref{eq:PDE}, see Theorem \ref{thm:FKForm}.
The regularized form of \eqref{eq:PDE} is the integro-PDE \eqref{eq:PDEReg}
for which we have established well-posedness in Proposition~\ref{C54} 1.
 In the sequel, we denote by $\gamma^\varepsilon$ the corresponding solution
and again by $u^{\varepsilon}(t,x) := (K_{\varepsilon} \ast \gamma^{\varepsilon}_t)(x)$,
see~\eqref{eq:ue}.
%in particular $u^\varepsilon$ verifies equation~\eqref{eq:RFK}, see Proposition~\ref{C54} 2.
We recall that $u^\varepsilon$  converges to $u$,  the unique solution of \eqref{eq:FKForm}
(or equivalently the unique mild solution of
\eqref{eq:PDE}),
 when the regularization parameter
 $\varepsilon$ vanishes to $0$, see Theorem \ref{thm:cvguu'}.
 In the present section, we propose a Monte Carlo approximation $u^{\varepsilon , N}$  of $u^{\varepsilon}$,  providing an original numerical approximation of 
the semilinear PDE~\eqref{eq:PDE}, 
when the regularization parameter $\varepsilon \rightarrow 0$
slowly enough, while the
 number of particles
 $N\rightarrow \infty$. 
%when the number of particles
% $N\rightarrow \infty$ 
%and the regularization parameter $\varepsilon \rightarrow 0$
%slowly enough.
% with a judicious relative rate. 
%Let us fix a filtered probability space $(\Omega, \shf,(\shf_t)_{t \geq 0},\P)$ and $
Let ${\bf u_0}$ be a Borel probability measure on $\shp(\R^d)$.
% and \ref{ass:RegSection}. 

\subsection{Convergence of the particle system}
\label {S51}
We suppose the validity of Assumption \ref{ass:mainP3}. 

For fixed  $N \in \N^{\star}$, let $(W^i)_{i = 1,\cdots,N}$ be a family of independent Brownian motions and $(Y^i_0)_{i=1,\cdots,N}$ be i.i.d. random variables distributed according to $u_0(x)dx$. For any $\varepsilon > 0$, we define the measure-valued functions  $(\gamma_t^{\varepsilon,N})_{t\in[0,T]}$ such that for any $t\in [0,T]$ 
\begin{equation}
\label{eq:XIiP3}
\left \{
\begin{array}{l}
\xi^{i}_t= \xi^{i}_0 +\int_0^t\Phi(s,\xi^{i}_s)dW^i_s+\int_0^tg(s,\xi^{i}_s)ds  \ , \quad \textrm{for}\ i=1,\cdots ,N\ ,\\
\xi^{i}_0 = Y^i_0 \quad \textrm{for}\ i=1,\cdots ,N\ ,\\
\gamma^{\varepsilon,N}_t = {\displaystyle \frac{1}{N}\sum_{i=1}^N  V_t\big (\xi^{i},(K_{\varepsilon} \ast \gamma^{\varepsilon,N})(\xi^{i}),(\nabla K_{\varepsilon} \ast \gamma^{\varepsilon,N})(\xi^{i}) \big ) } \delta_{\xi^{i}_t}, \ .
\end{array}
\right .
\end{equation}
where $(K_{\varepsilon})_{\varepsilon > 0})$ are 
 mollifiers fulfilling \eqref{eq:Keps} and \eqref{eq:HypK}. 
%such that for all $\varepsilon > 0$, $K_{\varepsilon} \in W^{1,1}(\R^d) \cap W^{1,\infty}(\R^d)$
 We recall that $V_t$ is given by \eqref{eq:VP3}.
The first line of \eqref{eq:XIiP3} is  a $d$-dimensional classical SDE whose strong existence and pathwise uniqueness are ensured by classical theorems for Lipschitz coefficients. Moreover  $\xi^i, i=1,\cdots,N$ are  i.i.d. 
In the following lemma, we prove by a fixed-point argument that the third line equation of \eqref{eq:XIiP3} has a unique solution. 
\begin{lem}
\label{lem:PtFixuSn}
We suppose the validity of Assumption \ref{ass:mainP3}.
Let us fix $\varepsilon > 0$ and $N \in \N^{\star}$. Consider the i.i.d. system $(\xi^i)_{i=1,\cdots ,N}$ of particles, 
solution of the two first equations of~\eqref{eq:XIiP3}. Then, there exists a unique function
 $\gamma^{\varepsilon,N} : [0,T] \rightarrow \shm_f(\R^d)$ such that for all $t \in [0,T]$, $\gamma_t^{\varepsilon , N}$ is solution of~\eqref{eq:XIiP3}. 
%\begin{eqnarray}
%\label{eq:gammaSNXi}
%\gamma^{\varepsilon,N}_t = {\displaystyle \frac{1}{N}\sum_{i=1}^N  V_t\big (\xi^{i},(K_{\varepsilon} \ast \gamma^{\varepsilon,N})(\xi^{i}),(\nabla K_{\varepsilon} \ast \gamma^{\varepsilon,N})(\xi^{i}) \big ) } \delta_{\xi^{i}_t} \ .
%\end{eqnarray}  
\end{lem}
\begin{proof}
%Let us fix $\varepsilon > 0$,  $N \in \N^{\star}$. 
The proof relies on a fixed-point argument applied to the map $T^{\varepsilon,N} : \shc([0,T],\shm_f(\R^d)) \rightarrow \shc([0,T],\shm_f(\R^d))$ given by
\begin{eqnarray}
\label{eq:TP3}
T^{\varepsilon,N}(\gamma)(t) : \gamma \mapsto \frac{1}{N}\sum_{i=1}^N  V_t\big (\xi^{i},(K_{\varepsilon} \ast \gamma)(\xi^{i}),(\nabla K_{\varepsilon} \ast \gamma)(\xi^{i}) \big ) \delta_{\xi^{i}_t} \ .
\end{eqnarray}
In the rest of the proof, the notation $T_t^{\varepsilon,N}(\gamma)$ will denote $T^{\varepsilon,N}(\gamma)(t)$. \\
In order to apply the Banach fixed-point theorem, we emphasize that $\shc \big([0,T],\shm_f(\R^d) \big)$ is equipped with one of  the equivalent norms 
%with the weak convergence topology and metrized by the total variation norm $\Vert \cdot \Vert_{var}$. In this case, it is well-known that $(\shm_f(\R^d), \Vert \cdot \Vert_{var})$ is a Banach space. Similarly as in the proof of Proposition \ref{prop:ExistRFK}, we equip $\shc([0,T],\shm_f(\R^d))$ with an equivalent norm
$\Vert \cdot \Vert_{TV,\lambda}$, $\lambda  \ge 0 $, defined by
\begin{eqnarray}
\Vert \gamma \Vert_{TV,\lambda} := \sup_{t \in [0,T]} e^{-\lambda t} \Vert \gamma(t,\cdot) \Vert_{TV} \ ,
\end{eqnarray}
and for which $\big( \shc([0,T],\shm_f(\R^d)),\Vert \cdot \Vert_{TV,\lambda} \big)$ is still complete. 

From now on, it remains to ensure that $T^{\varepsilon,N}$ is indeed a contraction
with respect $\Vert \gamma \Vert_{TV,\lambda}$ for some $\lambda$.
To simplify notations, we set for all $i \in \{1,\cdots,N \}$,
\begin{eqnarray}
T_t^{\varepsilon,N,i}(\gamma) := V_t\big (\xi^{i},(K_{\varepsilon} \ast \gamma)(\xi^{i}),(\nabla K_{\varepsilon} \ast \gamma)(\xi^{i}) \big ) \; , \quad (t,\gamma) \in [0,T] \times \shc \big([0,T],\shm_f(\R^d) \big) \ ,
\end{eqnarray}
to re-write \eqref{eq:TP3} in the  form
\begin{eqnarray}
T_t^{\varepsilon,N}(\gamma)  = \frac{1}{N} \sum_{i=1}^N T_t^{\varepsilon,N,i}(\gamma) \delta_{\xi^i_t} \; , \quad (t,\gamma) \in [0,T] \times \shc \big([0,T],\shm_f(\R^d) \big) \ .
\end{eqnarray}
Let $\lambda > 0$.
Consider now $\gamma^1, \gamma^2 \in \shc \big([0,T],\shm_f(\R^d) \big)$.
On the one hand, taking into account \eqref{eq:VP3} and \eqref{eq:LipV},  for all $t \in [0,T]$, $i \in \{1,\cdots,N\}$, we have
\begin{eqnarray}
\vert T_t^{\varepsilon,N,i}(\gamma^1) - T_t^{\varepsilon,N,i}(\gamma^2) \vert & \leq & L_{\Lambda}e^{TM_{\Lambda}} \int_0^t \Big( \vert (K_{\varepsilon} \ast \gamma^1)(\xi^i_s) - (K_{\varepsilon} \ast \gamma^2)(\xi^i_s) \vert \nonumber \\
&& \; + \; \vert (\nabla K_{\varepsilon} \ast \gamma^1)(\xi^i_s) - (\nabla K_{\varepsilon} \ast \gamma^2)(\xi^i_s) \Big) ds \, \nonumber \\
& \leq &  C \int_0^t \Vert \gamma^1_s - \gamma^2_s \Vert_{TV} ds \nonumber \\
& \leq & C \int_0^t e^{s \lambda} \; \Vert \gamma^1 - \gamma^2 \Vert_{TV,\lambda} \; ds \nonumber \\ 
& = & C \Vert \gamma^1 - \gamma^2 \Vert_{TV,\lambda} \; \frac{e^{\lambda t}-1}{\lambda} \ ,
\end{eqnarray}
with $C = C(T,\Vert K_{\varepsilon} \Vert_{\infty}, \Vert \nabla K_{\varepsilon} \Vert_{\infty},L_{\Lambda},M_{\Lambda})$.
It follows that 
\begin{eqnarray}
\Vert T^{\varepsilon,N}(\gamma^1) - T^{\varepsilon,N}(\gamma^2) \Vert_{TV,\lambda} & \leq & \frac{1}{N} \sum_{i=1}^N \Vert T^{\varepsilon,N,i}(\gamma^1) \delta_{\xi^i} - T^{\varepsilon,N,i}(\gamma^2) \delta_{\xi^i} \Vert_{TV,\lambda} \nonumber \\
& \leq & \frac{C}{\lambda} \Vert \gamma^1 - \gamma^2 \Vert_{TV,\lambda} \ .
\end{eqnarray}
By taking $\lambda > C$ and invoking Banach fixed-point theorem, we end the proof.
\end{proof}
After the preceding preliminary considerations, we can state and prove the main result of the section.
\begin{prop}
\label{prop:CvgParticle}
We suppose the validity of Assumption \ref{ass:mainP3}.
Assume that the kernel $K$ is verifying  
\eqref{eq:Kcond}. 
% \begin{equation}
% \label{eq:Kcond}
% K\in W^{1,1}(\R^d)\cap W^{1,\infty}(\R^d)\ ,\quad \int_{\R^d} \vert x \vert^{d+1} \; K(x) dx<\infty\ ,\quad \textrm{and}\quad \int_{\R^d}\vert x\vert^{d+1} \; \vert \nabla K (x)\vert dx<\infty\ .
% \end{equation}
%For $\varepsilon > 0$, we suppose that $K_{\varepsilon}$ is explicitly given by \eqref{eq:Keps} with $K$ satisfying \eqref{eq:Kcond}.
%\begin{equation}
%\label{eq:Keps1}
%K_{\varepsilon}(x) := \frac{1}{\varepsilon^d} K(\frac{x}{\varepsilon})\ , \quad x \in \R^d \ .
%\end{equation}
Let $u^{\varepsilon}$ be the real valued function defined by~\eqref{eq:ue}, and $u^{\varepsilon,N}$ such that for any $t \in [0,T]$, 
\begin{equation}
\label{eq:ueN}
u^{\varepsilon,N}(t,\cdot):=K_{\varepsilon}\ast \gamma_t^{\varepsilon,N} \ ,
\end{equation}
where $\gamma^{\varepsilon,N}$ is defined by the third line of~\eqref{eq:XIiP3}. 
%Let $u^{\varepsilon,N}$ (resp. $u^{\varepsilon}$) be the function defined as in the third line equation of \eqref{eq:uepsN} (resp. satisfying \eqref{eq:RFK}) {\color{red} with $K_{\varepsilon}$ as introduced just above}. \\ %with $K_{\varepsilon}$ defined by \eqref{eq:Keps}. \\
There is a constant 
 $C$ (only depending on $M_{\Phi}$, $M_g$, $M_{\Lambda}$, $\Vert K \Vert_{\infty}$, 
$ \Vert \nabla K \Vert_{\infty}$, 
$L_{\Phi}$, $L_g$, $L_{\Lambda}$, $T$, 
$\Vert u_0 \Vert_\infty$ and $C_u$)
such that the following holds. 
\begin{enumerate}
\item For all $t \in [0,T]$ and $N \in \N^{\ast}$, $\varepsilon > 0$ verifying $\textrm{min}(N,{N\varepsilon^{d}})>C$  we have  
\begin{equation}
\label{eq:CvgParticles}
\E \Big[ \Vert u_t^{\varepsilon,N} - u_t^{\varepsilon} \Vert_1 \Big]
+\E \Big[ \Vert \nabla u_t^{\varepsilon,N} - \nabla u_t^{\varepsilon} \Vert_1 \Big]
  \leq \frac{C}{\sqrt{N \varepsilon^{d+4}}}e^{\frac{C}{\varepsilon^{d+1}}} \ .
\end{equation}
\item In the particular case where the function $(t,x,y,z)\mapsto \Lambda(t,x,y,z)$ does not depend on the $z$ variable (corresponding to the gradient $\nabla u$), then previous item holds replacing \eqref{eq:CvgParticles}
 with
% there is a constant
%  $C$ (only depending on $M_{\Phi}$, $M_g$, $M_{\Lambda}$, $\Vert K \Vert_{\infty}$,
% $ \Vert \nabla K \Vert_{\infty}$,
% $L_{\Phi}$, $L_g$, $L_{\Lambda}$, $T$)
% such that the following holds.
%  For all $t \in [0,T]$ and $N \in \N^{\ast}$, $\varepsilon > 0$ verifying $\textrm{min}(N,{N\varepsilon^{d}})>C$  we have  
\begin{equation}
\label{eq:CvgParticles2}
\E \Big[ \Vert u_t^{\varepsilon,N} - u_t^{\varepsilon} \Vert_1 \Big]
 \leq \frac{C}{\sqrt{N \varepsilon^{d}}}.
%OLD  \leq \frac{C}{\sqrt{N \varepsilon^{d}}}e^{C\sqrt{T}} \ .
\end{equation}
\end{enumerate}
\end{prop}

\begin{proof}
We begin by establishing the proof of~\eqref{eq:CvgParticles}, in the general case where $\Lambda$ may depend both on $u$ and $\nabla u$. 
Let us fix $\varepsilon > 0$, $N \in \N^{\star}$. For any $\ell=1,\cdots , d$, we introduce the real-valued function $G^{\ell}_{\varepsilon}$ defined on $\R^d$ such that 
\begin{equation}
\label{eq:G}
G^{\ell}_{\varepsilon}(x):=\frac{1}{\varepsilon^d} \frac{\partial K}{\partial x_{\ell}}\left(\frac{x}{\varepsilon}\right)\ ,\quad \textrm{for almost all }\ x\in \R^d\ .
\end{equation}
%Owing to this definition, the partial derivatives of $K_{\varepsilon}$ can be written
%\begin{equation}
%\label{eq:DeriK}
%\frac{\partial K_{\varepsilon}}{\partial x_{\ell}}(x) = \frac{1}{\varepsilon} G^{\ell}_{\varepsilon}(x) \ ,\quad \textrm{for any}\ x\in \R^d\ , \ell = 1,\cdots, d \ .
%\end{equation}
By \eqref{eq:Kcond},  there exists a finite positive constant $C$ independent of $\varepsilon$ such that $\Vert G^{\ell}_{\varepsilon}\Vert_{\infty }\leq \frac{C}{\varepsilon^d}$ and $\Vert  G^{\ell}_{\varepsilon}\Vert_{1}=\Vert  G^{\ell}_1\Vert_{1}\leq C$. In the sequel, $C$ will always denote a finite positive constant independent of $(\varepsilon, N)$ that may change from line to line. 
For any $t\in [0,T]$, we introduce the random Borel measure $\tilde \gamma_t^{\varepsilon,N}$ on $\R^d$, defined by 
%\begin{equation}
%\label{eq:gamma}
%\gamma_t^{\varepsilon,N} = \frac{1}{N}\sum_{i=1}^N V_t\big (\xi^i,u^{\varepsilon,N}(\xi^i),\nabla u^{\varepsilon,N}(\xi^i)\big )\,\delta_{\xi^i_t}\ ,
%\end{equation}
%and
\begin{equation}
\label{eq:gammatilde}
\tilde \gamma_t^{\varepsilon,N}:=\frac{1}{N}\sum_{i=1}^N V_t\big (\xi^i,u^{\varepsilon}(\xi^i),\nabla u^{\varepsilon}(\xi^i)\big )\,\delta_{\xi^i_t}\ .
\end{equation}
%Equation \eqref{eq:gamma} follows by combining equation \eqref{eq:gammaSNXi} and the notation $u^{\varepsilon,N}(t,x) := (K_{\varepsilon} \ast \gamma_t^{\varepsilon,N})(x)$. \\
One can first decompose the error on the l.h.s of inequality~\eqref{eq:CvgParticles} as follows 
\begin{eqnarray}
\label{major-uu'L2P3}
\E \Big[ \Vert u_t^{\varepsilon,N}-u^{\varepsilon}_t\Vert_{1} \Big] +\E \Big[ \Vert \nabla u_t^{\varepsilon,N} - \nabla u_t^{\varepsilon} \Vert_1 \Big]
& = &
\E \Big[ \Vert K_{\varepsilon} \ast (\gamma_t^{\varepsilon ,N}-\gamma_t^{\varepsilon})\Vert_1 \Big]
+\frac{1}{\varepsilon}\sum_{\ell=1}^d \E \Big[ \Vert G^{\ell} _{\varepsilon}\ast (\gamma^{\varepsilon ,N}_t-\gamma^{\varepsilon}_t )\Vert_1 \Big] \nonumber \\
&\leq &
\E \Big[ \Vert K_{\varepsilon} \ast (\gamma_t^{\varepsilon ,N}-\tilde \gamma_t^{\varepsilon,N})\Vert_1 \Big] +
\frac{1}{\varepsilon}\sum_{\ell=1}^d \E \Big[ \Vert G^{\ell} _{\varepsilon}\ast (\gamma^{\varepsilon ,N}_t-\tilde \gamma^{\varepsilon,N}_t)\Vert_1 \Big] \nonumber \\
&&
+\E \Big[ \Vert K_{\varepsilon} \ast (\tilde \gamma_t^{\varepsilon ,N}-\gamma_t^{\varepsilon})\Vert_1 \Big]
+\frac{1}{\varepsilon}\sum_{\ell=1}^d \E \Big[ \Vert G^{\ell} _{\varepsilon}\ast (\tilde \gamma^{\varepsilon ,N}_t-\gamma^{\varepsilon}_t)\Vert_1 \Big] \nonumber \\
&=&
\E \Big[ \Vert A_t^{\varepsilon,N}\Vert_1 \Big] +\E \Big[ \Vert A'^{\varepsilon,N}_t\Vert_1 \Big] +
\E \Big[ \Vert B_t^{\varepsilon,N}\Vert_1 \Big] +\E \Big[ \Vert B'^{\varepsilon,N}_t\Vert_1 \Big] \ , \nonumber \\
\end{eqnarray}
where, for all $t \in [0,T]$,
\begin{equation}
\label{def-ABP3}
\left \{
\begin{array}{lll}
A^{\varepsilon,N}_t(x)
&:= &{\displaystyle \frac{1}{N}\sum_{i=1}^N K_{\varepsilon}(x-\xi^{i}_t)\Big [ V_t\big (\xi^i,u^{\varepsilon,N}(\xi^i),\nabla u^{\varepsilon,N}(\xi^i)\big ) - V_t\big (\xi^i,u^{\varepsilon}(\xi^i),\nabla u^{\varepsilon}(\xi^i)\big )\Big ]} \\
A'^{\varepsilon,N}_t(x)
&:= &{\displaystyle \frac{1}{\varepsilon}\sum_{\ell=1}^d \left \vert \frac{1}{N}\sum_{i=1}^N G^{\ell}_{\varepsilon}(x-\xi^{i}_t)\Big [ V_t\big (\xi^i,u^{\varepsilon,N}(\xi^i),\nabla u^{\varepsilon,N}(\xi^i)\big ) - V_t\big (\xi^i,u^{\varepsilon}(\xi^i),\nabla u^{\varepsilon}(\xi^i)\big )\Big ]\right \vert } \\
B^{\varepsilon,N}_t(x) 
&:=& \displaystyle{ \frac{1}{N} \sum_{i=1}^N  K_{\varepsilon}(x-\xi^i_t) V_t\big (\xi^i,u^{\varepsilon}(\xi^i),\nabla u^{\varepsilon}(\xi^i)\big )-\E\Big [K_{\varepsilon}(x-\xi^1_t) V_t\big (\xi^1,u^{\varepsilon}(\xi^1),\nabla u^{\varepsilon}(\xi^1)\big ) \Big]} \\
B'^{\varepsilon,N}_t(x) 
&:=& \displaystyle{ \frac{1}{\varepsilon}\sum_{\ell=1}^d\ \left \vert \frac{1}{N} \sum_{i=1}^N  G^{\ell}_{\varepsilon}(x-\xi^i_t) V_t\big (\xi^i,u^{\varepsilon}(\xi^i),\nabla u^{\varepsilon}(\xi^i)\big )-\E\Big [G^{\ell}_{\varepsilon}(x-\xi^1_t) V_t\big (\xi^1,u^{\varepsilon}(\xi^1),\nabla u^{\varepsilon}(\xi^1)\big ) \Big]\right \vert } \ .
\end{array}
\right .
\end{equation}
%where we recall that all the processes $\xi^{i}, 1 \le i \le N$ have the same law. \\
We will proceed in two steps, first bounding $\E \Big[\Vert B^{\varepsilon ,N}_t\Vert_1 \Big]$ and $\E \Big[\Vert B'^{\varepsilon ,N}_t\Vert_1 \Big]$ and then
  $\E \Big[\Vert A^{\varepsilon ,N}_t\Vert_1 \Big]$ and $\E \Big[\Vert A'^{\varepsilon ,N}_t\Vert_1\Big]$. 
\begin{description}
\item[Step 1.  Bounding $\E \Vert B^{\varepsilon ,N}_t\Vert_1 $ and $\E \Vert B'^{\varepsilon ,N}_t\Vert_1$.]
%%%%%%%%%%%%%%%%%%%%%%%%%%%%%%%%%%%%%%%%%%%%%%%%%%%%%%%%%
For any  $i \in \{1,\cdots,N\}$ and $(t,x) \in [0,T]\times \R^d$ we set 
\begin{equation}
\label{eq:P}
P^{\varepsilon}_i(t,x) := K_{\varepsilon}(x-\xi^i_t) V_t\big (\xi^i,u^{\varepsilon}(\xi^i),\nabla u^{\varepsilon}(\xi^i)\big )-\E\Big [K_{\varepsilon}(x-\xi^i_t) V_t\big (\xi^i,u^{\varepsilon}(\xi^i),\nabla u^{\varepsilon}(\xi^i)\big ) \Big]
\ .
\end{equation}
Notice that for fixed $(t,x) \in [0,T] \times \R^d$, $(P^{\varepsilon}_i(t,x))_{i=1,\cdots ,N}$ are i.i.d. centered square integrable random variables. Hence using Cauchy-Schwarz inequality, we have
\begin{eqnarray}
\label{majorAP3}
\E\Big[ \Vert B^{\varepsilon,N}_t \Vert_1 \Big] & = & 
%\E{\left[ \int_{\R^d} \vert B^{\varepsilon,N}_t(x) \vert dx\right]} \nonumber \\
%& = & \int_{\R^d} \E{ \Big[\vert B^{\varepsilon,N}_t(x) \vert \Big]} dx \nonumber \\
%& = & 
\int_{\R^d} \E{ \Big[ \Big \vert \frac{1}{N} \sum_{i=1}^N P^{\varepsilon}_i(t,x) \Big \vert \Big]} dx \nonumber \\
& \leq & \int_{\R^d} \sqrt{ \E{ \Big[ \Big( \frac{1}{N} \sum_{i=1}^N (P^{\varepsilon}_i(t,x)) \Big)^2 \Big]} } dx \nonumber \\
& = & \frac{1}{\sqrt{N}}  \int_{\R^d} \sqrt{ \E{ \Big[ (P^{\varepsilon}_1(t,x))^2 \Big]} } dx \ .
\end{eqnarray}
By the boundedness assumption on $\Lambda$ (item 5. of Assumption \ref{ass:mainP3}.),
$$
\E[(P^{\varepsilon}_1(t,x))^2]  \leq  4e^{2M_{\Lambda}T} \E[(K_{\varepsilon}(x-\xi^1_t))^2] \ ,
$$
which implies 
\begin{equation}
\label{eq:711}
\E \Big[ \Vert  B^{\varepsilon,N}_t \Vert_1 \Big] 
 \leq  \frac{C}{\sqrt{N}} \int_{\R^d} \sqrt{\E \Big[(K_{\varepsilon}(x-\xi^1_t))^2 \Big]} dx 
=\frac{C}{\sqrt{N}} \frac{\sqrt{\int_{\R^d} K^2(x)dx}}{\sqrt{\varepsilon^d}}\int_{\R^d} \sqrt {H_{\varepsilon}\ast p_t}(x)dx \ ,
\end{equation}
where we recall that $p_t$, defined in \eqref{UDensity},
%$$ p_t(x) = \int_{\R^d} u_0(x_0) p_t(x_0,x) dx_0, t  > 0,$$
is the law density of $Y_t$ (or $\xi^1_t$).
% where $Y$ is
%solution of \eqref{eq:SDE}  with $s=0$ and $Y_0 \sim u_0) dx_0$.
Moreover $H_{\varepsilon}$ is the probability density on $\R^d$  such that for almost all $x\in \R^d$,
%\begin{equation}
%\label{eq:H}
$H_{\varepsilon}(x) := \frac{1}{\int_{\R^d} K^2(x)dx} \frac{1}{\varepsilon^d} K^2(\frac{x}{\varepsilon})$, 
%\end{equation}
which is well-defined thanks to assumption~\eqref{eq:Kcond}.
Finally, applying Lemma~\ref{lem:carls:ext} with $G = \frac{K^2}{\Vert K \Vert_2^2}$ and $f = p_t$ we obtain 
\begin{equation}
\label{eq:Bbound}
\E \Big[ \Vert  B^{\varepsilon,N}_t \Vert_1 \Big]
\leq
\frac{C}{\sqrt{N\varepsilon^d}} \ , \quad \textrm{for $\varepsilon$ small enough.}
\end{equation}
Proceeding similarly for the term $B'^{\varepsilon,N}_t$ leads to
\begin{equation}
\label{eq:712}
\E \Big[ \Vert  B'^{\varepsilon,N}_t \Vert_1 \Big] 
 \leq  \frac{C}{\sqrt{N \varepsilon^2}} \sum_{\ell=1}^d \int_{\R^d} \sqrt{\E \Big[(G^{\ell}_{\varepsilon}(x-\xi^1_t))^2 \Big]} dx 
=\frac{C}{\sqrt{N \varepsilon^2}} \sum_{\ell=1}^d  \frac{\sqrt{\int_{\R^d} \vert \frac{\partial K}{\partial x_{\ell}}(x) \vert^2 dx}}{\sqrt{\varepsilon^d}}\int_{\R^d} \sqrt {H^{\ell}_{\varepsilon}\ast p_t}(x)dx \ ,
\end{equation}
where $H^{\ell}_{\varepsilon}, \ell = 1,\cdots,d$ denotes the probability densities on $\R^d$ such that for almost all 
$x\in \R^d$, 
$H^{\ell}_{\varepsilon}(x) := \frac{1}{\int_{\R^d} \vert \frac{\partial K}{\partial x_{\ell}}(x) \vert^2 dx} \frac{1}{\varepsilon^d} \vert \frac{\partial K}{\partial x_{\ell}}(\frac{x}{\varepsilon}) \vert^2$. Applying again Lemma~\ref{lem:carls:ext} with $G = \frac{ \vert \frac{\partial K}{\partial x_{\ell}} \vert^2}{ \Vert \frac{\partial K}{\partial x_{\ell}} \Vert_2^2}, \ell = 1,\cdots,d$ and $f$ being the $ f = p_t$ we obtain
\begin{equation}
\E \Big[ \Vert  B'^{\varepsilon,N}_t \Vert_1 \Big]
\leq
\frac{C}{\sqrt{N\varepsilon^{d+2}}} \ , \quad \textrm{for $\varepsilon$ small enough.}
\end{equation}

\item[Step 2. Bounding $\E\Vert A^{\varepsilon ,N}_t\Vert_1$ and $\E\Vert A'^{\varepsilon ,N}_t\Vert_1$.] 
%%%%%%%%%%%%%%%%%%%%%%%%%%%%%%%%%%%%%%%%%%%%%%%%
Recall that 
$A_t^{\varepsilon,N}(x)=K_{\varepsilon}\ast (\gamma_t^{\varepsilon,N}-\tilde \gamma_t^{\varepsilon,N})(x)$ and 
$A'^{\varepsilon,N}_t(x)=\frac{1}{\varepsilon}\sum_{\ell=1}^d\vert G^{\ell}_{\varepsilon}\ast (\gamma_t^{\varepsilon,N}-\tilde \gamma_t^{\varepsilon,N})\vert(x)$, 
which yields 
\begin{equation} 
\label{eq:AVT}
\E\Big[ \Vert A_t^{\varepsilon,N}\Vert_1\Big]+\E\Big[ \Vert A'^{\varepsilon,N}_t\Vert_1\Big]\leq \frac{C}{\varepsilon}\E \Big[\Vert \gamma_t^{\varepsilon,N}-\tilde \gamma_t^{\varepsilon,N}\Vert_{TV} \Big]\ .
\end{equation}
We are now interested in bounding the r.h.s. of \eqref{eq:AVT}. 

Recalling \eqref{eq:XIiP3}, \eqref{eq:gammatilde} and inequality \eqref{eq:LipV}, we have
\begin{eqnarray}
\label{eq:TV}
\E \Big[\Vert \gamma_t^{\varepsilon,N}-\tilde \gamma_t^{\varepsilon,N}\Vert_{TV} \Big]
&=&
\frac{1}{N}\sum_{i=1}^N \E \Big[\vert V_t(\xi^i,u^{\varepsilon,N},\nabla u^{\varepsilon,N})-V_t(\xi^i,u^{\varepsilon},\nabla u^{\varepsilon})\vert \Big] \nonumber \\
&\leq &
C\E\left [\int_0^t\big ( \vert  u_s^{\varepsilon,N}-u_s^{\varepsilon}\vert (\xi^1_s) +\vert \nabla u_s^{\varepsilon,N}-\nabla u_s^{\varepsilon}\vert (\xi^1_s)\big )\,ds\right ] \nonumber \\
&\leq &
C\int_0^t \left (\E \Big[\vert  K_{\varepsilon}\ast (\gamma_s^{\varepsilon,N}-\tilde \gamma_s^{\varepsilon,N})(\xi^1_s)  \vert \Big]
+\E \Big[\vert K_{\varepsilon}\ast (\tilde \gamma_s^{\varepsilon,N}- \gamma_s^{\varepsilon})(\xi^1_s)\vert  \Big] \right) ds \ ,\nonumber \\
&&+
\frac{C}{\varepsilon} \sum_{\ell=1}^d\int_0^t \left (\E \Big[ \vert   G^{\ell}_{\varepsilon}\ast (\gamma_s^{\varepsilon,N}-\tilde \gamma_s^{\varepsilon,N})(\xi^1_s) \vert \Big] 
+\E \Big[ \vert G^{\ell}_{\varepsilon}\ast (\tilde \gamma_s^{\varepsilon,N}- \gamma_s^{\varepsilon})(\xi^1_s)\vert \Big] \right )\, ds. \nonumber \\
\end{eqnarray}
%
%Recalling  that the law density of $\xi_s^1$ verifies  $p_s(z) = \int_{R^d} p(0,x_0,s,z) u_0(x_0) dx_0$, from 
By  inequality \eqref{eq:majordens} in Lemma \ref{lem:transfun}
and \eqref{UDensity},
% there exists a finite constant $C > 0$ such that 
$\Vert p_s \Vert _{\infty}\leq C_u \Vert u_0\Vert_{\infty}$ for all $s \in [0,T]$. 
Recalling that $\gamma^\varepsilon$ verifies \eqref{eq:FKRegul},
 using inequality~\eqref{eq:Bbound}, we obtain  
\begin{equation}
\label{eq:NhK}
\begin{array}{l}
{\displaystyle 
\E \Big[ \vert K_{\varepsilon}\ast (\tilde \gamma_s^{\varepsilon,N}- \gamma_s^{\varepsilon})(\xi^1_s)\vert \Big]
}
\\
%+\sum_{\ell=1}^d \E\vert \frac{1}{\varepsilon}G^{\ell}_{\varepsilon}\ast (\tilde \gamma_s^{\varepsilon,N}- \gamma_s^{\varepsilon})(\xi^1_s)\vert 
{\displaystyle 
\leq 
\frac{1}{N} \E \Big[ \left \vert K_{\varepsilon}(0)V_s(\xi^1,u^{\varepsilon}(\xi^1),\nabla u^{\varepsilon}(\xi^1))-\E[K_{\varepsilon}(\xi^1_s-\xi^2_s)V_s(\xi^2,u^{\varepsilon}(\xi^2),\nabla u^{\varepsilon}(\xi^2))\,\vert \, \xi^1]\right \vert \Big]
}
\\
{\displaystyle 
+\frac{N-1}{N}
\frac{1}{N-1} \int_{\R^d} \left \vert  \sum_{i=2}^N \Big [ K_{\varepsilon}(x-\xi^i_s) V_s(\xi^i,u^{\varepsilon}(\xi^i),\nabla u^{\varepsilon}(\xi^i))-\E[K(x-\xi^i_s)V_s(\xi^i,u^{\varepsilon}(\xi^i),\nabla u^{\varepsilon}(\xi^i))]\Big ] \right \vert p_s(x)dx }
\\
{\displaystyle 
\leq \frac{C}{N\varepsilon^{d}}+\frac{N-1}{N}\frac{C}{\sqrt{(N-1)\varepsilon^{d}}}
}\\
{\displaystyle 
\leq \frac{C}{\sqrt{N\varepsilon^{d}}}\quad\quad  \textrm{(for}\quad (N\ \textrm{and}\ {N\varepsilon^d})\quad \textrm{sufficiently large), } s \in [0,T] }\ .
\end{array}
\end{equation}
Similarly we get
\begin{equation}
\label{eq:NhG}
\sum_{\ell=1}^d \E \Big[ \vert \frac{1}{\varepsilon}G^{\ell}_{\varepsilon}\ast (\tilde \gamma_s^{\varepsilon,N}- \gamma_s^{\varepsilon})(\xi^1_s)\vert \Big]
\leq 
\frac{C}{\sqrt{N\varepsilon^{d+2}}}, \quad s \in [0,T] \ .
%+\frac{C}{N\varepsilon^{d+2}}
\end{equation}
Moreover,  for all $s \in [0,T]$, the boundedness of $ \vert K \vert$ and $\vert \nabla K \vert$ implies 
%$$
%\Vert  K_{\varepsilon}\ast (\gamma_s^{\varepsilon,N}-\tilde \gamma_s^{\varepsilon,N})(\xi^1_s)  \Vert_{\infty}
%+\Vert  \frac{1}{\varepsilon} G^{\ell}_{\varepsilon}\ast (\gamma_s^{\varepsilon,N}-\tilde \gamma_s^{\varepsilon,N})(\xi^1_s) \Vert_{\infty}
%\leq 
%(\Vert K_{\varepsilon} \Vert_{\infty}+\frac{1}{\varepsilon}\Vert G_{\varepsilon}\Vert_{\infty})\Vert  \gamma_s^{\varepsilon,N}-\tilde \gamma_s^{\varepsilon,N}\Vert_{TV}\ ,
%$$
%implies 
\begin{equation}
\label{eq:Gron}
\E \Big[ \vert  K_{\varepsilon}\ast (\gamma_s^{\varepsilon,N}-\tilde \gamma_s^{\varepsilon,N})(\xi^1_s)  \vert \Big]
+\sum_{\ell=1}^d \E \Big[ \vert  \frac{1}{\varepsilon} G^{\ell}_{\varepsilon}\ast (\gamma_s^{\varepsilon,N}-\tilde \gamma_s^{\varepsilon,N})(\xi^1_s) \vert \Big]
\leq 
\frac{C}{\varepsilon^{d+1}} \Big[ \Vert  \gamma_s^{\varepsilon,N}-\tilde \gamma_s^{\varepsilon,N}\Vert_{TV} \Big]\ .
\end{equation}
Injecting inequalities~\eqref{eq:NhK}~\eqref{eq:NhG} and~\eqref{eq:Gron} into~\eqref{eq:TV} gives 
$$
\E \Big[ \Vert \gamma_t^{\varepsilon,N}-\tilde \gamma_t^{\varepsilon,N}\Vert_{TV} \Big]
\leq \frac{C}{\sqrt{N\varepsilon^{d+2}}}
+
\frac{C}{\varepsilon^{d+1}}\int_0^t\E \Big[ \Vert  \gamma_s^{\varepsilon,N}-\tilde \gamma_s^{\varepsilon,N}\Vert_{TV} \Big] ds\ .
$$
By Gronwall's lemma we obtain 
%\begin{equation}
%\label{eq:TVend}
$
\E \Big[ \Vert \gamma_t^{\varepsilon,N}-\tilde \gamma_t^{\varepsilon}\Vert_{TV} \Big]
\leq 
\frac{C}{\sqrt{N\varepsilon^{d+2}}}e^{\frac{C}{\varepsilon^{d+1}}}\ ,
$
%\end{equation}
which together with~\eqref{eq:AVT} completes the proof of~\eqref{eq:CvgParticles} by implying the inequality
\begin{equation}
\label{eq:AVTBound }
\E[\Vert A_t^{\varepsilon,N}\Vert_1]+\E[\Vert A'^{\varepsilon,N}_t\Vert_1]\leq \frac{C}{\sqrt{N\varepsilon^{d+4}}}e^{\frac{C}{\varepsilon^{d+1}}}\ .
\end{equation}
%that completes the proof.
\end{description}
%%%%%%%%%%%%%%%%%%%%%%%%%%%%%%%%%%%%%%%%%%%%%%%%%%%%%%%%%%%%%%%%%%%%%%%%%%%%%%%%%%%%%%%%%%%%%%%%ù
%Gronwall
%\begin{prop}
%\label{prop:CvgParticle2}
%   la separation de $\xi^1$ que tu as considere comme independant 
%des particules.}
%\\
%Under Assumption \ref{ass:mainP3}, with  $K$ verifying~\eqref{eq:Kcond}, in the particular case where the function $\Lambda(t,x,u)$ does not depend on the gradient $\nabla u$,
%there is a constant
 %$C$ (only depending on $M_{\Phi}$, $M_g$, $M_{\Lambda}$, $\Vert K \Vert_{\infty}$,
%$ \Vert \nabla K \Vert_{\infty}$,
%$L_{\Phi}$, $L_g$, $L_{\Lambda}$, $T$)
%such that the following holds.
 %For all $t \in [0,T]$ and $N \in \N^{\ast}$, $\varepsilon > 0$ verifying $\textrm{min}(N,{N\varepsilon^{d}})>C$  we have  
%\begin{equation}
%\label{eq:CvgParticles2}
%\E \Big[ \Vert u_t^{\varepsilon,N} - u_t^{\varepsilon} \Vert_1 \Big]
  %\leq \frac{C}{\sqrt{N \varepsilon^{d}}}e^{C\sqrt{T}} \ .
%\end{equation}
%\end{prop}
%\begin{proof}
%%%%%%%%%%%%%%%%%%%%%%%%%%%%%%%%%%%%%%%%%%%%%%%%%%%%%%%%%%%%%%%%%%%%%%%%%%%%
\item
Now let us treat the proof of~\eqref{eq:CvgParticles2}, in the specific case where $\Lambda$ does not depend on $\nabla u$. Adapting~\eqref{major-uu'L2P3} when
 $\nabla u$ does not appear in $\Lambda$ yields
\begin{equation}
\label{major1}
\E \Big[ \Vert u_t^{\varepsilon,N}-u^{\varepsilon}_t\Vert_{1} \Big]
\leq
\E \Big[ \Vert A_t^{\varepsilon,N}\Vert_1 \Big] +
\E \Big[ \Vert B_t^{\varepsilon,N}\Vert_1 \Big]  \ ,
\end{equation}
where $A_t$ and $B_t$ are given by~\eqref{def-ABP3}.
To bound $\E \Big[ \Vert A_t^{\varepsilon,N}\Vert_1 \Big]$,
we rely again on~\eqref{eq:TV}, 
which gives
\begin{eqnarray}
\label{eq:TV2a}
\E \Big[ \Vert A_t^{\varepsilon,N}\Vert_1 \Big]&=&\E \Big[\Vert K_{\varepsilon}\ast (\gamma_t^{\varepsilon,N}-\tilde \gamma_t^{\varepsilon,N})\Vert_{1} \Big]
\nonumber \\
&\leq &\E \Big[\Vert \gamma_t^{\varepsilon,N}-\tilde \gamma_t^{\varepsilon,N}\Vert_{TV} \Big]\nonumber \\
&\leq & 
C\int_0^t \E [\vert (u_s^{\varepsilon ,N}-u^\varepsilon _s)(\xi^1_s)\vert ] ds \nonumber \\
&=&
C\int_0^t \E [\vert (K_{\varepsilon}\ast \gamma_s^{\varepsilon ,N})(\xi^1_s)- u^\varepsilon _s(\xi^1_s)\vert ] ds\ .  
\end{eqnarray}
Considering an additional particle $\xi^0$ such that  $(\xi^0,\xi^1,\cdots , \xi^N)$ are i.i.d. yields 
\begin{eqnarray*}
 \E [\vert (K_{\varepsilon}\ast \gamma_s^{\varepsilon ,N})(\xi^1_s)-u^\varepsilon _s(\xi^1_s)\vert ]
&=&
\E [\vert\frac{1}{N}\sum_{i=1}^N V_s(\xi^i,u^{\varepsilon,N}(\xi^i))K_{\varepsilon}(\xi^1_s-\xi^i_s)-u^\varepsilon _s(\xi^1_s) \vert ]\\
&\leq &
\E [\vert\frac{1}{N}V_s(\xi^1,u^{\varepsilon,N}(\xi^1))K_{\varepsilon}(0)\vert ]\\
&&
%%% FINE VERIFICATIONS FRANCESCO
+\E [\vert \frac{1}{N}\sum_{i=2}^N V_s(\xi^i,u^{\varepsilon,N}(\xi^i))K_{\varepsilon}(\xi^1_s-\xi^i_s)-u^\varepsilon _s(\xi^1_s) \vert ]\\
&\leq &
\frac{C}{N\varepsilon ^d}+
\E [\vert \frac{1}{N}\sum_{i=2}^N V_s(\xi^i,u^{\varepsilon,N}(\xi^i))K_{\varepsilon}(\xi^0_s-\xi^i_s)-u^\varepsilon _s(\xi^0_s) \vert ]\\
&\leq &
\frac{2C}{N\varepsilon ^d}+
\E [\vert \frac{1}{N}\sum_{i=1}^N V_s(\xi^i,u^{\varepsilon,N}(\xi^i))K_{\varepsilon}(\xi^0_s-\xi^i_s)-u^\varepsilon _s(\xi^0_s) \vert ]\\
&=&
\frac{2C}{N\varepsilon ^d}+
C\int_0^t \E [\vert (K_{\varepsilon}\ast \gamma_s^{\varepsilon ,N})(\xi^0_s)- (K_{\varepsilon}\ast \gamma^\varepsilon _s)(\xi^0_s)\vert ] ds\ .
\end{eqnarray*}
Injecting the above inequality in~\eqref{eq:TV2a} and using triangle inequality yields (reminding that $C$ is a constant that may change from line to line) 
%%%%%%%%%%%%%%%
\begin{eqnarray}
\label{eq:TV2}
\E \Big[ \Vert A_t^{\varepsilon,N}\Vert_1 \Big]
&\leq & 
\frac{C}{N\varepsilon^d}+
C\int_0^t \left (\E \Big[\vert  K_{\varepsilon}\ast (\gamma_s^{\varepsilon,N}-\tilde \gamma_s^{\varepsilon,N})(\xi^0_s)  \vert \Big]
+\E \Big[\vert K_{\varepsilon}\ast (\tilde \gamma_s^{\varepsilon,N}- \gamma_s^{\varepsilon})(\xi^0_s)\vert  \Big]  \right ) ds\nonumber \\
&\leq &
\frac{C}{N\varepsilon^d}+
C\int_0^t \Vert p_s\Vert_{\infty} \left (\E \Big[\Vert  K_{\varepsilon}\ast (\gamma_s^{\varepsilon,N}-\tilde \gamma_s^{\varepsilon,N}) \Vert_1 \Big]
+\E \Big[\Vert K_{\varepsilon}\ast (\tilde \gamma_s^{\varepsilon,N}- \gamma_s^{\varepsilon})\Vert_1  \Big]  \right ) ds\nonumber \\
&= &
\frac{C}{N\varepsilon^d}+
C\int_0^t \Vert p_s\Vert_{\infty} \big ( \E \Big[ \Vert A_s^{\varepsilon,N}\Vert_1 \Big]+\E \Big[ \Vert B_s^{\varepsilon,N}\Vert_1 \Big] \big ) ds\ .
\end{eqnarray}
Using the fact that 
$\Vert p_s\Vert_{\infty} \leq C_u \Vert u_0\Vert_\infty$ by 
\eqref{eq:majordens}
and inequality ~\eqref{eq:Bbound}, %$\E \Big[ \Vert  B^{\varepsilon,N}_t \Vert_1 \Big]\leq\frac{C}{\sqrt{N\varepsilon^d}} $
implies that for $\varepsilon$ small enough we obtain
% \begin{equation}
% \label{eq:GronBis}
% \E \Big[ \Vert A_t^{\varepsilon,N}\Vert_1 \Big]+\E \Big[ \Vert B_t^{\varepsilon,N}\Vert_1 \Big]
% \leq
% \frac{C}{\sqrt{N\varepsilon^d}}+C\int_0^t \sqrt{s}( \E \Big[ \Vert A_s^{\varepsilon,N}\Vert_1 \Big]+\Big[ \E \Vert B_s^{\varepsilon,N}\Vert_1 \Big] ) ds\ .
% \end{equation}
\begin{equation}
\label{eq:GronBis}
\E \Big[ \Vert A_t^{\varepsilon,N}\Vert_1 \Big]+\E \Big[ \Vert B_t^{\varepsilon,N}\Vert_1 \Big]
\leq
\frac{C}{\sqrt{N\varepsilon^d}}+C\int_0^t ( \E \Big[ \Vert A_s^{\varepsilon,N}\Vert_1 \Big]+\Big[ \E \Vert B_s^{\varepsilon,N}\Vert_1 \Big] ) ds\ .
\end{equation}

Gronwall's lemma gives
\begin{equation}
\label{eq:Gronb}
\E \Big[ \Vert A_t^{\varepsilon,N}\Vert_1 \Big]+\E \Big[ \Vert A_t^{\varepsilon,N}\Vert_1 \Big]
\leq \frac{C}{\sqrt{N\varepsilon^d}}.
\end{equation}

%%%%%%%%%%%%%%%%%%%%%%%%%%%%%%%%%%%%%%%%%%%%%%%%%%%%%%%%%%%%%%%%%%%%%%%%%%%%%%%%%%%%%%%%%%%%%%%%%
\end{proof}
%s a straightforward consequence of Proposition \ref{prop:CvgParticle} and Theorem \ref{thm:cvguu'}, the corollary below follows.
%we claim that $u^{\varepsilon , N}$ constitutes a convergent approximation of the solution $u$ of the semi-linear PDE~\eqref{eq:PDE}. 
\begin{corro}
\label{thm:ThmCvg}
We suppose the validity of Assumption \ref{ass:mainP3}.
Let 
Assume that the kernel $K$ is verifying  
\eqref{EKappa} and \eqref{eq:Kcond}.
\begin{enumerate}
\item
If $\varepsilon \rightarrow 0$, $N \rightarrow + \infty$ such that $
\frac{1}{\sqrt{N \varepsilon^{d+4}}}e^{\frac{C}{\varepsilon^{d+1}}} \rightarrow 0 \ ,
$ (where $C$ is the constant coming from Proposition \ref{prop:CvgParticle})
then
\begin{equation}
\label{eq:cv:vitesse0}
\E \Big[ \Vert u_t^{\varepsilon,N} - u_t \Vert_1 \Big] + \E \Big[ \Vert \nabla  u_t^{\varepsilon,N} - \nabla  u_t \Vert_1 \Big] \longrightarrow 0 \ .
\end{equation}
\item
In the particular case where the function $(t,x,y,z)\mapsto \Lambda(t,x,y,z)$ does not depend on the $z$ variable (corresponding to the gradient $\nabla u$), there is a constant 
 $C$ (only depending on $\kappa$, $C_u$, $M_{\Phi}$, $M_g$, $M_{\Lambda}$, $\Vert K \Vert_{\infty}$, 
$ \Vert \nabla K \Vert_{\infty}$, 
$L_{\Phi}$, $L_g$, $L_{\Lambda}$, $T$, 
$\Vert u_0 \Vert_\infty$),
such that the following holds. 
\begin{equation}
\label{eq:cv:vitesse}
\E \Big[ \Vert u_t^{\varepsilon,N} - u_t \Vert_1 \Big] \leq  C \left(\varepsilon +\frac{1}{\sqrt{N\varepsilon^d}}\right)\ .
\end{equation}

\end{enumerate}
\end{corro}
\begin{proof}
Let us fix $\varepsilon > 0$, $N \in \N^{\star}$, $t \in [0,T]$. The proof of~\eqref{eq:cv:vitesse0} is based on the  bound
\begin{eqnarray} \label{EFundBound}
\E \Big[ \Vert u_t^{\varepsilon,N} - u_t \Vert_1 \Big] + \E \Big[ \Vert \nabla  u_t^{\varepsilon,N} - \nabla  u_t \Vert_1 \Big]& \leq & \E \Big[ \Vert u_t^{\varepsilon,N} - u_t^{\varepsilon} \Vert_1 \Big] + \E \Big[ \Vert \nabla u_t^{\varepsilon,N} - \nabla u_t^{\varepsilon} \Vert_1 \Big] \nonumber \\
&& \; + \Vert u_t^{\varepsilon} - u_t \Vert_1 + \Vert \nabla u_t^{\varepsilon} - \nabla u_t \Vert_1 \ , \nonumber \\
& \leq & \frac{C}{\sqrt{N \varepsilon^{d+4}}}e^{\frac{C}{\varepsilon^{d+1}}} + \Vert u_t^{\varepsilon} - u_t \Vert_1 + \Vert \nabla u_t^{\varepsilon} - \nabla u_t \Vert_1 \ ,
\end{eqnarray}
where we have used Proposition \ref{prop:CvgParticle} for the second inequality above. 
 Taking into account Theorem \ref{thm:cvguu'} above, it appears clearly that the convergence of $u^{\varepsilon,N}$ (resp. $\nabla u^{\varepsilon,N}$) to $u$ (resp. $\nabla u$) will hold as soon as $\frac{1}{\sqrt{N \varepsilon^{d+4}}}e^{\frac{C}{\varepsilon^{d+1}}} \rightarrow 0$ when $\varepsilon \rightarrow 0$, $N \rightarrow + \infty$. This concludes the proof of~\eqref{eq:cv:vitesse0}.

The second inequality~\eqref{eq:cv:vitesse}, concerning the specific case where $\Lambda$ does not depend on the gradient $\nabla u$), is proved similarly by gathering inequality~\eqref{eq:Cvgu} from Proposition~\ref{C47} and inequality~\eqref{eq:CvgParticles2} from Proposition \ref{prop:CvgParticle}.
\end{proof}
%%% OLD
% Moreover, Theorem \ref{thm:cvguu'} gives $\Vert u_t^{\varepsilon} - u_t \Vert_1 + \Vert \nabla u_t^{\varepsilon} - \nabla u_t \Vert_1 \longrightarrow 0$ for $\varepsilon \rightarrow 0$. This concludes the proof of the corollary.
% \end{proof}
\begin{rem}
\label{rem:CvgPartic}
\begin{enumerate}
\item 
In the first statement of Corollary \ref{thm:ThmCvg} appears the condition
$\frac{1}{\sqrt{N \varepsilon^{d+4}}}e^{\frac{C}{\varepsilon^{d+1}}} \rightarrow 0$ when $\varepsilon \rightarrow 0$, $N \rightarrow + \infty$.
 This requires a "trade-off" between the speed of convergence of $N$ and $\varepsilon$. Setting $\Phi(\varepsilon) := \varepsilon^{-(d+4)} e^{\frac{2C}{\varepsilon^{d+1}}}$, the trade-off condition can be formulated as 
\begin{equation}
\label{eq:TradeOff}
 \frac{\Phi(\varepsilon)}{N} \rightarrow 0 \quad \textrm{when} \quad \varepsilon \rightarrow 0, \; N \rightarrow + \infty.
\end{equation}
An example of such trade-off between $N$ and $\varepsilon$ can be given by the relation $\varepsilon (N) = \left(\frac{1}{\log(N)}\right)^{\frac{1}{d+4}}$. 
\item The estimate \eqref{eq:cv:vitesse} recovers the same order of
convergence as the one encountered in classical density estimates, 
see e.g. (22) in \cite{holm92}. This happens in spite of the 
fact the weights $V_t$ in \eqref{eq:XIiP3} depend on the whole past of the
whole particle system.
\end{enumerate}
\end{rem}

\section{Appendix}
\label{SAppendixP3}
\setcounter{equation}{0}

\subsection{General inequalities}

If $f$ is a probability density on $\R^d$, $I(f)$ denotes the quantity $I(f):=\int_{\R^d} \vert x\vert^{d+1} f(x)dx$.
\begin{lem}[Multidimensional Carlson's inequality]
\label{carlson}
Let $f$ be a probability density  on $\R^d$ such that $I(f) < \infty$, then
\begin{equation}
\label{eq:carlson}
\int_{\R^d} \sqrt{f(x)}dx \leq  A_d \,I(f)^{\frac{d}{2(d+1)}}
%%LES CONSTANTES PRECEDENTES ME SEMBLAIENT IMPRECISES.
% \hspace{1cm} \textrm{where}\hspace{1cm} A_d=\left(\frac{(2\pi) ^{\frac{d+1}{2}}}{\Gamma
% (\frac{d+1}{2})}
% \right )^{1/2}\ .
\hspace{1cm} \textrm{where}\hspace{1cm} A_d=
\left(\frac{ 2 \pi ^{\frac{d+2}{2}}}{\Gamma
(\frac{d}{2})  d^{\frac{d}{d+1}} \sin\left(\frac{d\pi}{d+1}\right)}
\right )^{1/2}\ .
\end{equation}
\end{lem}
\begin{proof}
We apply (16) in Lemma 7 (p.251) of
\cite{holm92} setting $g = \sqrt f, \ \varepsilon = 1$. 
%,  where a more precise estimate is proved. 
\end{proof}
From Lemma \ref{carlson}, we deduce the following lemma.
\begin{lem}
\label{lem:carls:ext}
Let $G$ and $f$ be two probability densities defined on $\R^d$ such that 
\begin{equation}
\label{eq:Gf}
 I(G) < \infty\ ,\quad\textrm{and}\quad I(f) < \infty \ .
\end{equation}
Then for any strictly positive real  $\varepsilon \leq(1/I(G))^{\frac{1}{d+1}}$, 
\begin{equation}
\label{eq:carls:ext}
\int_{\R^d} \sqrt{(G_{\varepsilon}\ast f)(x)}\, dx \leq 2^{\frac{d}{2}}\,A_d\,[1+I(f)],\quad 
%%%OLD
% \textrm{and}\quad A_d=\left(\frac{(2\pi) ^{\frac{d+1}{2}}}{\Gamma
% (\frac{d+1}{2})}\right )^{1/2}\ ,
\textrm{where}\hspace{1cm} A_d=
\left(\frac{ 2 \pi ^{\frac{d+2}{2}}}{\Gamma
(\frac{d}{2})  d^{\frac{d}{d+1}} \sin\left(\frac{d\pi}{d+1}\right)}
\right )^{1/2},
\end{equation}
and $G_{\varepsilon}(\cdot) := \frac{1}{\varepsilon^d} G(\frac{\cdot}{\varepsilon})$.
\end{lem}
\begin{proof}
%Setting $I(G_{\varepsilon}\ast f) := \int_{\R^d} \vert x \vert^{d+1} (G_{\varepsilon} \ast f)(x) dx$, 
By Carlson's inequality~(\ref{eq:carlson}) we have
\begin{eqnarray}
\label{eq:Carls}
\int_{\R^d}\sqrt{(G_{\varepsilon}\ast f)(x)}dx\leq A_d\,[I(G_{\varepsilon}\ast f)]^{\frac{d}{2(d+1)}} \ .
\end{eqnarray}
Then
%% POUR MOI C'EST L'INEGALITE TRIANGULAIRE DANS L^{d+1} PAS LE CAS ICI
%by Minkowski's inequality, 
\begin{eqnarray*}
[I( G_{\varepsilon}\ast f)]^\frac{1}{d+1}
&=&\big [\,\int_{\R^d\times \R^d}|x|^{d+1}\,G_{\varepsilon}(x-y)\,f(y)dy\,dx\,\big
]^\frac{1}{d+1} \\ 
&=&\big [\,\int_{\R^d\times\R^d}|y+u\varepsilon|^{d+1}\,G(u)\,f(y)dy\,du\,\big ]^\frac{1}{d+1} \\
&\le& \big [\,\int_{\R^d\times\R^d}|y|^{d+1}\,G(u)\,f(y)dy\,du\,\big ]^\frac{1}{d+1} +
\varepsilon \big [\,\int_{\R^d\times\R^d}|u|^{d+1}\,G(u)\,f(y)dy\,du\,\big ]^\frac{1}{d+1} \\
&\leq & I(f)^\frac{1}{d+1}+\varepsilon\,I(G)^\frac{1}{d+1} \ .
\end{eqnarray*}
Since  $x \in \R^{+} \mapsto x^d$ is convex, it follows
\begin{eqnarray*}
I( G_{\varepsilon}\ast f)^\frac{d}{2(d+1)}
&\leq & 
2^\frac{d-1}{2}\,[\,[I(f)]^\frac{d}{d+1}+\varepsilon^d\,[I(G)]^\frac{d}{d+1}]^\frac{1}{2}\
.
\end{eqnarray*}
Hence, as soon as  $\varepsilon\leq (1/I(G))^\frac{1}{d+1}$, we have
\begin{equation}
[I( G_{\varepsilon}\ast f)]^\frac{d}{2(d+1)}
\leq 
2^\frac{d}{2}\,[1+I(f)] \ ,
\end{equation}
which, owing to \eqref{eq:Carls}, concludes the proof.  
\end{proof}
 
%%%%%%%%%%%%%%%%%%%%%%%%%%%%%%%%%%%%%%%%%%%%%%%%%%%%%%%%%%%%%%%%%%%%%%%%%%%%%%%%%%%%%%%%%%%%%%%%%%%%%%%
\begin{lem}
\label{lem:bias} 
Let $H$ be a density kernel on $\R^d$ 
satisfying 
\begin{equation}
\label{eq:kernel}
 H \ge 0,  \quad\int_{\R^d} H(x)\, dx = 1.
%\quad 
%   \int_{\R^d} x\, H(x)\, dx = 0 \ .
\end{equation}
Let $f : \R^d \rightarrow \R$ be a real-valued function. For any $\varepsilon > 0$, 
we consider the function $H_{\varepsilon}$ given by 
\begin{equation}
\label{eq:Heps}
H_{\varepsilon}(\cdot) := \frac{1}{\varepsilon^d}H\left(\frac{\cdot}{\varepsilon}\right) \ . \end{equation}
If $a:= \frac{1}{2}\int_{\R^d} \vert x \vert\, H(x)\, dx < \infty$ 
%(resp. $\tilde{a}:= \int_{\R^d} \vert x \vert\, H(x)\, dx < \infty$) and 
$f\in W^{1,p}$ for some integer $p\geq 1$, then for any $\varepsilon > 0$,
%\begin{equation}
%\label{eq:lem:bias'}
%\Vert H_{\varepsilon}\ast f-f\Vert_{p}\leq \varepsilon ^2\,a \sum_{i,j=1}^d \Vert \partial_i \partial_j f\Vert_{p}\ .
%\end{equation}
\begin{equation}
\label{eq:lem:biasW1}
\Vert H_{\varepsilon}\ast f-f\Vert_{p}\leq \varepsilon \,a \sum_{i=1}^d \Vert \partial_i f\Vert_{p} \ .
\end{equation}
%Moreover, if $f$ is only one time differentiable with in $W^{1,p}$ for some $p \geq 1$, we also have
%\begin{equation}
%\label{eq:lem:bias''}
%\Vert H_{\varepsilon}\ast f-f\Vert_{p}\leq \varepsilon \,a \sum_{i=1}^d \Vert \partial_i f\Vert_{p}\ ,
%\end{equation}
\end{lem}

\begin{proof}
The proof is modeled on~\cite{holm92}. 
For $\varepsilon > 0$ and any integer $1\leq i\leq d$, let us introduce the real-valued function
$L_{\varepsilon}^{i}$ defined on $\R^d$ with values in $\bar \R_+$, associated with $H$ such that for almost all $x \in \R^d$,
\begin{equation} \label{E710}
L_{\varepsilon}^{i}(x)=\frac{x_i}{\varepsilon} \int_0^1\frac{1-t}{t}H_{\varepsilon t}(x)\,dt,
\end{equation}
where $x_i$ is the $i$-th coordinate of $x$ and $H_t$ given by \eqref{eq:Heps}.
Observe that, for any $\varepsilon > 0$, $1\leq i\leq d$, 
\begin{equation}
\label{eq:L1Lij}
\sum_{i=1}^d \Vert L_{\varepsilon}^{i} \Vert_1 =  \int_{\R^d} \sum_{i=1}^d \vert L_{\varepsilon}^{i}(x)\vert dx= a\ ,
\end{equation}
%Since $K$ is supposed to be radially symmetric, $\int_{\R^d}
%L^{i,j}(x)dx=0$ when $i\neq j$ and $\int_{\R^d}
%L^{i,i}(x)dx=\displaystyle{a }$ for any integer $1 \leq i\leq d$.
which implies that $L_{\varepsilon}^{i} < \infty $ a.e.
Developing $f$ according to the Lagrange expansion up to order one,
yields, for almost all $(x,y) \in (\R^d)^2$,
$$
f(x-y)=f(x)-\sum_{i=1}^d\int_0^1(1-t)(\partial_if)(x-ty)y_i \,dt\
.
$$
Integrating this expression against $H_{\varepsilon}$ w.r.t. $y$ and using the
symmetry of $H$, yields for almost all $x \in \R^d$, \\
\begin{eqnarray}
\label{eq:kh:lem}
(H_{\varepsilon}\ast f)(x)-f(x)
&=&\int_{\R^d} [f(x-y)-f(x)]\,H_{\varepsilon}(y)\,dy \nonumber \\ 
&=& \sum_{i=1}^d\int_{\R^d} \int_0^1(1-t)\partial_if(x-ty)y_i\,dt\,H_{\varepsilon}(y)\,dy \nonumber \\ 
&=& \varepsilon\, \sum_{i=1}^d\int_{\R^d}  \partial_i f(x-u)\frac{u_i}{\varepsilon}\int_0^1
\frac{1-t}{t} H_{\varepsilon t}(u)\,dt\,du\nonumber \\ 
&=&
\varepsilon\, \sum_{i=1}^d (L^{i}_{\varepsilon}\ast (\partial_i f))(x).
\end{eqnarray}
Taking the $L^p$ norm  in equality~(\ref{eq:kh:lem}), Young's inequality 
 yields
$$
\Vert H_{\varepsilon}\ast f - f \Vert_p \leq \varepsilon \sum_{i=1}^d \Vert \partial_i  f\Vert_p \Vert L^{i}_{\varepsilon} \Vert_1 \ ,
$$
which gives the result by recalling \eqref{eq:L1Lij}.
\end{proof}
%%%%%%%%%%%%%%%%%%%%%%%%%%%%%%%%%%%%%%%%%%%%%%%%%%%%%%%%%%%%%%%%%%%%%%%%%%%%%%%%%%%%%%%%%%%%%%%%%%%%%%%%%%%%%%%%%%%

\subsection{About transition kernels}

In the following lemma, we state well-known technical properties about the transition probability function of a diffusion process. All the statements below are established  in \cite{friedmanEDS1}.
\begin{lem} 
\label{lem:transfun}
%Let us fix $(\Omega,\shf,(\shf)_{t \geq 0}, \P)$ be a filtered probability space.
 We assume here the validity of items 1. to 3. of Assumption \ref{ass:mainP3}. 
Consider a stochastic process $Y$, solution of the SDE 
\begin{eqnarray}
\label{eq:Ylem}
Y_t = Y_0 + \int_0^t \Phi(s,Y_s)dW_s + \int_0^t g(s,Y_s)ds.
\end{eqnarray}
$P(s,x_0,t,\Gamma)$ denotes its transition probability function,
 for all $(s,x_0,t,\Gamma) \in [0,T] \times \R^d \times [0,T] \times \shb(\R^d)$.
The following statements hold.
\begin{enumerate}
\item The transition probability function $P$ admits a density, i.e. there exists a Borel function $p : (s,x_0,t,x) \mapsto p(s,x_0,t,x)$ such that for all $(s,x_0,t) \in [0,T] \times \R^d \times [0,T]$ with $s < t$,
\begin{eqnarray}
P(s,x_0,t,\Gamma) = \int_{\Gamma} p(s,x_0,t,x)dx \quad , \quad \Gamma \in \shb(\R^d) \ .
\end{eqnarray}
\item The partial derivatives of the map $x_0 \mapsto p(s,x_0,t,x)$ exist in the distributional sense.
For almost all $0 \leq s < t \leq T$ and $x_0,x \in \R^d$ there are constants $C_u,c_u>0$, only depending on $\Phi, g$  
such that 
\begin{align}
\label{eq:1011a}
 p(s,x_0,t,x) \leq C_u
q(s,x_0,t,x)
\end{align}
and
\begin{align}
\label{eq:1011}
\left\vert \partial_{x_0} p(s,x_0,t,x) \right\vert \leq C_u
 \frac{1}{\sqrt{t-s}} q(s,x_0,t,x)\ ,
\end{align}
where
$q(s,x_0,t,x):=\left(\frac{c_u(t-s)}{\pi}\right)^{\frac{d}{2}} e^{-c_u \frac{\vert x-x_0 \vert^2}{t-s}}$ 
%The constant before the exponential term is there
is a Gaussian kernel.\\
%  VOIR SI LES DERIVEES PREMIERES SUFFISENT.
% VERIFIER DANS LA DEMONSTRATION OU LES $q$ APPARAISSENT.
% VERIFIER DANS LE TEXTE SI LA DEMONSTRATION CORRESPONDANTE EST COHERENTE
% \item There exist real constants $C_u, c_u > 0$ such that, for $0 \leq s < t \leq T$, 

% $(z,x)
%  \in \R^d \times \R^d$ and for all multi-index $m := (m_1,m_2)$ whose length $\vert m \vert := m_1 + m_2$ is less or equal to $2$, we have
% \begin{eqnarray}
% \label{eq:1011}
% \Big \vert \frac{\partial^{m_1}}{\partial z_i}\frac{\partial^{m_2}}{\partial x_j} p(s,z,t,x) \Big \vert \leq \frac{C_u}{(t-s)^{\frac{d + \vert m \vert}{2}}} e^{-c_u \frac{\vert x-z \vert^2}{t-s}}, \quad 0 \leq s < t \leq T, (z,x) \in (\R^d)^2 \ .
% \end{eqnarray}
 In particular
% there exists a constant $C > 0$ (only depending on $C_u$) such that
 for all $t \in [0,T]$ for almost all $r, x$ we have
% the law density $p_t$ of $Y_t$ satisfies
% \begin{eqnarray}
% \label{eq:majordens}
% \Vert p_t \Vert_{\infty} \leq C \Vert u_0 \Vert_{\infty} \ ,
% \end{eqnarray}
% where $p_t$ is given by $\int_{\R^d} p(0,x_0,t,x)u_0(x_0)dx_0$.
 \begin{equation}
 \label{eq:majordens}
  \sup_x \int p(r,x_0,t, x) d x_0 \leq C_u.
 \end{equation}
% where $p_t$ is given by $\int_{\R^d} p(0,x_0,t,x)u_0(x_0)dx_0$.

\end{enumerate}
\end{lem}
\begin{proof}
% NOTION DE ''FUNDAMENTAL SOLUTION'' IMPLICATIONS SUR LA DEMONSTRATION DU LEMME 2.2
%ET D'ABORD DE L'EQUATION ``RATHER IMPLICIT''\\
%Under Assumption \ref{ass:mainP3}, the results are a result of
% Theorems 4.6,  Theorems 4.7, Section 4 and
See Theorem 5.4, Section 5 in Chapter 6 in \cite{friedmanEDS1},
Section 4 of \cite{friedmanEDS1}, the fact that classical solutions of
\eqref{eq:FokkerPlanck} are also distributional solutions together with  
 Theorem 15, Section 9, chap. 1 in
\cite{friedmanEDP}
and inequalities (8.13) and (8.14) just before.
%% Il manquerait quelques connexions, genre le fait que solutions 
%% au bord "functions" s'etend au cas "mesures" comme la Dirac.
\end{proof}

\subsection{Proof of Proposition \ref{lem:MildWeak}}
\label{PL22}

For given $u$ we set 
$\hat \Lambda(s,x):= \Lambda(s, u(s,x), \nabla u(s,x)) u(s,x) $.
We first suppose that $u$ is a mild solution of \eqref{eq:PDE}. Taking 
into account that $P(s,x_0,t,\cdot)$ is a distributional solution of \eqref{eq:FKLin},  we show below
 that $u$ is indeed a weak solution of 
$\eqref{eq:PDE}$.

Indeed,
% for every probability measure $q_0$,   
for $ 0 \le r < t \le T$,
 \eqref{eq:rather_implicit},  gives
\begin{align}
\int\limits_{\mathbb{R}^d}  
\int\limits_{\mathbb{R}^d} \varphi(x)  P(r,x_0,t, dx)     {\bf u_0}(dx_0)  = 
&\int\limits_r^t 
%\int\limits_{\mathbb{R}^d}  L_s \varphi(x)
 \left( \int\limits_{\mathbb{R}^d} \int\limits_{\mathbb{R}^d} 
 L_s \varphi(x)  
P(r,x_0,s,dx)
 {\bf u_0}(dx_0) 
  \right)
   d s  \notag \\
&
+\int\limits_{\mathbb{R}^d}  
 \varphi(x_0) {\bf u_0}(dx_0),  
%+ \int\limits_{\mathbb{R}^d} \varphi(x) f(x) d x  
\qquad  \forall \varphi \in C_0^{\infty}. \label{eq:weak_sol_operator_equivelence_proof}
\end{align}
%By \eqref{eq:DefMildSol},
For every $s \in [0,T]$ we define the measure
% x \in \R^d$,
\begin{align}
 v(s,dx)&=\int\limits_{\mathbb{R}^d}  P(0,x_0,s,dx) {\bf u_0}(d x_0)+ \int\limits_0^s \int\limits_{\mathbb{R}^d} \hat{\Lambda}(r,x_0) P(r,x_0,s,dx)  d x_0 d r. \label{eq:varation_mild_sol_Lemma}
\end{align}
Since $u$ is a mild solution of \eqref{eq:PDE}, see \eqref{eq:DefMildSol},
we have 
$$ u(s,x)dx = v(s,dx).$$
In particular $v(s,\cdot)$ admits $u(s,\cdot)$  as density.
We need to show that for all $\varphi \in C_0^{\infty}$,
$t \in [0,T]$ we have
\eqref{eq:DefSolPDE}, i.e.
%Using Fubini's theorem and the definition of the fundamental solution, it is easy to check that $ v$ satisfies 
for any test function $\varphi \in C_0^{\infty}(\mathbb{R}^d)$ and any $t\in [0,T]$
\begin{align}
\int\limits_0^t \int\limits_{\mathbb{R}^d} v(s,dx) L_s \varphi(x) dx ds= 
&\int\limits_{\mathbb{R}^d} \varphi(x){v}(t,dx) -\int\limits_{\mathbb{R}^d} \varphi(x)
{\bf u_0} (dx) \notag \\ 
&- \int\limits_0^t \int\limits_{\mathbb{R}^d} \varphi(x)  \hat{\Lambda} (s,x) dx
ds. \label{eq:weak_sol_rearranged}
\end{align} 
%\begin{proof}[Proof \nopunct.]
%(of Lemma \ref{Lemma:equiv_mild_weak_sol}).\\ \\
%We first recall that a $v$ is weak solution of \eqref{eq:variation_mild_sol_generalized_pde} if and only if for any test function $\varphi \in C_0^{\infty}$ we have
%\begin{align}
%\frac{1}{2}\int\limits_0^t \int\limits_{\mathbb{R}^d} v(s,x)\partial^2 \varphi(x) d x \dif s= &\int\limits_{\mathbb{R}^d} \varphi(x)v(t,x)\dif x - \int\limits_{\mathbb{R}^d} \varphi(x)v(0,x)\dif x \notag \\ 
%& - \int\limits_0^t \int\limits_{\mathbb{R}^d} \varphi(x)  \hat{\Lambda} (s,x) \dif x\dif s. \label{eq:weak_sol_rearranged}
%\end{align}
Starting with the left-hand side of \eqref{eq:weak_sol_rearranged}, we start by plugging in the expression of $v$ in \eqref{eq:varation_mild_sol_Lemma} into
the right-hand side of \eqref{eq:weak_sol_rearranged}:
\begin{align}
\int\limits_0^t \int\limits_{\mathbb{R}^d}  L_s \varphi(x)  v(s,dx) d s 
= &\int\limits_0^t \int\limits_{\mathbb{R}^d} {\bf u_0}(dx_0)  \int\limits_{\mathbb{R}^d}  L_s \varphi(x) P(0,x_0,s,dx)  d s \notag \\  &+  
   \int\limits_0^t \left(\int\limits_{\mathbb{R}^d}
\int\limits_0^s 
 \hat{\Lambda}(r,x_0) 
\int\limits_{\mathbb{R}^d} L_s \varphi(x)   P(r,x_0,s,dx) dr d x_0
   \right) d s.   \label{eq:somewhere_in_the_proof_of_equivalence}
\end{align}
For the first term on the right-hand side of \eqref{eq:somewhere_in_the_proof_of_equivalence}, we can directly use identity \eqref{eq:weak_sol_operator_equivelence_proof} to infer
\begin{align}
&\int\limits_0^t \left(\int \limits_{\mathbb{R}^d}
 \left(\int\limits_{\mathbb{R}^d} 
 L_s \varphi(x) P(0,x_0,s,dx) \right)
 {\bf u_0}  (dx_0)  \right) d s \notag \\
=&\int\limits_{\mathbb{R}^d} {\bf u_0}(dx_0) 
 \left(\int\limits_{\mathbb{R}^d} \varphi (x) 
P(0,x_0,t,dx) \right)
  - \int\limits_{\mathbb{R}^d} \varphi(x_0) 
{\bf u_0}(dx_0). \label{E624}
\end{align}
For the second term on the right-hand side,
% of \eqref{eq:somewhere_in_the_proof_of_equivalence},
% we would also like to apply identity \eqref{eq:weak_sol_operator_equivelence_proof}. This, however, cannot be done directly because of the integral with respect to $z$ which does not appear in \eqref{eq:somewhere_in_the_proof_of_equivalence}. But
 a simple application of Fubini's Theorem
for random kernels
 enables us to ``pull out'' the integral with respect to $r$ and then apply \eqref{eq:weak_sol_operator_equivelence_proof}:
\begin{align}
&\int\limits_0^t \left(\int\limits_{\mathbb{R}^d}  \int\limits_0^s   \hat{\Lambda}(r,x_0) 
\int\limits_{\mathbb{R}^d}     L_s \varphi(x)          P(r,x_0,s,dx) d x_0 d r \right)
                d s  \notag \\
=&\int\limits_0^t \int\limits_{\mathbb{R}^d} \hat{\Lambda}(r,x_0)  
\left(\int\limits_r^t  \int\limits_{\mathbb{R}^d} P(r,x_0,s,dx) L_s \varphi(x)   d s \right) d x_0 d r
\notag  \\
=& \int\limits_0^t  \int\limits_{\mathbb{R}^d}  \hat{\Lambda}(r,x_0) 
\left( \int_{\mathbb{R}^d}\varphi(x)  P(r,x_0,t,dx)  -  \varphi(x_0)\right)   d x_0   d r. \label{E625}
\end{align}
Plugging the equalities \eqref{E624}, \eqref{E625}, into \eqref{eq:somewhere_in_the_proof_of_equivalence} leaves us with 
\begin{align*}
&\int\limits_0^t \int\limits_{\mathbb{R}^d} v(r,d x) L_s \varphi(x) d r\\
= &\int\limits_{\mathbb{R}^d}  \int\limits_{\mathbb{R}^d} \varphi (x)
P(0,x_0,t,dx) {\bf u_0}(d x_0) 
 - \int\limits_{\mathbb{R}^d}\varphi(x){\bf u_0}(dx) \\
&+\int\limits_0^t \int\limits_{\mathbb{R}^d} 
\hat{\Lambda}(r,x_0) 
\left( \int_{\mathbb{R}^d} \varphi(x) P(s,x_0,r,dx) \right) d x_0 d r
 - \int\limits_{\mathbb{R}^d} \varphi(x) \hat{\Lambda}(r,x) d x    d r \\
=&\int\limits_{\mathbb{R}^d} \varphi(x)v(t,dx) - \int\limits_{\mathbb{R}^d} 
\varphi(x) {\bf u_0}(dx) -\int\limits_0^t \int\limits_{\mathbb{R}^d} \varphi(x) 
 \hat{\Lambda} (s,x) d x d s,
\end{align*}
where for the latter equality we again used the definition of $v$ in \eqref{eq:varation_mild_sol_Lemma}. This is exactly \eqref{eq:weak_sol_rearranged} and thus completes the first part of the 
proof. %verifies that $v$ indeed is a weak solution to \eqref{eq:varation_mild_sol_Lemma}.
%\end{proof}
\\

Conversely, suppose that $u$ is a weak solution of \eqref{eq:PDE}, in the sense of Definition \ref{def:SolPDE}. We also consider
\begin{eqnarray}
\label{eq:WealMildSol]}
\bar{v}(t,dx) & := & \int_{\R} P(0,x_0,t,dx){\bf u_0}(dx_0) + \int_0^t ds \int_{\R^d} P(s,x_0,t,dx)u(s,dx_0) \hat \Lambda(s,x_0)
%\Lambda(s,x_0,u(s,x_0),\nabla u(s,x_0)) dx_0
 \nonumber \\
& = & \int_{\R} P(0,x_0,t,dx){\bf u_0}(dx_0) + \int_0^t ds \int_{\R^d} P(s,x_0,t,dx) \hat{\Lambda}(u)(s,x_0) dx_0.
\end{eqnarray}
%where $\bar{\Lambda}(u)(s,x) := u(s,x)\Lambda(s,x,u(s,x), \nabla u(s,x))$ for $(s,x) \in [0,T] \times \R^d$.
 We want to ensure that $ u = \bar{v}$. 
On the one hand, by the first part of the proof applied to $\bar v$
instead of $v$ we can show that 
%integrating the function $\bar{v}(t,\cdot)$ against a test function and using again that $P(s,x_0,t,\cdot)$ is a distributional solution of \eqref{eq:FKLin}, we obtain that $\bar{v}$  is a weak solution of 
%satisfies (in the sense of distributions)
\begin{equation}
\label{eq:SemiLinPDE}
\left \{ 
\begin{array}{l}
\partial_t \bar{v} = L^{\ast}_t \bar{v} + \hat{\Lambda} \\
\bar{v}(0,\cdot) = {\bf u_0}.
\end{array}
\right .
\end{equation}
On the other hand, $u$ being a weak solution of \eqref{eq:PDE}, 
it also a solution of \eqref{eq:SemiLinPDE} in the sense of distributions. 
%satisfies \eqref{eq:SemiLinPDE} (in the sense of distributions).
 We set $w := \bar{v} - u$. It follows that $w$ and the zero measure 
function $\bar w \equiv 0$
$\hat{v} := 0$ 
%%% OLD
both satisfy \eqref{eq:FokkerPlanck} in the sense of distributions,
see \eqref{eq:rather_implicit}. Uniqueness of the solution of \eqref{eq:FokkerPlanck} implies that $w =0$, which concludes the proof.
%%% FRANCESCO.... ATTENTION A LA CLASSE D'UNICITE DESPACE.
%MESURE OU $W^{1,1}$

\subsection{Proof of technicalities of Section \ref{SFK}}
We give in this section the proof of Lemma \ref{lem:RecolSol}.
\label{SProofLem:RecolSol}
\begin{proof}[Proof of Lemma \ref{lem:RecolSol}]
We only prove the direct implication since the converse follows
easier with similar arguments.
% Without restriction of generality, we can assume that $ T = N \tau$ for some integer $N \in \N$.
 The aim is to prove, for all $n \in \{1,\cdots,N\}$,
\begin{equation}
\label{eq:Induc}
(H_n) \; \left \{
\begin{array}{l}
\mu(t,dx) = \int_{\R^d} P(0,x_0,t,dx) {\bf u_0}(dx_0) + \int_0^t ds \int_{\R^d} P(s,x_0,t,dx)\tilde{\Lambda}(s,x_0) \mu(s,dx_0)  \ , \\
\textrm{for all } t \in [0,n\tau] \ .
\end{array}
\right .
\end{equation}
We are going to proceed by induction on $n$. For $n = 1$, formula \eqref{eq:Induc} follows from \eqref{eq:HypInduc} by taking $k = 0$. We suppose now that $(H_{n-1})$ holds for some integer $n \ge 1$. Then, by taking $t = (n-1)\tau$ in the first line equation of \eqref{eq:Induc}, it follows immediately that
\begin{eqnarray}
\label{eq:437}
\mu((n-1)\tau,dx_0) = \int_{\R^d} P(0,\widetilde{x_0},(n-1)\tau,dx_0) {\bf u_0}(d\widetilde{x_0}) + \int_0^{(n-1)\tau} ds \int_{\R^d} P(s,\widetilde{x_0},(n-1)\tau,dx_0)\tilde{\Lambda}(s,\widetilde{x_0}) \mu(s,d\widetilde{x_0})  \ . \nonumber \\
\end{eqnarray}
On the other hand, since \eqref{eq:HypInduc} is valid for all $t \in [(n-1)\tau, n\tau]$ by plugging $k = n-1$, we obtain
\begin{eqnarray}
\label{eq:438}
\mu(t,dx) = \int_{\R^d} P((n-1)\tau,x_0,t,dx) \mu((n-1)\tau,dx_0) + \int_{(n-1)\tau}^t ds \int_{\R^d} P(s,x_0,t,dx)\tilde{\Lambda}(s,x_0) \mu(s,dx_0)  \ ,\nonumber \\
\end{eqnarray}
for all $t \in [(n-1)\tau,n\tau]$. Inserting \eqref{eq:437} in \eqref{eq:438} yields
\begin{eqnarray}
\mu(t,dx) & = & \int_{\R^d} {\bf u_0}(d\widetilde{x_0}) \int_{\R^d} P(0,\widetilde{x_0},(n-1)\tau,dx_0) P((n-1)\tau,x_0,t,dx) \nonumber \\
&& + \int_0^{(n-1)\tau} ds \int_{\R^d} \mu(s,d\widetilde{x_0}) \tilde{\Lambda}(s,\widetilde{x_0}) \int_{\R^d} P(s,\widetilde{x_0},(n-1)\tau,dx_0) P((n-1)\tau,x_0,t,dx) \nonumber \\
&& + \int_{(n-1)\tau}^t ds \int_{\R^d} P(s,x_0,t,dx)\tilde{\Lambda}(s,x_0) \mu(s,dx_0)  \ , \: t \in [(n-1)\tau,n\tau] \ .
\end{eqnarray}
Invoking the Chapman-Kolmogorov equation satisfied by the transition probability function $P(s,x_0,t,dx)$ (see e.g. expression (2.1) in Section 2.2, Chapter 2 in \cite{stroock}), we have
\begin{eqnarray}
\label{eq:ChapKol}
P(s,\widetilde{x_0},t,dx) = \int_{\R^d} P(s,\widetilde{x_0},\theta,dz)P(\theta,z,t,dx), \quad s < \theta < t, \, (\widetilde{x_0},z) \in \R^d \times \R^d .
\end{eqnarray}
Applying \eqref{eq:ChapKol} with $\theta = (n-1) \tau$, it follows that for all $t \in [0,n\tau]$,
\begin{eqnarray}
\mu(t,dx) & = & \int_{\R^d} {\bf u_0}(d\widetilde{x_0}) P(0,\widetilde{x_0},t,dx) \nonumber \\
&& + \int_0^t ds \int_{\R^d} P(s,\widetilde{x_0},t,dx)\tilde{\Lambda}(s,\widetilde{x_0}) \mu(s,d\widetilde{x_0})  \ .
\end{eqnarray} 
This shows that $(H_n)$ holds. \\
%The converse is straightforwardly obtained through very similar computations as the ones already done above. This ends the proof of the lemma.
\end{proof}

 %Bien verifier si $u_0$ est une mesure, dans ce cas il faudrait l'ecrire
%%%  ${\bf u_0}$.

\noindent {\bf ACKNOWLEDGMENTS.} The authors are grateful to the Editors and to
the two Referees who have read extremely carefully the paper
stimulating with useful observations its improvement.

%Part of this work has been done during a stay of the third named author in the University of Bielefeld, SFB 701. He is grateful to Prof. Michael R\"ockner for
%the kind invitation.

%\newpage
% % %\addcontentsline{toc}{section}{Bibliography}
\bibliographystyle{plain}
\bibliography{NonConservativePDE}

\def\cprime{$'$}
\begin{thebibliography}{10}

\bibitem{aronsonb}
D.~G. Aronson and H.~F. Weinberger.
\newblock Multidimensional nonlinear diffusion arising in population genetics.
\newblock {\em Adv. Math.}, 30:33--76, 1978.

\bibitem{BRR2}
V.~Barbu, M.~R\"ockner, and F.~Russo.
\newblock Probabilistic representation for solutions of an irregular porous
  media type equation: the irregular degenerate case.
\newblock {\em Probab. Theory Related Fields}, 151(1-2):1--43, 2011.

\bibitem{BCR3}
N.~Belaribi, F.~Cuvelier, and F.~Russo.
\newblock Probabilistic and deterministic algorithms for space multidimensional
  irregular porous media equation.
\newblock {\em SPDEs: Analysis and Computations}, 1(1):3--62, 2013.

\bibitem{Ben_Vallois}
S.~Benachour, P.~Chassaing, B.~Roynette, and P.~Vallois.
\newblock Processus associ\'es \`a\ l'\'equation des milieux poreux.
\newblock {\em Ann. Scuola Norm. Sup. Pisa Cl. Sci. (4)}, 23(4):793--832, 1996.

\bibitem{BertShre}
D.~P. Bertsekas and S.~E. Shreve.
\newblock {\em Stochastic optimal control}, volume 139 of {\em Mathematics in
  Science and Engineering}.
\newblock Academic Press, Inc. [Harcourt Brace Jovanovich, Publishers], New
  York-London, 1978.
\newblock The discrete time case.

\bibitem{BRR1}
P.~Blanchard, M.~R{\"o}ckner, and F.~Russo.
\newblock Probabilistic representation for solutions of an irregular porous
  media type equation.
\newblock {\em Ann. Probab.}, 38(5):1870--1900, 2010.

\bibitem{bogachevkrylovBook}
V.~I. Bogachev, N.~V. Krylov, M.~R{\"o}ckner, and S.~V. Shaposhnikov.
\newblock {\em Fokker-{P}lanck-{K}olmogorov equations}, volume 207 of {\em
  Mathematical Surveys and Monographs}.
\newblock American Mathematical Society, Providence, RI, 2015.

\bibitem{piperno}
M.~Bossy, L.~Fezoui, and S.~Piperno.
\newblock Comparison of a stochastic particle method and a finite volume
  deterministic method applied to {B}urgers equation.
\newblock {\em Monte Carlo Methods Appl.}, 3(2):113--140, 1997.

\bibitem{bossytalay1}
M.~Bossy and D.~Talay.
\newblock Convergence rate for the approximation of the limit law of weakly
  interacting particles: application to the {B}urgers equation.
\newblock {\em Ann. Appl. Probab.}, 6(3):818--861, 1996.

\bibitem{cheridito}
P.~Cheridito, H.~M. Soner, N.~Touzi, and N.~Victoir.
\newblock Second-order backward stochastic differential equations and fully
  nonlinear parabolic {PDE}s.
\newblock {\em Comm. Pure Appl. Math.}, 60(7):1081--1110, 2007.

\bibitem{friedmanEDP}
A.~Friedman.
\newblock {\em Partial differential equations of parabolic type}.
\newblock Prentice-Hall, Inc., Englewood Cliffs, N.J., 1964.

\bibitem{friedmanEDS1}
A.~Friedman.
\newblock {\em Stochastic differential equations and applications. {V}ol. 1}.
\newblock Academic Press [Harcourt Brace Jovanovich, Publishers], New
  York-London, 1975.
\newblock Probability and Mathematical Statistics, Vol. 28.

\bibitem{labordere}
P.~Henry-Labord\`ere.
\newblock Counterparty risk valuation: {A} marked branching diffusion approach.
\newblock {\em Available at SSRN: http://ssrn.com/abstract=1995503 or
  http://dx.doi.org/10.2139/ssrn.1995503}, 2012.

\bibitem{HOTTW}
P.~Henry-Labord\`ere, N.~Oudjane, X.~Tan, N.~Touzi, and X.~Warin.
\newblock Branching diffusion representation of semilinear pdes and {M}onte
  {C}arlo approximations.
\newblock {\em Available at http://arxiv.org/pdf/1603.01727v1.pdf}, 2016.

\bibitem{LabordereTouziTan}
P.~Henry-Labord{\`e}re, X.~Tan, and N.~Touzi.
\newblock A numerical algorithm for a class of {BSDE}s via the branching
  process.
\newblock {\em Stochastic Process. Appl.}, 124(2):1112--1140, 2014.

\bibitem{holm92}
L.~Holmstr{\"o}m and J.~Klemel{\"a}.
\newblock Asymptotic bounds for the expected {$L^1$} error of a multivariate
  kernel density estimator.
\newblock {\em J. Multivariate Anal.}, 42(2):245--266, 1992.

\bibitem{JourMeleard}
B.~Jourdain and S.~M{\'e}l{\'e}ard.
\newblock Propagation of chaos and fluctuations for a moderate model with
  smooth initial data.
\newblock {\em Ann. Inst. H. Poincar\'e Probab. Statist.}, 34(6):727--766,
  1998.

\bibitem{kac}
M.~Kac.
\newblock {\em Probability and related topics in physical sciences}, volume
  1957 of {\em With special lectures by G. E. Uhlenbeck, A. R. Hibbs, and B.
  van der Pol. Lectures in Applied Mathematics. Proceedings of the Summer
  Seminar, Boulder, Colo.}
\newblock Interscience Publishers, London-New York, 1959.

\bibitem{keener}
J.~P. Keener and J.~Sneyd.
\newblock {\em Mathematical Physiology II: Systems Physiology}.
\newblock Springer, New York, 2008.

\bibitem{LOR2}
A.~Le~Cavil, N.~Oudjane, and F.~Russo.
\newblock Particle system algorithm and chaos propagation related to a
  non-conservative {M}c{K}ean type stochastic differential equations.
\newblock {\em Stochastics and Partial Differential Equations: Analysis and
  Computation}, pages 1--37, 2016.

\bibitem{LOR1}
A.~Le~Cavil, N.~Oudjane, and F.~Russo.
\newblock Probabilistic representation of a class of non-conservative nonlinear
  partial differential equations.
\newblock {\em ALEA Lat. Am. J. Probab. Math. Stat.}, 13(2):1189--1233, 2016.

\bibitem{LOR4}
A.~Le~Cavil, N.~Oudjane, and F.~Russo.
\newblock Monte-{C}arlo algorithms for a forward {F}eynman--{K}ac-type
  representation for semilinear nonconservative partial differential equations.
\newblock {\em Monte Carlo Methods Appl.}, 24(1):55--70, 2018.

\bibitem{LieberOR}
J.~Lieber, N.~Oudjane, and F.~Russo.
\newblock On the well-posedness of a class {M}c{K}ean {F}eynman-{K}ac
  equations.
\newblock {\em In preparation.}

\bibitem{Mckean}
H.~P.~Jr. McKean.
\newblock Propagation of chaos for a class of non-linear parabolic equations.
\newblock In {\em Stochastic {D}ifferential {E}quations ({L}ecture {S}eries in
  {D}ifferential {E}quations, {S}ession 7, {C}atholic {U}niv., 1967)}, pages
  41--57. Air Force Office Sci. Res., Arlington, Va., 1967.

\bibitem{MeleardVlasov}
S.~M{\'e}l{\'e}ard.
\newblock Asymptotic behaviour of some interacting particle systems;
  {M}c{K}ean-{V}lasov and {B}oltzmann models.
\newblock In {\em Probabilistic models for nonlinear partial differential
  equations ({M}ontecatini {T}erme, 1995)}, volume 1627 of {\em Lecture Notes
  in Math.}, pages 42--95. Springer, Berlin, 1996.

\bibitem{MeleaRoel}
S.~M{\'e}l{\'e}ard and S.~Roelly-Coppoletta.
\newblock A propagation of chaos result for a system of particles with moderate
  interaction.
\newblock {\em Stochastic Process. Appl.}, 26(2):317--332, 1987.

\bibitem{murray}
J.~D. Murray.
\newblock {\em Mathematical biology. {I}}, volume~17 of {\em Interdisciplinary
  Applied Mathematics}.
\newblock Springer-Verlag, New York, third edition, 2002.
\newblock An introduction.

\bibitem{pardouxgeilo}
E.~Pardoux.
\newblock Backward stochastic differential equations and viscosity solutions of
  systems of semilinear parabolic and elliptic {PDE}s of second order.
\newblock In {\em Stochastic analysis and related topics, {VI} ({G}eilo,
  1996)}, volume~42 of {\em Progr. Probab.}, pages 79--127. Birkh\"auser
  Boston, Boston, MA, 1998.

\bibitem{pardoux}
{\'E}.~Pardoux and S.~G. Peng.
\newblock Adapted solution of a backward stochastic differential equation.
\newblock {\em Systems Control Lett.}, 14(1):55--61, 1990.

\bibitem{rascanu}
E.~Pardoux and A.~Ra{\c{s}}canu.
\newblock {\em Stochastic differential equations, Backward SDEs, Partial
  differential equations}, volume~69.
\newblock Springer, 2014.

\bibitem{stroock}
D.~W. Stroock and S.~R.~S. Varadhan.
\newblock {\em Multidimensional diffusion processes}.
\newblock Classics in Mathematics. Springer-Verlag, Berlin, 2006.
\newblock Reprint of the 1997 edition.

\bibitem{sznitman}
A-S. Sznitman.
\newblock Topics in propagation of chaos.
\newblock In {\em \'{E}cole d'\'{E}t\'e de {P}robabilit\'es de {S}aint-{F}lour
  {XIX}---1989}, volume 1464 of {\em Lecture Notes in Math.}, pages 165--251.
  Springer, Berlin, 1991.

\bibitem{wang}
X.~Y. Wang, Z.~S. Zhu, and Y.~K. Lu.
\newblock Solitary wave solutions of the generalized burgers-huxley equation.
\newblock {\em J. Phys. A: Math. Gen.}, 23:271--274, 1990.

\end{thebibliography}
%\bibliography{../../../NonConservativePDE_bib/NonConservativePDE}
%\bibliography{NonConservativePDE_bib/NonConservativePDE}
\end{document}